\documentclass{amsart}

\usepackage{amsmath, amssymb, amsthm, amsfonts, mathrsfs, mathabx}

\usepackage{graphicx}

\usepackage[cmtip,all]{xy}
\usepackage{xcolor}
\usepackage{tikz}
\usepackage{tikz-cd}

\newtheorem{theorem}{Theorem}[section]
\newtheorem{lemma}[theorem]{Lemma}
\newtheorem{corollary}[theorem]{Corollary}

\theoremstyle{definition}
\newtheorem{definition}[theorem]{Definition}

\theoremstyle{remark}

\numberwithin{equation}{section}

\newcommand{\dom}{\mbox{\rm dom}}
\newcommand{\range}{\mbox{\rm range}}
\newcommand{\diam}{\mbox{\rm diam}}
\newcommand{\rest}{\!\!\upharpoonright\!}
\newcommand{\Aut}{\mbox{\rm Aut}}
\newcommand{\Part}{\mbox{\rm Part}}
\newcommand{\actson}{\curvearrowright}
\newcommand{\Kfin}{\mathcal{K}^{\mathcal{L}}_{\mbox{\rm\scriptsize fin}}}
\newcommand{\Ksep}{\mathcal{K}^{\mathcal{L}}_{\mbox{\rm\scriptsize sep}}}

\begin{document}

\title[AP and Urysohn Structures in Continuous Logic]{The Amalgamation Property and Urysohn Structures in Continuous Logic}


\author{Su Gao}
\address{School of Mathematical Sciences and LPMC, Nankai University, Tianjin 300071, P.R. China}
\email{sgao@nankai.edu.cn}
\thanks{The first author acknowledges the partial support of his resarch by the National Natural Science Foundation of
China (NSFC) grants 12250710128 and 12271263.
}

\author{Xuanzhi Ren}
\address{School of Mathematical Sciences and LPMC, Nankai University, Tianjin 300071, P.R. China}
\email{renxz@mail.nankai.edu.cn}
\thanks{Research of the second author was partially supported by the Innovation and Entrepreneurship Training Program for
College Students of Tianjin, no. 202210055334.
}


\date{\today}


\begin{abstract} In this paper we consider the classes of all continuous $\mathcal{L}$-(pre-)structures for a continuous first-order signature $\mathcal{L}$. We characterize the moduli of continuity for which the classes of finite, countable, or all continuous $\mathcal{L}$-(pre-)structures have the amalgamation property. We also characterize when Urysohn continuous $\mathcal{L}$-(pre)-structures exist, establish that certain classes of finite continuous $\mathcal{L}$-structures are countable Fra\"iss\'e classes, prove the coherent EPPA for these classes of finite continuous $\mathcal{L}$-structures, and show that actions by automorphisms on finite $\mathcal{L}$-structures also form a Fra\"iss\'e class. As consequences, we have that the automorphism group of the Urysohn continuous $\mathcal{L}$-structure is a universal Polish group and that Hall's universal locally finite group is contained in the automorphism group of the Urysohn continuous $\mathcal{L}$-structure as a dense subgroup.
\end{abstract}

\maketitle



\section{Introduction}
Continuous first-order logic was developed by Ben Yaccov and Usvyatsov in \cite{BU2008} as a variant of the continuous logic studied by Chang and Keisler \cite{CK}. This logic turns out to be very useful in the
study of metric structures. For instance, Ben Yaacov \cite{BY2014} proved that the linear isometry
group of the Gurarij space is universal among all Polish groups by viewing Banach spaces
as continuous first-order structures. Another example is the metric Scott analysis developed
in \cite{BDNT} where the infinitary continuous first-order logic is used.

This paper is a contribution to the continuous model theory as a theory of metric structures. Previous developments of the continuous model theory have been done mostly as a generalization of the classical model theory. In this paper, we take the slightly different point of view of regarding the continuous model theory as a generalization of the rich theory of separable metric spaces.

It is well known that there is a unique complete separable metric space that is both universal (i.e. containing a copy of every separable metric space as a subspace) and ultrahomogeneous (\cite{Urysohn}). This space has been known as the {\em universal Urysohn metric space} and has been denoted as $\mathbb{U}$ to emphasize its canonical nature. Built on results of Kat\v{e}tov (\cite{Katetov}), Uspenskij \cite{Uspenskij} showed that the isometry group of $\mathbb{U}$ is a universal Polish group. Thus one naturally wonders whether for any continuous first-order signature $\mathcal{L}$ there exists a Urysohn continuous $\mathcal{L}$-structure, i.e., a complete separable continuous $\mathcal{L}$-structure which is both universal and ultrahomogeneous. Moreover, if it exists, whether its automorphism group is a universal Polish group.

It turns out that the answer depends on the continuous first-order signature $\mathcal{L}$. In general, a continuous first-order signature $\mathcal{L}$ consists of relation symbols and function symbols of various arities. In this paper, we consider continuous first-order signatures $\mathcal{L}$ with only finitely many relation symbols. Just as in a classical first-order signature we do not formally include the equality relation $=$ but consider it to be a part of the logic, in continuous first-order logic we also tacitly include a symbol $d$ which is always interpreted as a metric on a continuous structure. For each $n$-ary relation symbol $R\in\mathcal{L}$ and $1\leq i\leq n$, we also associate a modulus of continuity $\mathfrak{u}_{R,i}$. In an interpretation of $R$ in a model (which is a function from the domain of the model to the interval $[0,1]$), $R$ as a function of the $i$-th coordinate is required to be a uniformly continuous function with $\mathfrak{u}_{R,i}$ as a modulus of continuity.

We will define two notions which are properties of the moduli of continuity associated to a continuous signature $\mathcal{L}$. They are called {\em proper} and {\em semiproper}, the difference being wether the moduli of continuity for unary relation symbols in $\mathcal{L}$ (if any) are upper semicontinuous. We will prove the following theorem.

\begin{theorem}\label{thm:U} Let $\mathcal{L}$ be a continuous signature with only finitely many relation symbols, where all the associated moduli of continuity are nondecreasing. Then the following hold:
\begin{enumerate}
\item[(i)] There exists a (separable) Urysohn $\mathcal{L}$-structure iff $\mathcal{L}$ is proper.
\item[(ii)] There exists a Urysohn $\mathcal{L}$-pre-structure iff $\mathcal{L}$ is semiproper.
\end{enumerate}
\end{theorem}

Here a continuous $\mathcal{L}$-pre-structure is one in which the symbol $d$ is interpreted as a metric; in contrast, in a continuous $\mathcal{L}$-structure $d$ is interpreted as a complete metric.

It turns out that properness and semiproperness also characeterize the amalgamation property for various classes of continuous $\mathcal{L}$-structures. In particular, we have the following theorems.

\begin{theorem}\label{thm:pre} Let $\mathcal{L}$ be a continuous signature with only finitely many relation symbols, where all the associated moduli of continuity are nondecreasing. Then the following are equivalent:
\begin{enumerate}
\item[(i)] $\mathcal{L}$ is semiproper.
\item[(ii)] The class of all finite continuous $\mathcal{L}$-pre-structures has the amalgamation property.
\item[(iii)] The class of all finite continuous $\mathcal{L}$-pre-structures has the strong amalgamation property.
\end{enumerate}
\end{theorem}

\begin{theorem}\label{thm:1} Let $\mathcal{L}$ be a continuous signature with only finitely many relation symbols, where all the associated moduli of continuity are nondecreasing. Then the following are equivalent:
\begin{enumerate}
\item[(i)] $\mathcal{L}$ is proper.
\item[(ii)] The class of all (countable) continuous $\mathcal{L}$-pre-structures has the amalgamation property.
\item[(iii)] The class of all (countable) continuous $\mathcal{L}$-pre-structures has the strong amalgamation property.
\end{enumerate}
\end{theorem}

We also define an intermediate notion of strong semiproperness between semiproperness and properness, and characterize the amalgamation property of the class of continuous $\mathcal{L}$-structures.

\begin{theorem}\label{thm:1.5} Let $\mathcal{L}$ be a continuous signature with only finitely many relation symbols, where all the associated moduli of continuity are nondecreasing. Then the following are equivalent:
\begin{enumerate}
\item[(i)] $\mathcal{L}$ is strongly semiproper.
\item[(ii)] The class of all (countable) continuous $\mathcal{L}$-structures has the amalgamation property.
\item[(iii)] The class of all (countable) continuous $\mathcal{L}$-structures has the strong amalgamation property.
\end{enumerate}
\end{theorem}

The separable Urysohn continuous $\mathcal{L}$-structure will enjoy similar canonicity as its classical counterpart, and thus we denote it by $\mathbb{U}_\mathcal{L}$. Because $\mathbb{U}_\mathcal{L}$ is constructed with a Kat\v{e}tov-style construction, it will follow from Uspenskij's argument that the automorphism group of $\mathbb{U}_{\mathcal{L}}$ is a universal Polish group.

We will also consider certain classes of finite continuous $\mathcal{L}$-structures for semiproper $\mathcal{L}$ and show that they are countable Fra\"iss\'e classes. In particular, we obtain a {\em rational Urysohn continuous $\mathcal{L}$-pre-structure} $\mathbb{QU}_{\mathcal{L}}$ in analogy with its classical counterpart $\mathbb{QU}$.

We will show that for any semiproper $\mathcal{L}$ the class of all finite continuous $\mathcal{L}$-structures have the coherent EPPA. EPPA stands for the Extension Property for Partial Automorphisms and is a term coined by Hrushovski when he first showed the EPPA for graphs (\cite{Hru}). Later on, the EPPA has been proven for a great number of classes of structures, most notably for classical finite relational structures (Herwig--Lascar \cite{HL}) and for finite metric spaces (Solecki \cite{Solecki}). For a survey of recent results on EPPA see \cite{HKN}. Siniora--Solecki \cite{Siniora} defined the coherent EPPA and proved it for a large number of classes of classical structures. Here we prove the following theorem.

\begin{theorem}\label{thm:2} Let $\mathcal{L}$ be any semiproper continuous signature. The class of all finite continuous $\mathcal{L}$-structures has the coherent EPPA.
\end{theorem}

It turns out that the (coherent) EPPA for the class of all finite continuous $\mathcal{L}$-structures is essentially equivalent to the semiproperness of $\mathcal{L}$.

The method of Herwig--Lascar made deep connections between properties of free groups and extensions of partial automorphisms or group actions. These group-theoretic properties were studied further by other authors, including Coulbois \cite{Coulbois}, Rosendal \cite{RosendalI} \cite{Rosendal}, and Etedadialiabadi--Gao \cite{EG0} \cite{EG}. These investigations culminated in \cite{EGLMM} to show that the actions by automorphisms of finite groups on finite classical relational structures form a Fra\"iss\'e class. Here we prove an analogous result for continuous structures as follows.

\begin{theorem}\label{thm:3} Let $\mathcal{L}$ be any semiproper continuous signature. The class of all actions by automorphisms of finite groups on finite continuous $\mathcal{L}$-structures form a Fra\"iss\'e class. Consequently, if $\mathcal{L}$ is proper, the automorphism group of $\mathbb{U}_{\mathcal{L}}$ contains Hall's universal locally finite group $\mathbb{H}$ as a dense subgroup.
\end{theorem}

The rest of the paper is organized as follows. In Section 2 we recall some basic definitions. In Section 3 we define properness, semiproperness, and strong semiproperness for continuous first-order signatures and prove Theorems~\ref{thm:pre},  \ref{thm:1} and \ref{thm:1.5}. In Section 4 we give a Kat\v{e}tov-style construction of the Urysohn continuous structure for a proper continuous first-order signature. In particular, we prove Theorem~\ref{thm:U}. In Section 5 we consider certain classes of finite continuous structures and show that they form countable Fra\"iss\'e classes. We also show that the completion of $\mathbb{QU}_{\mathcal{L}}$ is isomorphic to $\mathbb{U}_{\mathcal{L}}$ when $\mathcal{L}$ is proper. In Section 6 we prove Theorem~\ref{thm:2}. In Section 7 we prove the first half of Theorem~\ref{thm:3}. For this the main effort is to prove an analog of a theorem of Rosendal on finite approximations of actions on continuous structures. In Section 8 we prove the second half of Theorem~\ref{thm:3}.

\section{Preliminaries}

We first recall some basic definitions in continuous first-order logic. This is mostly following \cite{BU2008} but we do deviate at various places for the convenience of our presentation.

\begin{definition}{\ }
\begin{enumerate}
\item[(i)] A {\em modulus of continuity} is a function $$\mathfrak{u}:(0,\infty)\to (0, \infty]$$ such that $\lim_{\delta\to 0}\mathfrak{u}(\delta)=0$.
\item[(ii)] Let $\mathfrak{u}$ be a modulus of continuity. Let $(X_1, d_1)$ and $(X_2, d_2)$ be metric spaces. We say that a mapping $f: X_1\to X_2$ is {\em uniformly continuous  with respect to $\mathfrak{u}$} (or simply that $f$ {\em respects} $\mathfrak{u}$) if for all $x, y\in X_1$, we have
$$d_2(f(x),f(y))\leq\mathfrak{u}(d_1(x,y)).$$
\end{enumerate}
\end{definition}

Whenever convenient, we consider a modulus of continuity function $\mathfrak{u}$ to be defined at $0$ and set $\mathfrak{u}(0)=0$. Then $\mathfrak{u}$ is continuous at $0$. If $f$ respects a modulus of continuity $\mathfrak{u}$, then by redefining
$$ \mathfrak{u}'(\delta)=\sup_{\tau\leq\delta}\mathfrak{u}(\tau) $$
we get that $\mathfrak{u}'$ is nondecreasing but still a modulus of continuity that $f$ respects. Later when we consider superadditive and upper semicontinuous moduli of continuity, this redefinition with the supremum operation can still keep the superadditivity and upper semicontinuity, respectively, and at the same time makes the modulus of continuity nondecreasing. So throughout this paper we assume that all of the moduli of continuity we consider are nondecreasing.

\begin{definition} {\ }
\begin{enumerate}
\item[(i)] A function $u:[a, b]\to \mathbb{R}$ is {\em subadditive} if for all $x, y\in [a, b]$, if $x+y\in [a, b]$, then
$$ u(x+y)\leq u(x)+u(y). $$
\item[(ii)] A function $u:[a, b]\to \mathbb{R}$ is {\em superadditive} if for all $x, y\in [a, b]$, if $x+y\in [a, b]$, then
$$ u(x+y)\geq u(x)+u(y). $$
\item[(iii)] A function $u:[a, b]\to \mathbb{R}$ is {\em upper semicontinuous} if for all $x\in [a, b]$ and $\epsilon>0$, there is $\delta>0$ such that for all $x'\in (x-\delta, x+\delta)\cap [a,b]$, we have $u(x')<u(x)+\epsilon$.
\end{enumerate}
\end{definition}

It is easy to see that if $u:[0, b]\to\mathbb{R}$ is superadditive and nondecreasing then $u$ is strictly increasing. For a nondecreasing $u:[a,b]\to\mathbb{R}$, $u$ is upper semicontinuous iff for all $x\in [a,b]$, $\lim_{\delta\to 0^+} u(x+\delta)=u(x)$. The following lemma is also easy.

\begin{lemma}\label{lem:u*} Let $\mathfrak{u}$ be a modulus of continuity and $c>0$ be such that $\mathfrak{u}$ is nondecreasing and superadditive on $[0,c]$. Let $a=\mathfrak{u}(c)$. Define $\mathfrak{u}^*:[0,a)\to [0, c]$ by
$$ \mathfrak{u}^*(t)=\inf\{r\in (0,c]\,:\, t<\mathfrak{u}(r)\}. $$
Then $\mathfrak{u}^*(0)=0$ and $u^*$ is a subadditive, nondecreasing, and upper semicontinuous function. Moreover,
\begin{enumerate}
\item $\mathfrak{u}^*$ is continuous at $0$,

\item for all $r\in [0, c]$, $\mathfrak{u}^*(\mathfrak{u}(r))=r$, and
\item if $\mathfrak{u}$ is upper semicontinuous, then for all $t\in[0,a)$, $\mathfrak{u}(\mathfrak{u}^*(t))\geq t$.
\end{enumerate}
\end{lemma}

\begin{definition}
A {\em continuous} ({\em first-order}) {\em signature} consists of
\begin{enumerate}
\item[(i)] A specific binary predicate symbol $d$, known as the {\em distance symbol};
\item[(ii)] A set of relation symbols, each with a finite arity, such that for each $n$-ary relation symbol $R$ and for each $1\leq i\leq n$ there is a  modulus of continuity $\mathfrak{u}_{R,i}$, which we call the {\em uniform continuity modulus} of $R$ with respect to the $i$-th argument;
\item[(iii)] A set of function symbols, each with a finite arity, such that for each $n$-ary function symbol $F$ and for each $1\leq i\leq n$ there is a modulus of continuity $\mathfrak{u}_{F,i}$, known as the {\em uniform continuity modulus} of $F$ with respect to the $i$-th argument.
\end{enumerate}
\end{definition}

When we speak of a continuous signature $\mathcal{L}$, it is customary to omit the distance symbol $d$ and consider only the relation symbols and function symbols as elements of $\mathcal{L}$.

\begin{definition}
Let $\mathcal{L}$ be a continuous signature. A {\em continuous $\mathcal{L}$-pre-structure} $\mathcal{M}$ is a set $M$ together with an interpretation $d^M$ for $d$, $R^M$ for every relation symbol $R$ in $\mathcal{L}$, and $F^M$ for every function symbol $F$ in $\mathcal{L}$, such that
\begin{enumerate}
\item[(i)] $d^M$ is a metric on $M$;
\item[(ii)] for each $n$-ary relation symbol $R$, $R^M: M^n\to [0,1]$ is such that for each $1\leq i\leq n$ and for all $x_1,\dots,x_{i-1},y, z, x_{i+1},\dots x_n\in M$,
\[\label{eqn:ucl} \begin{array}{l}|R^M(x_1,\dots, x_{i-1},y,x_{i+1},\dots, x_n)-\\
\ \ \ \ \ \ \ \ \ \ R^M(x_1,\dots, x_{i-1},z,x_{i+1},\dots, x_n)| \leq \mathfrak{u}_{R,i}(d^M(y,z));\end{array}\tag{$\mbox{UC}_{\mathcal{L}}$}
\]
\item[(iii)] for each $n$-ary function symbol $F$, $F^M: M^n\to M$ is such that for each $1\leq i\leq n$ and for all $x_1,\dots,x_{i-1},y, z, x_{i+1},\dots x_n\in M$,
$$\begin{array}{l} d^M(F^M(x_1,\dots, x_{i-1},y,x_{i+1},\dots, x_n), \\
\ \ \ \ \ \ \ \ \ \ F^M(x_1,\dots, x_{i-1},z,x_{i+1},\dots, x_n))\leq \mathfrak{u}_{F,i}(d^M(y,z)).\end{array}$$
\end{enumerate}

A {\em continuous $\mathcal{L}$-structure} is a continuous $\mathcal{L}$-pre-structure in which $d^M$ is a complete metric.
\end{definition}

Note that in \cite{BU2008} the interpretation $d^{M}$ in an $\mathcal{L}$-pre-structure $\mathcal{M}$ can be a pseudo-metric instead of a metric. Our definition here is more restrictive. Also, in this paper we will be considering relational continuous structures, that is, those continuous $\mathcal{L}$-structures for which the continuous signature $\mathcal{L}$ does not contain function symbols. Thus only the moduli of continuity $\mathfrak{u}_{R,i}$ for relational symbols $R\in\mathcal{L}$ will be of interest to us. For this reason we use (\ref{eqn:ucl}) to denote the uniform continuity conditions only for relation symbols.

\begin{definition} Let $\mathcal{L}$ be a continuous signature and $\mathcal{M}$ be a continuous $\mathcal{L}$-pre-structure. We say that $\mathcal{M}$ is {\em finite, countable, separable, etc.} if $(M, d^M)$ is finite, countable, separable, etc. respectively.
\end{definition}

Note that a finite continuous $\mathcal{L}$-pre-structure is necessarily a continuous $\mathcal{L}$-structure. In general, given a continuous signature $\mathcal{L}$ and a continuous $\mathcal{L}$-pre-structure $\mathcal{M}$, one may consider the completion $(\bar{M}, d^{\bar{M}})$ of $(M, d^M)$ and define, for any relation symbol $R\in\mathcal{L}$ and function symbol $F\in\mathcal{L}$, $R^{\bar{M}}$ and $F^{\bar{M}}$ naturally. However, the resulting structure might not satisfy (\ref{eqn:ucl}) and hence might not be a continuous $\mathcal{L}$-structure. Under the assumption that all the moduli of continuity associated to $\mathcal{L}$ are upper semicontinuous, one can deduce (\ref{eqn:ucl}) and obtain a continuous $\mathcal{L}$-structure by completion.

\begin{definition}
Let $\mathcal{L}$ be a continuous signature and $\mathcal{M}$, $\mathcal{N}$ be continuous $\mathcal{L}$-pre-structures. We say that $\mathcal{N}$ is a {\em substructure} of $\mathcal{M}$ if all of the following hold:
\begin{enumerate}
\item[(i)] $N\subseteq M$;
\item[(ii)] $d^M\rest (N\times N)=d^N$;
\item[(iii)] for each $n$-ary relation symbol $R$ in $\mathcal{L}$ and for all $u_1,\dots, u_n\in N$,
$$ R^M(u_1,\dots, u_n)=R^N(u_1,\dots, u_n);$$
\item[(iv)] for each $n$-ary function symbol $F$ in $\mathcal{L}$ and for all $u_1,\dots, u_n\in N$,
$$ F^M(u_1,\dots, u_n)=F^N(u_1,\dots, u_n). $$
\end{enumerate}
When $\mathcal{N}$ is a substructure of $\mathcal{M}$, we also say that $\mathcal{M}$ is an {\em extension} of $\mathcal{N}$.
\end{definition}

Note that in case $\mathcal{M}$ and $\mathcal{N}$ are $\mathcal{L}$-structures, from the requirements of completeness of the specified metrics, it follows that the domain of a substructure $\mathcal{N}$ is necessarily a closed subset of the domain of the ambient structure $\mathcal{M}$.

\begin{definition}
Let $\mathcal{L}$ be a continuous signature and $\mathcal{M}, \mathcal{N}$ be continuous $\mathcal{L}$-pre-structures. An {\em isomorphism} from $\mathcal{M}$ to $\mathcal{N}$ is a bijection $\varphi: M\to N$ such that
\begin{enumerate}
\item[(i)] $\varphi$ is an isometry from $(M, d^M)$ to $(N, d^N)$;
\item[(ii)] for each $n$-ary relation symbol $R$ in $\mathcal{L}$ and for all $x_1,\dots, x_n\in M$,
$$ R^M(x_1,\dots, x_n)=R^N(\varphi(x_1),\dots, \varphi(x_n));$$
\item[(iii)] for each $n$-ary function symbol $F$ in $\mathcal{L}$ and for all $x_1,\dots, x_n\in M$,
$$ \varphi(F^M(x_1,\dots, x_n))=F^N(\varphi(x_1),\dots, \varphi(x_n)). $$
\end{enumerate}
An isomorphism from $\mathcal{M}$ to itself is called an {\em automorphism}. The set of all automorphisms of $\mathcal{M}$ is denoted $\Aut(\mathcal{M})$.
\end{definition}

\begin{definition}
Let $\mathcal{L}$ be a continuous signature and $\mathcal{M}, \mathcal{N}$ be continuous $\mathcal{L}$-pre-structures. An {\em isomorphic embedding} from $\mathcal{N}$ into $\mathcal{M}$ is an isomorphism from $\mathcal{N}$ to a substructure of $\mathcal{M}$.
\end{definition}

\begin{definition}
Let $\mathcal{L}$ be a continuous signature and $\mathcal{M}$ be a continuous $\mathcal{L}$-pre-structure. A {\em partial isomorphism} of $\mathcal{M}$ is an isomorphism $p$ from $\mathcal{N}_1$ to $\mathcal{N}_2$, where $\mathcal{N}_1$ and $\mathcal{N}_2$ are substructures of $\mathcal{M}$. The set of all partial isomorphisms of $\mathcal{M}$ is denoted $\Part(\mathcal{M})$.
\end{definition}

\section{Characterizations of Amalgamation Properties}

Throughout the rest of the paper we consider continuous signatures with only finitely many relation symbols. In this section we give characterizations of those moduli of continuity associated with a continuous signature $\mathcal{L}$ for which the classes of finite, countable, or all continuous $\mathcal{L}$-structures have the amalgamation property.

We first recall the definitions of the amalgamation property and of the strong amalgamation property.

\begin{definition} Let $\mathcal{L}$ be a continuous signature and $\mathcal{K}$ be a set of continuous $\mathcal{L}$-pre-structures.
\begin{enumerate}
\item[(i)] $\mathcal{K}$ has the {\em amalgamation property} (AP for short) if for any $\mathcal{M}$, $\mathcal{P}$, $\mathcal{Q}\in\mathcal{K}$ and isomorphic embeddings $\varphi: \mathcal{M}\to \mathcal{P}$ and $\psi:\mathcal{M}\to\mathcal{Q}$, there exist $\mathcal{N}\in\mathcal{K}$ and isomorphic embeddings $\iota:\mathcal{P}\to\mathcal{N}$ and $\tau:\mathcal{Q}\to \mathcal{N}$ such that $\iota\circ \varphi=\tau\circ\psi$, i.e., the following diagram commutes:
\begin{center}
\begin{tikzcd}
& \mathcal{P}\arrow[dr, dashed, "\iota"] & \\
\mathcal{M} \arrow[ur,"\varphi"]\arrow[dr,"\psi"] & & \mathcal{N} \\
& \mathcal{Q}\arrow[ur, dashed, "\tau"] &
\end{tikzcd}
\end{center}
The continuous $\mathcal{L}$-pre-structure $\mathcal{N}$ is called an {\em amalgam} of $\mathcal{P}$ and $\mathcal{Q}$ over $\mathcal{M}$.
\item[(ii)] $\mathcal{K}$ has the {\em strong amalgamation property} (SAP for short) if in the above definition we have in addition that
$$\mbox{range}(\iota)\cap\mbox{range}(\tau)=\mbox{range}(\iota\circ\varphi)=\mbox{range}(\tau\circ\psi). $$
\end{enumerate}
\end{definition}

Our characterizations will involve the following definition.

\begin{definition} Let $\mathcal{L}$ be a continuous signature with only finitely many relation symbols, where all the associated moduli of continuity are nondecreasing.
\begin{enumerate}
\item[(i)] $\mathcal{L}$ is {\em semiproper} if
\begin{itemize}
\item for any $n$-ary $R\in\mathcal{L}$ and $1\leq i\leq n$,
$$I_{R,i}=\{ r\in [0,+\infty)\,:\, \mathfrak{u}_{R,i}(r)<1\}$$ is bounded, and $\mathfrak{u}_{R,i}$ is superadditive on $I_{R,i}$; and \item for any $n$-ary $R\in\mathcal{L}$, where $n\geq 2$, and $1\leq i\leq n$, there is $K_{R,i}>0$ such that $\mathfrak{u}_{R,i}(r)=K_{R,i}r$ for all $r\in I_{R,i}$.
\end{itemize}
\item[(ii)] $\mathcal{L}$ is {\em proper} if $\mathcal{L}$ is semiproper and each $\mathfrak{u}_{R,i}$ is upper semicontinuous on $I_{R,i}$.
\end{enumerate}
\end{definition}

The amalgamation property requires a method to construct extensions of partially defined continuous structures to fully defined ones. In the following we first develop this method. For this, we fix a continuous signature $\mathcal{L}$ with only finitely many relation symbols, where all the associated moduli of continuity are nondecreasing.

\begin{definition} Let $(X, d^X)$ be a metric space and $\mathfrak{u}$ be a nondecreasing modulus of continuity. For any $x, y\in X$, define a pseudo-metric on $X$ by
$$\begin{array}{rcl} d^X_{\mathfrak{u}}(x,y)&=&\inf\left\{\displaystyle\sum_{i=1}^{m}\mathfrak{u}(d^X(z_{i-1}, z_i))\,:\, z_0=x, z_1,\dots, z_m=y\in X\right\}.
\end{array}
$$
\end{definition}

\begin{definition} Let $(X, d^X)$ be a metric space and let  $R\in\mathcal{L}$ be $n$-ary. Define
$$ d^X_R(\overline{x}, \overline{y})=\displaystyle\sum_{i=1}^n d^X_{\mathfrak{u}_{R,i}}(x_i, y_i)$$
for $\overline{x}=(x_1,\dots, x_n), \overline{y}=(y_1,\dots, y_n)\in X^n$.
\end{definition}

For each $n$-ary $R\in \mathcal{L}$, $d^X_R$ is a pseudo-metric on $X^n$.

\begin{lemma}\label{lem:1L} Let $(M,d^M)$ be a metric space with $R^M$ defined on $M^n$ for all $n$-ary $R\in\mathcal{L}$. Then $\mathcal{M}=(M, d^M, (R^M)_{R\in\mathcal{L}})$ is a continuous $\mathcal{L}$-pre-structure, i.e., {\rm (\ref{eqn:ucl})} holds for $\mathcal{M}$ iff for any $n$-ary $R\in\mathcal{L}$, $R^M$ is $1$-Lipschitz with respect to $d^M_R$.
\end{lemma}

\begin{proof} Suppose $R\in\mathcal{L}$ is $n$-ary and $R^M$ is $1$-Lipschitz with respect to $d^M_R$. Let $\overline{x}, \overline{y}\in M^n$ only differ at the $i$-th coordinate. Then
$$ |R^M(\overline{x})-R^M(\overline{y})|\leq d^M_{\mathfrak{u}_{R,i}}(x_i,y_i)\leq \mathfrak{u}_{R,i}(d^M(x_i,y_i)). $$

Conversely, suppose (\ref{eqn:ucl}) holds for $\mathcal{M}$ for $n$-ary $R\in\mathcal{L}$. Let $\overline{x}, \overline{y}\in M^n$. For $0\leq p\leq n$, let
$$ \overline{z}^p=(x_1,\dots, x_p, y_{p+1},\dots, y_n). $$
For each $1\leq i\leq n$, let $w_{i,0}=x_i, w_{i,1},\dots, w_{i,m_i}=y_i\in M$ be arbitrary and let
$$ \overline{z}^{i,j}=(x_1,\dots, x_{i-1}, w_{i,j}, y_{i+1},\dots, y_n) $$
for $0\leq j\leq m_i$. Then
$$\begin{array}{rcl}
|R^M(\overline{x})-R^M(\overline{y})|&\leq &\displaystyle\sum_{i=1}^{n}|R^M(\overline{z}^{i-1})-R^M(\overline{z}^{i})| \\
&\leq & \displaystyle\sum_{i=1}^n\sum_{j=1}^{m_i}|R^M(\overline{z}^{i,j-1})-R^M(\overline{z}^{i,j})| \\
&\leq &\displaystyle\sum_{i=1}^n \sum_{j=1}^{m_i} \mathfrak{u}_{R,i}(d^M(w_{i,j-1}, w_j))
\end{array}
$$
Taking the infimum among all $w_{i,0}, \dots, w_{i,m_i}$, we get
$$ |R^M(\overline{x})-R^M(\overline{y})|\leq \displaystyle\sum_{i=1}^n d^M_{\mathfrak{u}_{R,i}}(x_i,y_i)=d^M_R(\overline{x}, \overline{y}). $$
\end{proof}

\begin{definition}{\ }
\begin{enumerate}
\item[(i)] A {\em partially defined continuous $\mathcal{L}$-pre-structure} is a tuple
    $$\mathcal{X}=(X, d^X, (R^X)_{R\in\mathcal{L}})$$ where $(X, d^X)$ is a metric space, and for each $n$-ary $R\in \mathcal{L}$, $R^X:\dom(R^X)\to [0,1]$ is a function with domain $\dom(R^X)\subseteq X^n$ which is $1$-Lipschitz with respect to $d^X_R$.
\item[(ii)] Let $\mathcal{X}$ be a partially defined continuous $\mathcal{L}$-pre-structure. A continuous $\mathcal{L}$-pre-structure $\mathcal{M}$ is a {\em conservative extension} of $\mathcal{X}$ if $(M,d^M)=(X,d^X)$ and for any $R\in\mathcal{L}$, $R^M\rest \dom(R^X)=R^X$.
\end{enumerate}
\end{definition}

\begin{lemma}\label{lem:gconservative} Any partially defined continuous $\mathcal{L}$-pre-structure has a conservative extension.
\end{lemma}

\begin{proof} Let $\mathcal{X}$ be a partially defined continuous $\mathcal{L}$-pre-structure. Define $(M, d^M)=(X,d^X)$ and for any $n$-ary $R\in\mathcal{L}$ and $\overline{x}\in M^n$, define
$$ R^M(\overline{x})=\max\{0,\sup\{ R^X(\overline{y})-d^X_R(\overline{x}, \overline{y})\,:\, \overline{y}\in\dom(R^X)\}\}.
$$
It is clear that if $\overline{x}\in\dom(R^X)$, then $R^M(\overline{x})=R^X(\overline{x})$.

To complete the proof, it suffices to verify that $\mathcal{M}$ satisfies (\ref{eqn:ucl}). By Lemma~\ref{lem:1L} it suffices to show that for any $n$-ary $R\in\mathcal{L}$, $R^M$ is $1$-Lipschitz with respect to $d^M_R$. Suppose $\overline{x}, \overline{y}\in M^n$. Assume toward a contradiction that $R^M(\overline{x})-R^M(\overline{y})>d^M_R(\overline{x},\overline{y})$. Let $\epsilon>0$ such that
$$ R^M(\overline{x})-R^M(\overline{y})>d^M_R(\overline{x},\overline{y})+\epsilon. $$
In particular $R^M(\overline{x})>0$.
By the definition of $R^M(\overline{x})$ there is $\overline{z}\in \dom(R^X)$ such that
$$ R^X(\overline{z})-d^X_R(\overline{x},\overline{z})> R^M(\overline{x})-\epsilon. $$
Then
$$\begin{array}{rcl} R^X(\overline{z})-R^M(\overline{y})&=& R^X(\overline{z})-R^M(\overline{x})+R^M(\overline{x})-R^M(\overline{y}) \\
&>&  d^X_R(\overline{x},\overline{z})-\epsilon+ d^X_R(\overline{x},\overline{y})+\epsilon \\
&\geq&  d^X_R(\overline{z},\overline{y}).
\end{array}$$
Thus $R^X(\overline{z})- d^X_R(\overline{z},\overline{y})>R^M(\overline{y})$, contradicting the definition of $R^M(\overline{y})$.
\end{proof}

We are now ready to prove the first characterization.

\begin{theorem}\label{thm:finap} Let $\mathcal{L}$ be a continuous signature with only finitely many relation symbols, where all the associated moduli of continuity are nondecreasing. The following are equivalent:
\begin{enumerate}
\item $\mathcal{L}$ is semiproper.
\item The class of all finite continuous $\mathcal{L}$-structures has the AP.
\item The class of all finite continuous $\mathcal{L}$-structures has the SAP.
\end{enumerate}
\end{theorem}

\begin{proof} We prove the implications (1)$\Rightarrow$(3) and (2)$\Rightarrow$(1).

For (1)$\Rightarrow$(3), let $\mathcal{M}$, $\mathcal{P}$, $\mathcal{Q}$ be finite continuous $\mathcal{L}$-structures, and let $\varphi: \mathcal{M}\to \mathcal{P}$ and $\psi:\mathcal{M}\to\mathcal{Q}$ be isomorphic embeddings. To ease notation we regard $\varphi$ and $\psi$ as identity maps. Let $X$ be the disjoint union of $M$, $P\setminus M$, and $Q\setminus M$. Define a metric $d^X$ on $X$ by the obvious definitions except for $x\in P\setminus M$ and $y\in Q\setminus M$ we set
$$ d^X(x, y)=\inf\{d^P(x,z)+d^Q(z,y)\,:\,z\in M\}.$$
For an $n$-ary $R\in\mathcal{L}$, $R^X$ is naturally defined on
$\dom(R^X)= P^n \cup Q^n$.

We verify that $\mathcal{X}=(X,d^X, (R^X)_{R\in\mathcal{L}})$ is a partially defined continuous $\mathcal{L}$-pre-structure. For this, fix an $n$-ary $R\in\mathcal{L}$ and consider $\overline{x}=(x_1,\dots, x_n), \overline{y}=(y_1,\dots, y_n)\in \dom(R^X)$.

First suppose $n=1$. In this case we claim that for any $u, v\in P\cup Q$,
$$ |R^X(u)-R^X(v)|\leq \mathfrak{u}_{R,1}(d^X(u,v)). $$
From the claim it follows quickly that
$$ |R^X(x_1)-R^X(y_1)|\leq d^X_R(x_1, y_1). $$
To prove the claim, we only need to consider the situation where $u\in P\setminus M$ and $v\in Q\setminus M$. Note that for any $z\in M$,
$$\begin{array}{rcl} |R^X(u)-R^X(v)|&\leq& |R^X(u)-R^X(z)|+|R^X(z)-R^X(v)|\\
&=&|R^P(u)-R^P(z)|+|R^Q(z)-R^Q(v)| \\
&\leq& \mathfrak{u}_{R,1}(d^P(u, z))+\mathfrak{u}_{R,1}(d^Q(z,v))
\end{array}
$$
by (\ref{eqn:ucl}) for $R^P$ and $R^Q$.
If $d^X(u,v)\not\in I_{R,1}$, then we certainly have
$$ |R^X(u)-R^X(v)|\leq 1\leq \mathfrak{u}_{R,1}(d^X(u,v)). $$
Otherwise, we may consider only those $z\in M$ with $d^P(u, z)+d^Q( z,v)\in I_{R,1}$. By the superadditivity of $\mathfrak{u}_{R,1}$ on $I_{R,1}$, we have
$$ |R^X(u)-R^X(v)|\leq \mathfrak{u}_{R,1}(d^P(u,z)+d^Q(z,v)).
$$
Taking the infimum over such $z\in M$, and noting that $M$ is finite, we have by the monotonicity of $\mathfrak{u}_{R,1}$ that
$$ |R^X(u)-R^X(v)|\leq \mathfrak{u}_{R,1}(d^X(u,v)). $$

Next suppose $n\geq 2$. In this case we have that for all $1\leq i\leq n$, there is some $K_{R,i}>0$ such that $\mathfrak{u}_{R,i}(r)=K_{R,i} r$ for all $r\in I_{R,i}$. It is easy to check that for any $u, v\in P$,
$$d^X_{\mathfrak{u}_{R,i}}(u, v)= d^P_{\mathfrak{u}_{R,i}}(u, v)=K_{R,i}d^P(u, v). $$
Similarly for $u, v\in Q$. From these, the cases when $\overline{x}, \overline{y}\in P^n$ and $\overline{x}, \overline{y}\in Q^n$ quickly follow from Lemma~\ref{lem:1L}.

Next we consider the case where $\overline{x}\in P^n$ and $\overline{y}\in Q^n$. Define
$$\begin{array}{rcl}
S_0 & =& \{  i\,:\, x_i, y_i\in M\} \\
S_1 & =& \{  i\,:\, x_i\in P\setminus M, y_i\in M\} \\
S_2 & =& \{  i\,:\, x_i\in M, y_i\in Q\setminus M\} \\
S_3 & =& \{ i \,:\, x_i\in P\setminus M, y_i\in Q\setminus M\}.
\end{array}
$$
Define $\overline{u}=(u_1,\dots, u_n), \overline{v}=(v_1,\dots, v_n)\in\dom(R^X)$ by
$$ u_i=\left\{\begin{array}{ll} x_i & \mbox{ if $i\in S_2\cup S_3$} \\ y_i & \mbox{ if $i\in S_0\cup S_1$} \end{array}\right. $$
and
$$ v_i=\left\{\begin{array}{ll} x_i & \mbox{ if $i\in S_2$} \\ y_i & \mbox{ if $i\in S_0\cup S_1\cup S_3$.} \end{array}\right. $$

Given any $\overline{z}\in M^n$, by Lemma~\ref{lem:1L} we have
$$\begin{array}{rcl} |R^X(\overline{u})-R^X(\overline{v})| &\leq & |R^X(\overline{u})-R^X(\overline{z})|+|R^X(\overline{z})-R^X(\overline{v})| \\
& \leq & d^P_R(\overline{u},\overline{z})+d^Q_R(\overline{z},\overline{v}) \\
& =& \displaystyle\sum_{i=1}^n d^P_{\mathfrak{u}_{R,i}}(u_i,z_i)+\sum_{i=1}^n d^Q_{\mathfrak{u}_{R,i}}(z_i,v_i) \\
&=& \displaystyle\sum_{i=1}^n K_{R,i}(d^P(u_i, z_i)+d^Q(z_i, v_i))
\end{array}
$$
Taking the infimum over all $\overline{z}\in M^n$, we get that
$$ |R^X(\overline{u})-R^X(\overline{v})|\leq \displaystyle\sum_{i=1}^n K_{R,i}d^X(u_i,v_i)=\sum_{i=1}^nd^X_{\mathfrak{u}_{R,i}}(u_i, v_i)= d^X_R(\overline{u},\overline{v}).
$$
Thus
$$\begin{array}{rcl} & & |R^X(\overline{x})-R^X(\overline{y})|\\
& \leq & |R^X(\overline{x})-R^X(\overline{u})|+|R^X(\overline{u})-R^X(\overline{v})|+|R^X(\overline{v})-R^X(\overline{y})| \\
& \leq & d_R^X(\overline{x},\overline{u})+d^X_R(\overline{u},\overline{v})+d^X_R(\overline{v},\overline{y}) \\
& \leq & \displaystyle\sum_{i\in S_0\cup S_1} d^X_{\mathfrak{u}_{R,i}}(x_i, y_i) + \sum_{i\in S_3} d^X_{\mathfrak{u}_{R,i}}(x_i, y_i)+\sum_{i\in S_2} d^X_{\mathfrak{u}_{R,i}}(x_i,y_i) \\
&=& d^X_R(\overline{x},\overline{y}).
\end{array}
$$

We have thus completed the verification that $\mathcal{X}$ is a partially defined continuous $\mathcal{L}$-pre-structure.
Applying  Lemma~\ref{lem:gconservative} to $\mathcal{X}$, we obtain a conservative extension $\mathcal{N}$ of $\mathcal{X}$. This $\mathcal{N}$ satisfies the requirements of the strong amalgamation property.

(2)$\Rightarrow$(1). Let $R\in\mathcal{L}$ be $n$-ary and $1\leq i\leq n$. We claim that
$$\mathfrak{u}_{R,i}(r_1+r_2)\geq \mathfrak{u}_{R,i}(r_1)+\mathfrak{u}_{R,i}(r_2)$$ for all $r_1, r_2>0$ with $r_1+r_2\in I_{R,i}=\{r\,:\, \mathfrak{u}_{R,i}(r)< 1\}$. To prove the claim, let $r_1, r_2>0$ so that $r_1+r_2\in I_{R,i}$. Consider finite continuous $\mathcal{L}$-structures $\mathcal{M}$, $\mathcal{P}$, $\mathcal{Q}$ where for all $R'\in\mathcal{L}$ with $R'\neq R$, the values of ${R'}^M$, ${R'}^P$, ${R'}^Q$ are identically $0$, and
\begin{itemize}
\item $M=\{x_0\}$, $R^M(\overline{x})=\mathfrak{u}_{R,i}(r_1)$ where $\overline{x}=(x_0,\dots, x_0)\in M^n$;
\item $P= \{ x_0, x_1\}$, $d^P(x_0, x_1)=r_1$, and for all $\overline{y}=(y_1, \dots, y_n)\in P^n$,
$$R^P(\overline{y})=\left\{\begin{array}{ll} \mathfrak{u}_{R, i}(r_1) & \mbox{ if $y_i=x_0$} \\ 0 & \mbox{otherwise;} \end{array}\right.$$
\item $Q=\{x_0, x_2\}$, $d^Q(x_0,x_2)=r_2$, and for all $\overline{y}=(y_1, \dots, y_n)\in Q^n$,
$$R^Q(\overline{y})=\left\{\begin{array}{ll} \mathfrak{u}_{R, i}(r_1) & \mbox{ if $y_i=x_0$} \\ \min\{1, \mathfrak{u}_{R, i}(r_1)+\mathfrak{u}_{R, i}(r_2)\} & \mbox{otherwise.} \end{array}\right.$$
\end{itemize}
Let $\mathcal{N}$ be an amalgam of $\mathcal{P}$ and $\mathcal{Q}$ over $\mathcal{M}$. Let $\overline{x}'$ be obtained from $\overline{x}$ by replacing the $i$-th coordinate of $\overline{x}$ by $x_1$, and similarly $\overline{x}''$ be obtained from $\overline{x}$ by replacing the $i$-th coordinate of $\overline{x}$ by $x_2$. Then by the monotonicity of $\mathfrak{u}_{R,i}$, we have
$$\begin{array}{rcl} \mbox{min}\{1, \mathfrak{u}_{R,i}(r_1)+\mathfrak{u}_{R,i}(r_2)\}&=&|R^N(\overline{x}')-R^N(\overline{x}'')| \\
&\leq& \mathfrak{u}_{R,i}(d^N(x_1, x_2)) \\
&\leq& \mathfrak{u}_{R,i}(d^P(x_1,x_0)+d^Q(x_0,x_2))=\mathfrak{u}_{R,i}(r_1+r_2).
\end{array}$$
Since $\mathfrak{u}_{R,i}(r_1+r_2)<1$, we have $\mathfrak{u}_{R,i}(r_1)+\mathfrak{u}_{R,i}(r_2)\leq \mathfrak{u}_{R,i}(r_1+r_2)$ as desired.

It follows that $I_{R,i}$ is bounded. In fact, if $I_{R,i}$ were unbounded then $I_{R,i}=[0,+\infty)$ and $M=\sup\{\, \mathfrak{u}_{R,i}(r)\,:\, r\geq 0\}\leq 1$. Let $r_1, r_2>0$ be such that $\mathfrak{u}_{R,i}(r_1),\mathfrak{u}_{R,i}(r_2)>M/2$. Then $\mathfrak{u}_{R,i}(r_1+r_2)>M$, contradicting the definition of $M$.

For the rest of the semiproperness of $\mathcal{L}$, suppose $n\geq 2$. Without loss of generality consider $i=1$. The proof for $1<i\leq n$ is similar. Let $r_1, r_2$ be such that $r_1+r_2\in I_{R,i}$. Let $r>r_1+r_2$ be an upper bound for $I_{R,n}$. Consider finite continuous $\mathcal{L}$-structures $\mathcal{M}$, $\mathcal{P}$, $\mathcal{Q}$ where for all $R'\in\mathcal{L}$ with $R'\neq R$, the values of ${R'}^M$, ${R'}^P$, ${R'}^Q$ are identically $0$, and
\begin{itemize}
\item $M=\{x_0, u_0\}$, $d^M(x_0, u_0)=r_1+r_2$, $R^M(\overline{x})=0$ for all $\overline{x}\in M^n$;
\item $P= \{ x_0, u_0, x_1\}$, $d^P(x_0, x_1)=d^P(u_0, x_1)=r>r_1+r_2$,
$$R^P(\overline{y})=\left\{\begin{array}{ll} \mathfrak{u}_{R,1}(r_1+r_2) & \mbox{ if $y_1=x_0$ and $y_n=x_1$} \\
0 & \mbox{ otherwise;} \end{array}\right.$$
\item $Q=\{x_0, u_0, x_2\}$, $d^Q(x_0,x_2)=r_1$, $d^Q(u_0, x_2)=r_2$,
$R^Q(\overline{z})=0$ for all $\overline{z}\in Q^n$.
\end{itemize}
Let $\mathcal{N}$ be an amalgam of $\mathcal{P}$ and $\mathcal{Q}$ over $\mathcal{M}$. Let
$$\begin{array}{rcl}
\overline{y}&=& (x_0, x_0, \dots, x_0, x_1)\in P \\
\overline{y}'&=& (u_0, x_0, \dots, x_0, x_1)\in P \\
\overline{w}&=& (x_2, x_0, \dots, x_0, x_1)\in N.
\end{array}
$$
Then
$$\begin{array}{rcl}
\mathfrak{u}_{R,1}(r_1+r_2)&=& |R^N(\overline{y})-R^N(\overline{y}')| \\
&\leq& |R^N(\overline{y})-R^N(\overline{w})|+|R^N(\overline{w})-R^N(\overline{y}')| \\
&\leq& \mathfrak{u}_{R,1}(d^N(x_0, x_2))+\mathfrak{u}_{R,1}(d^N(u_0, x_2)) \\
&=& \mathfrak{u}_{R,1}(r_1)+\mathfrak{u}_{R,1}(r_2).
\end{array}
$$
Thus we actually have $\mathfrak{u}_{R,1}(r_1+r_2)=\mathfrak{u}_{R,1}(r_1)+\mathfrak{u}_{R,1}(r_2)$ for all $r_1+r_2\in I_{R,1}$. This implies that there is $K_1>0$ such that $\mathfrak{u}_{R,1}(r)=K_1r$ for all $r\in I_{R,1}$.
\end{proof}

\begin{theorem}\label{thm:allap} Let $\mathcal{L}$ be a continuous signature with only finitely many relation symbols, where all the associated moduli of continuity are nondecreasing. The following are equivalent:
\begin{enumerate}
\item $\mathcal{L}$ is proper.
\item The class of all countable continuous $\mathcal{L}$-pre-structures has the AP.
\item The class of all countable continuous $\mathcal{L}$-pre-structures has the SAP.
\item The class of all continuous $\mathcal{L}$-pre-structures has the AP.
    \item The class of all continuous $\mathcal{L}$-pre-structures has the SAP.
\end{enumerate}
\end{theorem}

\begin{proof} The proof of (1)$\Rightarrow$(3) follows exactly the same proof of (1)$\Rightarrow$(3) of Theorem~\ref{thm:finap}, except in the last step of the verification of the $1$-Lipschitz property for the unary relation symbol, instead of using the finiteness of $M$ we use the upper semicontinuity of $\mathfrak{u}_{R,1}$. The proof of (1)$\Rightarrow$(5) is identical.

Conversely, for (2)$\Rightarrow$(1) or (4)$\Rightarrow$(1), we make the following observation. If $R\in\mathcal{L}$ is unary and not upper semicontinuous at $r_0\in I_{R,1}$, then choose some $t_0\in (0,1)$ such that
$$ \mathfrak{u}_{R,1}(r_0)<t_0<\min\{\inf\{\mathfrak{u}_{R,1}(r): r>r_0\}, 1\}. $$

Consider countable continuous $\mathcal{L}$-pre-structures $\mathcal{M}$, $\mathcal{P}$, $\mathcal{Q}$ where for all $R'\in\mathcal{L}$ with $R'\neq R$, the values of ${R'}^M$, ${R'}^P$, ${R'}^Q$ are identically $0$, and
\begin{itemize}
\item $M=\{x_i\,:\, i\geq 1\}$, $d^M(x_i, x_j)=|2^{-i}-2^{-j}|$, for all $i, j\geq 1$, and $R^M(x_0)=0$ for all $i\geq 1$;
\item $P= M\cup\{x_0\}$, $d^P(x_0, x_i)=r_0+2^{-i}$,
$R^P(x_0)=t_0$;
\item $Q=M\cup\{y_0\}$, $d^Q(y_0,x_i)=2^{-i}$, $R^Q(y_0)=0$.
\end{itemize}
Let $\mathcal{N}$ be an amalgam of $\mathcal{P}$ and $\mathcal{Q}$ over $\mathcal{M}$. Then
$$\begin{array}{rcl} r_0=|d^N(x_0, x_i)-d^N(y_0, x_i)|&\leq& d^N(x_0,y_0) \\
 &\leq& d^N(x_0, x_i)+d^N(y_0,x_i)=r_0+2^{-i+1}
\end{array} $$
for all $i\geq 1$. Letting $i\to\infty$, we get $d^N(x_0,y_0)=r_0$. Now
$$ t_0=|R^N(x_0)-R^N(y_0)|\leq \mathfrak{u}_{R,1}(d^N(x_0,y_0))=\mathfrak{u}_{R,1}(r_0),$$
a contradiction.
\end{proof}

Finally we give a characterization the AP and the SAP for the class of all (countable) $\mathcal{L}$-structures. For this we need to introduce the following notion of strong semiproperness.

\begin{definition} Let $\mathcal{L}$ be a continuous signature with only finitely many relation symbols, where all the associated moduli of continuity are nondecreasing. We say that $\mathcal{L}$ is {\em strongly semiproper} if $\mathcal{L}$ is semiproper and for any unary $R\in \mathcal{L}$ and $r_1, r_2>0$ with $r_1+r_2\in I_{R,1}$,
$$ \mathfrak{u}_{R,1}(r_1+r_2)\geq \inf\{ \mathfrak{u}_{R,1}(r)\,:\, r> r_1\}+\inf \{\mathfrak{u}_{R,1}(r)\,:\, r> r_2\}. $$
\end{definition}

It is obvious that strong semiproperness implies semiproperness, and it is easy to see that properness implies strong semiproperness.

\begin{theorem}\label{thm:allstructuresap} Let $\mathcal{L}$ be a continuous signature with only finitely many relation symbols, where all the associated moduli of continuity are nondecreasing. The following are equivalent:
\begin{enumerate}
\item $\mathcal{L}$ is strongly semiproper.
\item The class of all countable continuous $\mathcal{L}$-structures has the AP.
\item The class of all countable continuous $\mathcal{L}$-structures has the SAP.
\item The class of all continuous $\mathcal{L}$-structures has the AP.
    \item The class of all continuous $\mathcal{L}$-structures has the SAP.
\end{enumerate}
\end{theorem}

\begin{proof} The proof of (1)$\Rightarrow$(3) again follows exactly the same proof of (1)$\Rightarrow$(3) of Theorem~\ref{thm:finap}, except that we need to modify the argument in the last step of the verification of the $1$-Lipschitz property for the unary relation symbol. First we note that the metric space $(X,d^X)$ is complete if both $(P,d^P)$ and $(Q,d^Q)$ are complete, and countable if both $P$ and $Q$ are countable. Let $R\in\mathcal{L}$ be unary, $u\in P\setminus M$ and $v\in Q\setminus M$. Then
$$ |R^X(u)-R^X(v)|\leq \inf\{ \mathfrak{u}_{R,1}(d^P(u,z))+\mathfrak{u}_{R,1}(d^Q(z, v))\,:\, z\in M\}. $$
We may choose $\{z_n\}\subseteq M$ such that
$$ d^X(u,v)=\inf\{ d^P(u,z)+d^Q(z, v)\,:\, z\in M\}=\lim_n \left(d^P(u,z_n)+d^Q(z_n, v)\right) $$
and both $\{d^P(u, z_n)\}_n$ and $\{d^Q(z_n, v)\}_n$ are monotone. Let
$$ r_1=\lim_n d^P(u, z_n) \mbox{ and } r_2=\lim_n d^Q(u, z_n). $$
If $r_1+r_2\not\in I_{R,1}$ then we have
$$ |R^X(u)-R^X(v)|\leq \mathfrak{u}_{R,1}(r_1+r_2)=\mathfrak{u}_{R,1}(d^X(u,v)). $$
Suppose $r_1+r_2\in I_{R,1}$. Then
$$\begin{array}{rcl} |R^X(u)-R^X(v)|&\leq& \inf\{ \mathfrak{u}_{R,1}(d^P(u,z))+\mathfrak{u}_{R,1}(d^Q(z, v))\,:\, z\in M\} \\
&\leq& \lim_n \left(\mathfrak{u}_{R,1}(d^P(u,z_n))+\mathfrak{u}_{R,1}(d^Q(z_n, v))\right)\\
&\leq & \inf\{\mathfrak{u}_{R,1}(r)\,:\, r>r_1\}+\inf\{\mathfrak{u}_{R,1}(r)\,:\, r>r_2\} \\
&\leq& \mathfrak{u}_{R,1}(r_1+r_2)=\mathfrak{u}_{R,1}(d^X(u,v)).
\end{array}$$
The proof of (1)$\Rightarrow$(5) is identical.

Conversely, we prove (2)$\Rightarrow$(1) and (4)$\Rightarrow$(1). By the proof of (2)$\Rightarrow$(1) of Theorem~\ref{thm:finap}, $\mathcal{L}$ is semiproper. To see that $\mathcal{L}$ is strongly semiproper, let $R\in\mathcal{L}$ be unary and let $r_1, r_2>0$ be such that $r_1+r_2\in I_{R,1}$. Without loss of generality assume $r_1\leq r_2$. Let $M=\{x_n\,:\, n\in\mathbb{N}\}$ be a countably infinite set and define $d^M(x_n,x_m)=r_1$ for all distinct $n,m\in \mathbb{N}$. Define $R^M(x)=\inf\{\mathfrak{u}_{R,1}(r)\,:\,r>r_1\}$ for all $x\in M$ and for any other $\tilde{R}\in\mathcal{L}$, let $\tilde{R}^M$ be identically $0$. This defines a countable continuous $\mathcal{L}$-structure $\mathcal{M}$. Let $P=M\cup\{y\}$ where $y$ is a fresh point. Define $d^P(x_n,y)=(1+2^{-n})r_1$ for all $n\in\mathbb{N}$, $R^P(y)=0$. For $\tilde{R}\in\mathcal{L}$, let $\tilde{R}^P$ be identically $0$. This defines a countable continuous $\mathcal{L}$-structure $\mathcal{P}$ that is an extension of $\mathcal{M}$. Let $Q=M\cup\{z\}$ where $z$ is a fresh point. Define $d^Q(x_n, z)=r_2+2^{-n}r_1$ and
$$ R^Q(z)=\mbox{min}\left\{1, \inf\{ \mathfrak{u}_{R,1}(r)\,:\, r>r_1\}+\inf\{\mathfrak{u}_{R,1}(r)\,:\, r>r_2\}\right\}. $$
For $\tilde{R}\in\mathcal{L}$, let $\tilde{R}^Q$ be identically $0$. This defines a countable continuous $\mathcal{L}$-structure $\mathcal{Q}$ that is an extension of $\mathcal{M}$. Let $\mathcal{N}$ be an amalgam of $\mathcal{P}$ and $\mathcal{Q}$ over $\mathcal{M}$. Then for any $n\in\mathbb{N}$,
$$ d^N(y,z)\leq d^P(x_n,y)+d^Q(x_n, z)=r_1+r_2+2^{-n+1}r_1. $$
Since $n$ is arbitrary, we have $d^N(y,z)\leq r_1+r_2$. Then
$$\begin{array}{rcl}
 & & \mbox{min}\left\{1, \inf\{ \mathfrak{u}_{R,1}(r)\,:\, r>r_1\}+\inf\{\mathfrak{u}_{R,1}(r)\,:\, r>r_2\}\right\} \\
 &=&|R^Q(z)-R^P(y)|=|R^N(z)-R^N(y)|\\
 &=& \mathfrak{u}_{R,1}(d^N(y,z))\leq \mathfrak{u}_{R,1}(r_1+r_2).
 \end{array}$$
This shows that $\mathcal{L}$ is strongly semiproper.
\end{proof}

\section{Urysohn Continuous Structures}

In this section we give a Kat\v{e}tov-style construction of a Urysohn continuous $\mathcal{L}$-structure for any proper continuous signature $\mathcal{L}$. This will not only establish the usual properties of the Urysohn structure such as universality, ultrahomogeneity, and uniqueness, but also prove that its automorphism group is a universal Polish group.

Throughout this section we assume that $\mathcal{L}$ is a proper continuous signature unless explicitly stated otherwise. For each unary $R\in \mathcal{L}$, the modulus of continuity $\mathfrak{u}_{R,1}$ is nondecreasing, superadditive, and upper semicontinous on $I_{R,1}$. As in Lemma~\ref{lem:u*}, we define $\mathfrak{u}_{R,1}^*: [0,1)\to I_{R,1}$, which is in particular subadditive on $[0,1)$. For $n$-ary $R\in\mathcal{L}$ where $n\geq 2$, we fix $K_{R,i}>0$ for $1\leq i\leq n$ so that $\mathfrak{u}_{R,i}(r)=K_{R,i}r$ for all $r\in I_{R,i}$.

\begin{definition} Let $\mathcal{L}$ be a proper continuous signature, $\mathcal{M}$, $\mathcal{N}$, $\mathcal{U}$ be continuous $\mathcal{L}$-pre-structures, and $\mathcal{K}$ a class of continuous $\mathcal{L}$-pre-structures.
\begin{enumerate}
\item[(i)] $\mathcal{N}$ is a {\em one-point extension} of $\mathcal{M}$ if $\mathcal{M}$ is a substructure of $\mathcal{N}$ and $N\setminus M$ is a singleton.
\item[(ii)] $\mathcal{U}$ has the {\em Urysohn property} if given any finite continuous $\mathcal{L}$-structure $\mathcal{M}$, a one-point extension $\mathcal{N}$ of $\mathcal{M}$, and an isomorphic embedding $\varphi$ from $\mathcal{M}$ into $\mathcal{U}$, there is an isomorphic embedding $\psi$ from $\mathcal{N}$ into $\mathcal{U}$ such that $\psi\!\upharpoonright\! M=\varphi$.
We say $\mathcal{U}$ is {\em Urysohn} if it satisfies the Urysohn property.
\item[(iii)] $\mathcal{U}$ is {\em universal} for $\mathcal{K}$ if for any $\mathcal{M}\in\mathcal{K}$ there is an isomorphic embedding from $\mathcal{M}$ into $\mathcal{U}$.
\item[(iv)] $\mathcal{U}$ is {\em ultrahomogeneous} if for any finite substructures $\mathcal{M}$ and $\mathcal{N}$ of $\mathcal{U}$ and an isomorphism $\varphi$ between $\mathcal{M}$ and $\mathcal{N}$, there is an automorphism $\psi$ of $\mathcal{U}$ such that $\psi\rest M=\varphi$.
\end{enumerate}
\end{definition}

Given a proper continuous signature $\mathcal{L}$, let
$$ \Kfin=\mbox{ the class of all finite continuous $\mathcal{L}$-structures} $$
and
$$ \Ksep=\mbox{ the class of all separable continuous $\mathcal{L}$-structures}. $$
By a standard argument, a separable continuous $\mathcal{L}$-structure has the Urysohn property iff it is ultrahomogeneous and universal for $\Kfin$. In this case, it is in fact universal for $\Ksep$ and unique up to isomorphism.

By Theorem~\ref{thm:allap}, the class of all continuous $\mathcal{L}$-pre-structures has the SAP. We call the continuous $\mathcal{L}$-pre-structure $\mathcal{N}$ constructed in the proof of Theorem~\ref{thm:allap} the {\em canonical amalgam} of $\mathcal{P}$ and $\mathcal{Q}$ over $\mathcal{M}$.

\begin{definition} Let $\mathcal{M}, \mathcal{N}$ be continuous $\mathcal{L}$-pre-structures such that $\mathcal{M}$ is a substructure of $\mathcal{N}$. $\mathcal{N}$ is said to be a {\em finitely supported extension} of $\mathcal{M}$ if there is a finite subset $F\subseteq M$ such that, letting $\mathcal{F}$ be the substructure of $\mathcal{M}$ with domain $F$ and $\mathcal{A}$ be the substructure of $\mathcal{N}$ with domain $A=F\cup (N\setminus M)$, $\mathcal{N}$ is the canonical amalgam of $\mathcal{A}$ and $\mathcal{M}$ over $\mathcal{F}$. In this case the set $F$ is called a {\em finite support} of $\mathcal{N}$ over $\mathcal{M}$.
\end{definition}

We will be working with finitely supported, one-point extensions of continuous $\mathcal{L}$-pre-structures. In the following first step, we focus on one-point extensions.

\subsection{One-point extensions\label{subsec:4.1}}

In this subsection we study one-point extensions of continuous $\mathcal{L}$-pre-structures and their amalgams.

Fix a continuous $\mathcal{L}$-pre-structure $\mathcal{M}$. If $\mathcal{N}$ is a one-point extension of $\mathcal{M}$, we denote the unique element of $N\setminus M$ by $x_{\mathcal{N}}$. Conversely, if $x$ is the unique element of $N\setminus M$, we also denote $\mathcal{N}$ by $\mathcal{M}_x$.

If $\mathcal{M}_x$ and $\mathcal{M}_y$ are two one-point extensions of $\mathcal{M}$, consider the map $\varphi: M_x\to M_y$ with $\varphi(x)=y$ and $\varphi(z)=z$ for all $z\in M$. Define an equivalence relation $x\sim y$ if $\varphi$ is an isomorphism between $\mathcal{M}_x$ and $\mathcal{M}_y$ as continuous $\mathcal{L}$-pre-structures.

Let $\tilde{E}(\mathcal{M})$ be the collection of all $x_{\mathcal{N}}$ for one-point extensions $\mathcal{N}$ of $\mathcal{M}$.  Let $E'(\mathcal{M})$ be the quotient space $\tilde{E}(\mathcal{M})/\!\sim$. In $E'(\mathcal{M})$ we identify $\sim$-equivalent $x$ and $y$ and consider them the same object, and write $x=y$. Finally, let
$$ E(\mathcal{M})=M\cup E'(\mathcal{M}) $$
be the disjoint union of $M$ and $E'(\mathcal{M})$.

We define a continuous $\mathcal{L}$-pre-structure on $E(\mathcal{M})$ as follows. For any $n\geq 1$,  $x, y\in E(\mathcal{M})$ and $\overline{u}=(u_1,\dots, u_n)\in E(\mathcal{M})^n$, define $\overline{u}(y|x)=(v_1,\dots, v_n)\in E(\mathcal{M})^n$ by letting, for $1\leq i\leq n$,
$$ v_i=\left\{\begin{array}{ll} u_i & \mbox{ if $u_i\neq x$} \\ y & \mbox{ if $u_i=x$.} \end{array}\right. $$
Also, for $\overline{u}\in E(\mathcal{M})^n$, let
$$\|\overline{u}\|=\displaystyle\sum_{u_i\not\in M}K_{R,i}.$$
For $x, y\in E(\mathcal{M})$, define a metric $d^E(x,y)$ so that
$$ d^E(x, y)=\left\{\begin{array}{ll} d^M(x,y) & \mbox{ if $x, y\in M$} \\ d^{M_x}(x, y) & \mbox{ if $x\in E'(\mathcal{M})$ and $y\in M$} \\ d^{M_y}(x,y) & \mbox{ if $x\in M$ and $y\in E'(\mathcal{M})$.}\end{array}\right. $$
If $x, y\in E'(\mathcal{M})$, we define
$$\begin{array}{c} \mu(x,y) = \max\left\{\ \mathfrak{u}_{R,1}^*(|R^{M_x}(x)-R^{M_y}(y)|)\,:\, R\in\mathcal{L} \mbox{ is unary}\right\} \\ \\
\rho(x, y)=\sup\left\{\  \|\overline{u}\|^{-1}|R^{M_x}(\overline{u})-R^{M_y}(\overline{u}(y|x))|\ :\right. \ \ \ \ \ \ \ \ \ \ \ \ \ \ \ \\
\ \ \ \ \ \ \ \ \ \ \  \ \ \ \ \ \ \ \ \ \ \ \ \ \ \ \ \ \ \ \ \ \overline{u}\in M_x^n\setminus M^n, R\in {\mathcal{L}} \mbox{ is $n$-ary for $n\geq 2$}\bigr\}
\end{array}$$
$$ \lambda(x,y)=\sup\left\{\ |d^{M_x}(x, z)-d^{M_y}(y,z)|\,:\, z\in M\right\} $$
and
$$ d^E(x,y)=\max\left\{\,\mu(x,y),\ \rho(x,y),\ \lambda(x,y)\right\}. $$

\begin{lemma}\label{lem:dEM} $d^E$ is a metric on $E(\mathcal{M})$.
\end{lemma}

\begin{proof} We verify that if $d^E(x,y)=0$ then $x=y$. This is clear when at least one of $x, y$ is in $M$. Assume $x, y\in E'(\mathcal{M})$. Let $\varphi: M_x\to M_y$ be the map with $\varphi(x)=y$ and $\varphi(z)=z$ for all $z\in M$. If $d^E(x,y)=0$ then $\mu(x,y)=\rho(x,y)=\lambda(x,y)=0$, which implies that $\varphi$ is an isomorphism between $\mathcal{M}_x$ and $\mathcal{M}_y$ as continuous $\mathcal{L}$-pre-structures. Thus $x=y$.

Note that $\mu$ satisfies the triangle inequality because by Lemma~\ref{lem:u*} $\mathfrak{u}_{R,1}^*$ is subadditive for unary $R\in\mathcal{L}$. For all $x, y, z\in E'(\mathcal{M})$, $d(x, y)+d(y, z)\geq d(x, z)$, because all of $\mu$, $\rho$ and $\lambda$ satisfy the triangle inequality. Also, it is easy to verify that $\max\{\mu(x, y), \rho(x, y), \lambda(x, y)\}\leq \inf\left\{d^{M_x}(x, z)+d^{M_y}(y, z)\,:\, z\in M\right\}$, thus $d^E$ also satisfies the triangle inequality.
\end{proof}

Define $(X, d^X)=(E(\mathcal{M}), d^E)$ and for any $n$-ary $R\in\mathcal{L}$ define $R^X$ naturally on
$$ \dom(R^X)=\bigcup\{ M_x^n\,:\, x\in E'(\mathcal{M})\}. $$

\begin{lemma}\label{lem:preEM} $\mathcal{X}=(X, d^X, (R^X)_{R\in\mathcal{L}})$ is a partially defined continuous $\mathcal{L}$-pre-structure.
\end{lemma}

\begin{proof} We verify that for any $n$-ary $R\in\mathcal{L}$, $R^X$ is $1$-Lipschitz with respect to $d^X_R$.

Suppose first $n=1$. It suffices to show that for any $x, y\in E'(\mathcal{M})$,
$$ |R^X(x)-R^X(y)|\leq \mathfrak{u}_{R,1}(d^X(x,y)). $$
For this, note that by Lemma~\ref{lem:u*} (3), we have
$$\begin{array}{rcl} |R^X(x)-R^X(y)|&=&|R^{M_x}(x)-R^{M_y}(y)| \\&\leq& \mathfrak{u}_{R,1}(\mathfrak{u}_{R,1}^*(|R^{M_x}(x)-R^{M_y}(y)|)) \\
&\leq& \mathfrak{u}_{R,1}(\mu(x,y)) \\
&\leq& \mathfrak{u}_{R,1}(d^X(x,y)).
\end{array}
$$

Next suppose $n\geq 2$. The statement follows from Lemma~\ref{lem:1L} in all cases except when $\overline{u}\in M^n_x\setminus M^n, \overline{v}\in M_y^n\setminus M^n$ for distinct $x, y\in E'(\mathcal{M})$.

Suppose $\overline{u}=(u_1,\dots, u_n),\overline{v}=(v_1,\dots, v_n)$ are given as above. Let
$$\begin{array}{rcl}
S_0&=& \{ i\,:\, u_i, v_i\in M\} \\
S_1&=& \{i\,:\, u_i=x, v_i\in M\} \\
S_2&=& \{i\,:\, u_i\in M, v_i=y\} \\
S_3&=& \{i\,:\, u_i=x, v_i=y\}.
\end{array}
$$
Define $\overline{w}=(w_1,\dots, w_n)\in M_x^n$ by
$$ w_i=\left\{\begin{array}{ll}u_i & \mbox{ if $i\in S_2\cup S_3$}\\  v_i & \mbox{ if $i\in S_0\cup S_1$.} \end{array}\right. $$
Then
$$\begin{array}{rcl} & & |R^X(\overline{u})-R^X(\overline{v})| \\
&\leq & |R^X(\overline{u})-R^X(\overline{w})|+|R^X(\overline{w})-R^X(\overline{w}(y|x))|+|R^X(\overline{w}(y|x))-R^X(\overline{v})| \\
&\leq & d^X_R(\overline{u},\overline{w})+\|w\|\rho(x, y)+d^X_R(\overline{w}(y|x),\overline{v}) \\
&\leq & \displaystyle\sum_{i\in S_0\cup S_1}K_{R,i}d^X(u_i,v_i)+\sum_{i\in S_3}K_{R,i}d^X(u_i,v_i)+\sum_{i\in S_2} K_{R,i}d^X(u_i,v_i) \\
&=& d^X_R(\overline{u},\overline{v})
\end{array}
$$
as required.
\end{proof}

By Lemma~\ref{lem:gconservative} we obtain a continuous $\mathcal{L}$-pre-structure which is a conservative extension of $\mathcal{X}$. We denote this continuous $\mathcal{L}$-pre-structure as
$$ \mathcal{E}(\mathcal{M})=(E(\mathcal{M}), d^E, (R^E)_{R\in\mathcal{L}}). $$

From the proof of Lemma~\ref{lem:gconservative}, we have that for any $n$-ary $R\in\mathcal{L}$ and $\overline{x}\in E(\mathcal{M})^n$,
$$ R^E(\overline{x})= \max\left\{\,0,\ \sup\{R^{M_y}(\overline{u})-d^E_R(\overline{x},\overline{u})\,:\, y\in E'(\mathcal{M}), \overline{u}\in M_y^n\}\right\}.
$$

\subsection{Finitely-supported one-point extensions}

In this subsection we turn to finitely-supported one-point extensions.

Suppose $\mathcal{M}_x$ is a one-point extension of $\mathcal{M}$ with finite support $F\subseteq M$. Then $\mathcal{M}_x$ is a canonical amalgam of $\mathcal{M}$ and $\mathcal{F}_x$ over $\mathcal{F}$. This implies that for any $y\in M$,
$$ d^{M_x}(x, y)=\inf\{ d^{M_x}(x, z)+d^M(z, y)\,:\, z\in F\}, $$
and for any $n$-ary $R\in \mathcal{L}$ where $n\geq 2$ and $\overline{u}\in M_x^n$,
$$ R^{M_x}(\overline{u})=\max\left\{\,0,\ \sup\{ R^{M_x}(\overline{v})-d^{M_x}_R(\overline{u},\overline{v})\,:\, \overline{v}\in F_x^n\cup M^n\}\right\}. $$

We define
$$\begin{array}{rcl} E(\mathcal{M},\omega)=M\cup \{ x\in E'(\mathcal{M})&\!\!\!\!:\!\!\!\!&\mbox{ $\mathcal{M}_x$ is a finitely-supported} \\
& & \ \mbox{ one-point extension of $\mathcal{M}$}\,\}
\end{array} $$
and let $\mathcal{E}(\mathcal{M},\omega)$ be the substructure of $\mathcal{E}(\mathcal{M})$ with domain $E(\mathcal{M},\omega)$.

\begin{theorem}\label{prop:EMo} Let $\mathcal{L}$ be a proper continuous signature and $\mathcal{M}$ be a continuous $\mathcal{L}$-pre-structure. If $\mathcal{M}$ is separable, then so is $\mathcal{E}(\mathcal{M},\omega)$.
\end{theorem}

\begin{proof}  Let $N$ be the largest arity of all $R\in\mathcal{L}$. Let $K$ be the least of all $K_{R,i}$ for $n$-ary $R\in\mathcal{L}$ where $n\geq 2$ and $1\leq i\leq n$. Let $J$ be the largest value of all $K_{R,i}/K$ for $n$-ary $R\in\mathcal{L}$ where $n\geq 2$ and $1\leq i\leq n$.

Let $D$ be a countable dense subset of $M$. Let
$$ V=\{ d^M(x,y), R^M(\overline{u})\,:\, x, y\in D, \overline{u}\in D^n, R\in\mathcal{L} \mbox{ is $n$-ary}\} $$
and let $G$ be the additive subgroup of $\mathbb{R}$ generated by $V\cup\mathbb{Q}$. Then $G$ is a countable dense subset of $\mathbb{R}$.

We say that a continuous $\mathcal{L}$-pre-structure $\mathcal{A}$ is {\em $G$-valued} if
$$ \{d^A(x, y), R^A(\overline{u})\,:\, x, y\in A, \overline{u}\in A^n, R\in\mathcal{L} \mbox{ is $n$-ary}\}\subseteq G. $$

Let
$$\begin{array}{rcl} B&\!\!\!\!\!=\!\!\!\!\!&\{ x\in E'(\mathcal{M})\,:\, \mbox{ there is a finite support $F\subseteq D$ for $\mathcal{M}_x$} \\
& & \ \ \ \ \ \ \ \ \ \ \ \ \ \ \ \ \ \ \ \ \ \ \ \ \ \ \ \ \ \ \ \ \ \ \ \ \ \ \ \ \ \ \ \ \ \ \ \ \  \mbox{ such that $\mathcal{F}_x$ is $G$-valued}\}.
\end{array} $$
Then $B\subseteq E(\mathcal{M},\omega)\cap E'(\mathcal{M})$ is countable. It suffices to show that $B$ is dense in $E(\mathcal{M},\omega)\cap E'(\mathcal{M})$.

Let $x\in E(\mathcal{M},\omega)\cap E'(\mathcal{M})$ and $1>\epsilon>0$. Suppose $A=\{a_0, \dots, a_m\}\subseteq M$ is a finite support of $\mathcal{M}_x$ such that
$$ d^{M_x}(x, a_0)\geq d^{M_x}(x, a_1)\geq \cdots\geq d^{M_x}(x, a_m)>0. $$
Let
$$ \delta=\min\{ d^{M_x}(y,z)\,:\, y\neq z\in A_x=A\cup\{x\}\}>0. $$
Choose
$$ 0<\epsilon_0<\displaystyle\frac{\epsilon\min\{1,\delta\}\min\{1,K\}}{(9m+9)N}. $$
Then
$$ \epsilon_0<\min\left\{\displaystyle\frac{\min\{\epsilon, \delta\}}{9m+9}, \frac{\epsilon\delta}{(9m+9)N}, \frac{K\epsilon\delta}{4}\right\}. $$
Since $D$ is dense, we can find distinct $b_0,\dots, b_m\in D$ such that
$$ d^M(a_i, b_i)<\displaystyle\frac{\epsilon_0}{2NJ} $$
for all $0\leq i\leq m$. Let $F=\{b_0,\dots, b_m\}\subseteq D$.

Define $g: F\to  G$ by
$$ g(b_i)=d^{M_x}(x, a_i)+(3i+1)\epsilon_0+\frac{\epsilon}{2}+\epsilon_i'$$
for $0\leq i\leq m$, where we choose $\epsilon_i'\in (0,\epsilon_0)$ so that $g(b_i)\in G$. Then the following computations demonstrate that for all $0\leq i<j\leq m$,
$$ |g(b_i)-g(b_j)|\leq d^M(b_i, b_j)\leq g(b_i)+g(b_j). $$
Indeed, suppose $0\leq i<j\leq m$. We have
$$\begin{array}{rcl} g(b_j)-g(b_i)&=&
d^{M_x}(x, a_j)-d^{M_x}(x, a_i)+3(j-i)\epsilon_0+\epsilon_j'-\epsilon_i' \\
&<& (3m+1)\epsilon_0 \\
&=& 3(m+1)\epsilon_0-2\epsilon_0 \\
&<& \delta-(d^M(a_i, b_i)+d^M(a_j,b_j)) \\
&\leq & d^M(a_i, a_j)-d^M(a_i,b_i)-d^M(a_j, b_j) \\
&\leq & d^M(b_i,b_j),
\end{array}
$$
$$ \begin{array}{rcl} g(b_i)-g(b_j)&=& d^{M_x}(x, a_i)-d^{M_x}(x, a_j)-3(j-i)\epsilon_0+\epsilon_i'-\epsilon_j' \\
&\leq & d^{M_x}(x, a_i)-d^{M_x}(x, a_j)-2\epsilon_0 \\
&<& d^M(a_i,a_j)-d^M(a_i, b_i)-d^M(a_j, b_j) \\
&\leq& d^M(b_i,b_j)
\end{array}
$$
and
$$\begin{array}{rcl} d^M(b_i,b_j)& \leq & d^{M_x}(x, a_i)+d^{M_x}(x, a_j)+d^M(a_i, b_i)+d^M(a_j, b_j) \\
& <& d^{M_x}(x, a_i)+d^{M_x}(x, a_j)+2\epsilon_0 \\
&<& g(b_i)+g(b_j).
\end{array}
$$
We can thus define a one-point extension of the finite metric space $(F,d^F)$ by a point $u$ such that for all $0\leq i\leq m$, $d^{F_u}(u, b_i)=g(b_i)$. By defining
$$ d^{M_u}(u, y)=\inf\{ g(b_i)+d^M(b_i, y)\,:\, 0\leq i\leq m\} $$
for any $y\in M$, we extend the metric to $M_u=M\cup\{u\}$.

Next we define a continuous $\mathcal{L}$-structure $\mathcal{F}_u$ with domain $F_u=F\cup\{u\}$ by defining the values of $R^{F_u}$ for all $R\in\mathcal{L}$.

Suppose first $R\in\mathcal{L}$ is unary. Since $G$ is dense in $\mathbb{R}$, we may choose $R^{F_u}(u)\in G\cap [0,1]$ such that
$$ |R^{F_u}(u)-R^{M_x}(x)|<\mathfrak{u}_{R,1}\left(\frac{\epsilon}{2}\right). $$
We verify that (\ref{eqn:ucl}) holds for $R$ in $F_u$. In fact, for any $0\leq i\leq m$, if $g(b_i)\in I_{R,1}$, then by the superadditivity of $\mathfrak{u}_{R,1}$, we have
$$\begin{array}{rcl} & & |R^M(b_i)-R^{F_u}(u)| \\
&\leq & |R^M(b_i)-R^M(a_i)|+|R^M(a_i)-R^{M_x}(x)|+|R^{M_x}(x)-R^{F_u}(u)|  \\
&\leq & \mathfrak{u}_{R,1}(d^M(b_i,a_i))+\mathfrak{u}_{R,1}(d^{M_x}(x, a_i))+\mathfrak{u}_{R,1}(\frac{\epsilon}{2}) \\
&\leq& \mathfrak{u}_{R,1}(\epsilon_0)+\mathfrak{u}_{R,1}(d^{M_x}(x, a_i))+\mathfrak{u}_{R,1}(\frac{\epsilon}{2}) \\
&\leq& \mathfrak{u}_{R,1}(g(b_i))=\mathfrak{u}_{R,1}(d^{F_u}(u,b_i)).
\end{array}
$$
If $g(b_i)\not\in I_{R,1}$ then
$$ |R^M(b_i)-R^{F_u}(u)|\leq 1\leq \mathfrak{u}_{R,1}(g(b_i))=\mathfrak{u}_{R,1}(d^{F_u}(u,b_i)). $$

Next suppose $R\in\mathcal{L}$ is $n$-ary for $n\geq 2$. If $\overline{v}\in F^n$, then let $R^{F_u}(\overline{v})=R^M(\overline{v})$. Next we define $R^{F_u}(\overline{v})$ for $\overline{v}=(v_1,\dots, v_n)\in F_u^n\setminus F^n$. For this, let $\overline{v}'=(v_1',\dots, v_n')\in A_x^n\setminus A^n$ be defined by
$$ v_i'=\left\{\begin{array}{ll} a_j & \mbox{ if $v_i=b_j$} \\ x & \mbox{ if $v_i=u$.}\end{array}\right. $$
We define $R^{F_u}: F_u^n\setminus F^n\to G\cap [0,1]$ such that
\begin{enumerate}
\item[(a)] for all $\overline{v}\in F_u^n\setminus F^n$, $|R^{F_u}(\overline{v})-R^{A_x}(\overline{v}')|<K\displaystyle\frac{\epsilon}{2}$, and
\item[(b)] for all $\overline{v}, \overline{w}\in F_u^n\setminus F^n$,
$$ |R^{F_u}(\overline{v})-R^{F_u}(\overline{w})|\leq \left( 1-\displaystyle\frac{\epsilon_0}{\delta}\right) |R^{A_x}(\overline{v}')-R^{A_x}(\overline{w}')|. $$
\end{enumerate}

Note that our choice of $\epsilon_0$ guarantees that
$$ \displaystyle\frac{\epsilon_0}{\delta}<\frac{\epsilon}{(9m+9)N}\ll 1 \mbox{ and } \displaystyle\frac{2\epsilon_0}{\delta}<K\frac{\epsilon}{2}. $$
To define $R^{F_u}$ we enumerate all elements of $F_u^n\setminus F^n$ as
$$ \overline{v}^0, \overline{v}^1,\dots, \overline{v}^{\ell} $$
such that
$$ R^{A_x}({\overline{v}^0}')\leq R^{A_x}({\overline{v}^1}')\leq\cdots \leq R^{A_x}({\overline{v}^{\ell}}'). $$
If $R^{A_x}({\overline{v}^i}')=0$ we define $R^{F_u}(\overline{v})=0$. We also make the commitment that if $R^{A_x}({\overline{v}^i}')=R^{A_x}({\overline{v}^j}')$ then $R^{F_u}(\overline{v}^i)=R^{F_u}(\overline{v}^j)$. With this commitment we assume without loss of generality that
$$ 0<R^{A_x}({\overline{v}^0}')< R^{A_x}({\overline{v}^1}')<\cdots < R^{A_x}({\overline{v}^{\ell}}'). $$
We will then make our definition by induction on $i=0,\dots, \ell$ with the following inductive hypotheses:
\begin{enumerate}
\item[(H1)] $R^{F_u}(\overline{v}^i)\in G$ and
$$ \left(1-\displaystyle\frac{2\epsilon_0}{\delta}\right) R^{A_x}({\overline{v}^i}')< R^{F_u}(\overline{v}^i)<\left(1-\displaystyle\frac{\epsilon_0}{\delta}\right) R^{A_x}({\overline{v}^i}'); $$
\item[(H2)] for all $j<i$,
$$ R^{F_u}(\overline{v}^i)\leq R^{F_u}(\overline{v}^j)+\left(1-\displaystyle\frac{\epsilon_0}{\delta}\right)(R^{A_x}({\overline{v}^i}')-R^{A_x}({\overline{v}^j}')). $$
\end{enumerate}
For $i=0$ (H2) is vacuous and (H1) can be accomplished by the density of $G$. In general, suppose we have defined $R^{F_u}(\overline{v}^j)$ for $j\leq i$ to satisfy (H1) and (H2). Note that
$$  \left(1-\displaystyle\frac{2\epsilon_0}{\delta}\right) R^{A_x}({\overline{v}^{i+1}}')<R^{F_u}(\overline{v}^i)+\left(1-\displaystyle\frac{\epsilon_0}{\delta}\right)(R^{A_x}({\overline{v}^{i+1}}')-R^{A_x}({\overline{v}^i}')). $$
We can then define $R^{F_u}(\overline{v}^{i+1})\in G$ to be a value in between the above two quantities. To see that the inductive hypotheses are maintained, note that by this definition (H2) and the left half of (H1) are immediate. For the right half of (H1), just note that
$$ R^{F_u}(\overline{v}^i)+\left(1-\displaystyle\frac{\epsilon_0}{\delta}\right)(R^{A_x}({\overline{v}^{i+1}}')-R^{A_x}({\overline{v}^i}'))<\left(1-\displaystyle\frac{\epsilon_0}{\delta}\right) R^{A_x}({\overline{v}^{i+1}}'). $$
This finishes the definition of $R^{F_u}$. It is easy to see that (a) and (b) hold.

We verify (\ref{eqn:ucl}) for $R$ in $F_u$. For this, let $\overline{v},\overline{w}\in F_u^n$ differ at exactly one coordinate, say the $i$-th coordinate. If both $\overline{v},\overline{w}\in F^n$ then
$$ |R^{F_u}(\overline{v})-R^{F_u}(\overline{w})|=|R^M(\overline{v})-R^M(\overline{w})|\leq K_{R,i}d^M(v_i,w_i) $$
by the (\ref{eqn:ucl}) for $R$ in $M$. If both $\overline{v},\overline{w}\in F_u^n\setminus F^n$, then by (b) we have
$$\begin{array}{rcl} |R^{F_u}(\overline{v})-R^{F_u}(\overline{w})|&\leq& \left( 1-\displaystyle\frac{\epsilon_0}{\delta}\right) |R^{A_x}(\overline{v}')-R^{A_x}(\overline{w}')| \\
&\leq& \left( 1-\displaystyle\frac{\epsilon_0}{\delta}\right)K_{R,i}d^{A_x}(v_i', w_i') \\
&\leq& K_{R,i}d^{A_x}(v_i',w_i') -K_{R,i}\epsilon_0.
\end{array}
$$
Now if $v_i, w_i\in F$ then we have $v_i', w_i'\in A$, $d^{A_x}(v_i', w_i')=d^M(v_i', w_i')$,
$$ d^M(v_i', v_i), d^M(w_i', w_i)<\displaystyle\frac{\epsilon_0}{2NJ} \leq \frac{\epsilon_0}{2}$$
and
$$\begin{array}{rcl} |R^{F_u}(\overline{v})-R^{F_u}(\overline{w})|&\leq& K_{R,i}d^M(v_i',w_i')-K_{R,i}\epsilon_0 \\
& \leq &  K_{R,i}d^M(v_i', w_i')-K_{R,i}(d^M(v_i',v_i)+d^M(w_i',w_i)) \\
&\leq & K_{R,i}d^M(v_i, w_i).
\end{array}$$
If $v_i\in F$ and $w_i=u$, then $v_i'\in A$ and $w_i'=x$, and
$$ |R^{F_u}(\overline{v})-R^{F_u}(\overline{w})|\leq K_{R,i}d^{A_x}(v_i', x)-K_{R,i}\epsilon_0\leq K_{R,i}d^{F_u}(v_i, u). $$
This finishes the proof of (\ref{eqn:ucl}) for $R$ for the case $\overline{v},\overline{w}\in F_u^n\setminus F^n$. Finally, suppose $\overline{v}\in F_u^n\setminus F^n$ and $\overline{w}\in F^n$. Since $\overline{v}$ and $\overline{w}$ differ at only coordinate $i$, we have $v_i=u$ and $w_i\in F$.
We have
$$\begin{array}{rcl} & &  |R^{F_u}(\overline{v})-R^{F_u}(\overline{w})| \\
&\leq & |R^{F_u}(\overline{v})-R^{A_x}(\overline{v}')|+|R^{A_x}(\overline{v}')-R^{M}(\overline{w}')| +|R^{M}(\overline{w}')-R^M(\overline{w})| \\
&\leq & K\displaystyle\frac{\epsilon}{2}+d^{A_x}_R(\overline{v}', \overline{w}')+d^M_R(\overline{w}',\overline{w}) \\
&\leq & K\displaystyle\frac{\epsilon}{2}+K_{R,i}d^{A_x}(x, w_i')+\sum_{j=1}^n K_{R,j}d^M(w_j', w_j) \\
&\leq & K_{R,i}d^{A_x}(x, w_i')+K_{R,i}\displaystyle\frac{\epsilon}{2}+NJK_{R,i}\frac{\epsilon_0}{2NJ} \\
&< & K_{R,i}d^{F_u}(u, w_i).
\end{array}
$$
Now we have established (\ref{eqn:ucl}) for all $R\in\mathcal{L}$ in $F_u$, we get that $\mathcal{F}_u$ is indeed a one-point extension of $\mathcal{F}$ as a continuous $\mathcal{L}$-structure.

 By Theorem~\ref{thm:allap}, the canonical amalgam of $\mathcal{F}_u$ and $\mathcal{M}$ over $\mathcal{F}$ gives us a one-point extension $\mathcal{M}_u$ of $\mathcal{M}$. Hence $u\in E(\mathcal{M},\omega)\cap E'(\mathcal{M})$.

To complete the proof of the theorem, we claim that $d^E(x, u)<(NJ+1)\epsilon$. For this, we note that for any $y\in M$,
$$ d^{M_x}(x, y)=\inf\{d^{A_x}(x, a_i)+d^M(a_i, y)\,:\, 0\leq i\leq m\} $$
and
$$ d^{M_u}(u, y)=\inf\{d^{F_u}(u, b_i)+d^M(b_i, y)\,:\, 0\leq i\leq m\}. $$
Fix an arbitrary $y\in M$. Suppose first $d^{M_u}(u, y)< d^{M_x}(x,y)$, and let $d^{M_u}(u,y)=d^{F_u}(u, b_i)+d^M(b_i,y)$. Then
$$\begin{array}{rcl} & & d^{M_x}(x, y)-d^{M_u}(u,y) \\
&\leq & d^{A_x}(x, a_i)+d^M(a_i, y)-(d^{F_u}(u, b_i)+d^M(b_i, y)) \\
&\leq & d^{A_x}(x, a_i)-d^{F_u}(u, b_i)+d^M(a_i, b_i) \\
&\leq & \displaystyle\frac{\epsilon_0}{2NJ}-\displaystyle\frac{\epsilon}{2}<0,
\end{array}
$$
contradicting our assumption. Hence we must have $d^{M_u}(u, y)\geq d^{M_x}(x, y)$. Let $d^{M_x}(x,y)=d^{A_x}(x, a_i)+d^M(a_i, y)$. Then
$$\begin{array}{rcl} & & d^{M_u}(u, y)-d^{M_x}(x, y) \\
&\leq& d^{F_u}(u, b_i)+d^M(b_i,y)-(d^{A_x}(x, a_i)+d^M(a_i, y)) \\
&\leq & d^{F_u}(u,b_i)-d^{A_x}(x, a_i)+d^M(a_i, b_i) \\
&\leq & (3m+2)\epsilon_0+\displaystyle\frac{\epsilon}{2}+\frac{\epsilon_0}{2NJ}<\frac{5}{6}\epsilon.
\end{array}
$$
It follows that $\lambda(x, u)<\epsilon$.

For any unary $R\in\mathcal{L}$, we have
$$ \mathfrak{u}^*_{R,1}(|R^{M_x}(x)-R^{M_u}(u)|)\leq
\mathfrak{u}^*_{R,1}(\mathfrak{u}_{R,1}\left(\displaystyle\frac{\epsilon}{2}\right))=\frac{\epsilon}{2}. $$
Thus $\mu(x, u)<\epsilon$.

To compute $\rho(x,u)$, we fix an $n$-ary $R\in\mathcal{L}$ for $n\geq 2$ and $\overline{v}\in M_u^n\setminus M^n$. Since $\mathcal{M}_u$ is the canonical amalgam of $\mathcal{F}_u$ and $\mathcal{M}$ over $\mathcal{F}$, we have that
$$ R^{M_u}(\overline{v})=\max\{0, \sup\{R^{M_u}(\overline{z})-d^{M_u}_R(\overline{v},\overline{z})\,:\, \overline{z}\in F_u^n\cup M^n\}\}. $$
Similarly,
$$ R^{M_x}(\overline{v}(x|u))=\max\{0, \sup\{R^{M_x}(\overline{w})-d^{M_x}_R(\overline{v}(x|u), \overline{w})\,:\, \overline{w}\in A_x^n\cup M^n\}\}. $$
Suppose first $R^{M_u}(\overline{v})\geq R^{M_x}(\overline{v}(x|u))$. If $R^{M_u}(\overline{v})=0$ we must have $R^{M_x}(\overline{v}(x|u))=0$ and there is nothing to prove. Assume $R^{M_u}(\overline{v})>0$. Let $\eta>0$ be arbitrary. Then for some $\overline{z}\in F_u^n\cup M^n$,
$$ R^{M_u}(\overline{v})\leq R^{M_u}(\overline{z})-d^{M_u}_R(\overline{v},\overline{z})+\eta. $$
Consider first the case $\overline{z}\in M^n$. Then
$$\begin{array}{rcl} & &  R^{M_u}(\overline{v})-R^{M_x}(\overline{v}(x|u)) \\
&\leq &
R^{M_u}(\overline{z})-d^{M_u}_R(\overline{v},\overline{z})+\eta-(R^{M_x}(\overline{z})-d^{M_x}_R(\overline{v}(x|u), \overline{z})) \\
&\leq & \|\overline{v}\|\lambda(u,x)+\eta.
\end{array}
$$
Since $\eta$ is arbitrary, we have
$$ \|\overline{v}\|^{-1}|R^{M_u}(\overline{v})-R^{M_x}(\overline{v}(x|u))|<\epsilon. $$
Next consider the case $\overline{z}\in F_u^n\setminus F^n$. Then
$$\begin{array}{rcl} & &  R^{M_u}(\overline{v})-R^{M_x}(\overline{v}(x|u)) \\
&\leq & R^{M_u}(\overline{z})-d^{M_u}_R(\overline{v},\overline{z})+\eta-(R^{M_x}(\overline{z}')-d^{M_x}_R(\overline{v}(x|u), \overline{z}') \\
&= & R^{M_u}(\overline{z})-R^{M_x}(\overline{z}')+(d^{M_x}_R(\overline{v}(x|u), \overline{z}')-d^{M_u}_R(\overline{v}, \overline{z}))+\eta \\
& \leq & K\displaystyle\frac{\epsilon}{2}+NKJ \epsilon+\eta.
\end{array}
$$
The last inequality follows from the observation that there are four cases for the values of $v_i, z_i\in M_u$:
\begin{quote}
\begin{enumerate}
\item[Case 1.] $v_i=u$ and $z_i=u$. Then $z_i'=x$ and $v_i(x|u)=x$. In this case
$d^{M_x}(v_i(x|u), z_i')-d^{M_u}(v_i, z_i)=0$.
\item[Case 2.] $v_i\in M$ and $z_i=b_j\in F$. Then $z_i'=a_j\in A$ and $v_i(x|u)=v_i$. In this case
$$|d^{M_x}(v_i(x|u), z_i')-d^{M_u}(v_i,z_i)|\leq d^M(a_j, b_j)<\displaystyle\frac{\epsilon_0}{2NJ}<\epsilon. $$
\item[Case 3.] $v_i=u$ and $z_i=b_j\in F$. Then $z_i'=a_j\in A$ and $v_i(x|u)=x$. In this case
$$\begin{array}{rcl} |d^{M_x}(v_i(x|u),z_i')-d^{M_u}(v_i, z_i)|&=&|d^{M_x}(x, a_j)-d^{M_u}(u, b_j)| \\
&\leq& (3m+2)\epsilon_0+\displaystyle\frac{\epsilon}{2}<\epsilon.
\end{array} $$
\item[Case 4.] $v_i\in M$ and $z_i=u$. Then $z_i'=x$ and $v_i(x|u)=v_i$. In this case
$$\begin{array}{rcl} |d^{M_x}(v_i(x|u), z_i')-d^{M_u}(v_i,z_i)|&=&|d^{M_x}(v_i, x)-d^{M_u}(v_i, u)| \\
&\leq& \lambda(u,x)<\epsilon.
\end{array}$$
\end{enumerate}
\end{quote}
Since $\eta$ is arbitrary, we have
$$ R^{M_u}(\overline{v})-R^{M_x}(\overline{v}(x|u))\leq K(NJ+1)\epsilon. $$
Since $\|\overline{v}\|\geq K$, we have
$$ \|\overline{v}\|^{-1} |R^{M_u}(\overline{v})-R^{M_x}(\overline{v}(x|u))|\leq (NJ+1)\epsilon. $$
For $R^{M_u}(\overline{v})\leq R^{M_x}(\overline{v}(x|u))$ we get a similar estimate by a symmetric argument. Thus we have $\rho(x, u)<(NJ+1)\epsilon$ and $d^E(x,u)<(NJ+1)\epsilon$.
\end{proof}

\subsection{Construction of the Urysohn structure}

Given any proper continuous signature $\mathcal{L}$ and continuous $\mathcal{L}$-pre-structure $\mathcal{M}$, define by induction
$$ \begin{array}{c} \mathcal{M}_0=\mathcal{M} \ \ \ \ \ \   \\ \mathcal{M}_{n+1}=\mathcal{E}(\mathcal{M}_n,\omega)\end{array}$$
and let
$$ \mathcal{M}_\omega=\bigcup_{n\in\omega}\mathcal{M}_{n}. $$
Then by Theorem~\ref{prop:EMo}, if $\mathcal{M}$ is separable, then so is $\mathcal{M}_\omega$. From the construction of $\mathcal{M}_\omega$ it is easy to see that it always has the Urysohn property.

By the properness of $\mathcal{L}$, $\mathcal{M}_\omega$ has a unique completion $\bar{\mathcal{M}}_\omega$ which is a continuous $\mathcal{L}$-structure. We show that $\bar{\mathcal{M}}_\omega$ still has the Urysohn property. For this we first note the following corollary of the proof of Theorem~\ref{prop:EMo}.

\begin{lemma}\label{lem:EMU} Let $\mathcal{L}$ be a proper continuous signature and let $\mathcal{M}$ be a continuous $\mathcal{L}$-pre-structure. Suppose $\mathcal{M}$ has the Urysohn property. Given any finite substructure $\mathcal{A}$ of $\bar{\mathcal{M}}$, a one-point extension $\mathcal{A}_x$ of $\mathcal{A}$, and $\epsilon>0$, there is $u\in M$ such that
$$ |d^{A_x}(x, a)-d^{\bar{M}}(u,a)|<\epsilon \mbox{ for all $a\in A$,} $$
and for any $n$-ary $R\in\mathcal{L}$ and $\overline{y}\in A_x^n$,
$$ |R^{A_x}(\overline{y})-R^{\bar{M}}(\overline{y}(u|x))|<\epsilon. $$
\end{lemma}

\begin{proof} Let $\mathcal{A}_x$ and $\epsilon>0$ be given. Let $N$, $K$, $J$ be defined as in the proof of Theorem~\ref{prop:EMo}. Suppose $A=\{a_0,\dots, a_m\}\subseteq \bar{M}$. Since $M$ is dense in $\bar{M}$, we get a finite subset $F=\{b_0,\dots, b_m\}\subseteq M$ such that
$$d^{\bar{M}}(a_i, b_i)<\displaystyle\frac{\epsilon}{2NKJ} $$
for all $0\leq i\leq m$. By the proof of Theorem~\ref{prop:EMo}, we obtain an one-point extension $\mathcal{F}_u$ of $\mathcal{F}$ such that
$$ |d^{A_x}(x, a)-d^{F_u}(u, b)|<\displaystyle\frac{\epsilon}{2} $$
for all $a\in A$ where $b=b_j$ if $a=a_j$ for $0\leq j\leq m$, and for any $n$-ary $R\in\mathcal{L}$ and $\overline{y}\in A^n_x$,
$$ |R^{A_x}(\overline{y})-R^{F_u}(\overline{y}')|<\displaystyle\frac{\epsilon}{2}, $$
where $\overline{y}'$ is defined as in the proof of Theorem~\ref{prop:EMo}.

Since $M$ has the Urysohn property, we may find such a $u\in M$. Now for any $a\in A$, letting $b\in F$ be defined as above, we have
$$\begin{array}{rcl} & & |d^{A_x}(x, a)-d^{\bar{M}}(u,a)| \\&\leq& |d^{A_x}(x, a)-d^{F_u}(u, b)|+|d^{F_u}(u, b)-d^{\bar{M}}(u, a)| \\
&<& \displaystyle\frac{\epsilon}{2}+\frac{\epsilon}{2}=\epsilon.
\end{array}$$

For $n$-ary $R\in\mathcal{L}$ and $\overline{y}\in A_x^n$,
$$\begin{array}{rcl} & & |R^{A_x}(\overline{y})-R^{\bar{M}}(\overline{y}(u|x))| \\
&\leq & |R^{A_x}(\overline{y})-R^{F_u}(\overline{y}')|+|R^{F_u}(\overline{y}')-R^{\bar{M}}(\overline{y}(u|x))| \\
&\leq & \displaystyle\frac{\epsilon}{2}+d_R^{\bar{M}}(\overline{y}',\overline{y}(u|x)) \\
&\leq & \displaystyle\frac{\epsilon}{2}+NKJ\frac{\epsilon}{2NKJ}=\epsilon.
\end{array}
$$
\end{proof}

The following lemma is also obvious.

\begin{lemma}\label{lem:dEe} Let $\mathcal{M}$ be a continuous $\mathcal{L}$-pre-structure, and $\mathcal{M}_x$ and $\mathcal{M}_y$ be two nonisomorphic one-point extensions of $\mathcal{M}$. Then $\mathcal{M}$ has an extension $\mathcal{M}_{xy}$ with $M_{xy}=M\cup\{x, y\}$ such that
$$d^{M_{xy}}(x, y)=\max\{\mu(x,y), \rho(x,y), \lambda(x,y)\}=d^E(x,y).$$
\end{lemma}

\begin{proof} Take $\mathcal{M}_{xy}$ to be the substructure of $\mathcal{E}(\mathcal{M})$ with domain $M_{xy}$.
\end{proof}

\begin{theorem}\label{prop:UM} Let $\mathcal{L}$ be a proper continuous signature and let $\mathcal{M}$ be a continuous $\mathcal{L}$-pre-structure. If $\mathcal{M}$ has the Urysohn property, then so does its completion $\bar{\mathcal{M}}$.
\end{theorem}

\begin{proof} Let $A\subseteq \bar{M}$ and $\mathcal{A}_x$ be a one-point extension of $\mathcal{A}$. Let
$$ \delta=\min\{ d^{A_x}(a, x)\,:\, a\in A\}>0. $$
We define a sequence $(y_m)$ in $M$ by induction on $m$. By Lemma~\ref{lem:EMU} there is $y_0\in M$ such that
$$ |d^{A_x}(x,a)-d^{\bar{M}}(y_0, a)|<\delta \mbox{ for all $a\in A$,} $$
for any unary $R\in\mathcal{L}$,
$$ |R^{A_x}(x)-R^{\bar{M}}(y_0)|<\mathfrak{u}_{R,1}(\delta), $$
and for any $n$-ary $R\in\mathcal{L}$ and $\overline{z}\in A^n_x$,
$$ |R^{A_x}(\overline{z})-R^{\bar{M}}(\overline{z}(y_0|x))|<\delta K, $$
where $K$ is defnied as in the proof of Theorem~\ref{prop:EMo}.
Let $B_0=A\cup\{y_0\}\subseteq \bar{M}$. By Lemma~\ref{lem:dEe} we obtain a one-point extension $(\mathcal{B}_0)_x=\mathcal{A}_{xy_0}$ such that
$$ d^{(B_0)_x}(x, y_0)=\max\{\mu(x, y_0), \rho(x,y_0),\lambda(x,y_0)\}<\delta. $$
In general, assume $y_m$ has been defined to satisfy the following inductive hypothesis: for
$$ B_m=A\cup \{y_0,\dots, y_m\}\subseteq \bar{M} $$
there is an one-point extension $(\mathcal{B}_m)_x$ of $\mathcal{B}_m$ (obtained from Lemma~\ref{lem:dEe}) such that
$$ d^{(B_m)_x}(x, y_m)=\max\{\mu(x, y_m), \rho(x,y_m), \lambda(x,y_m)\}<\displaystyle\frac{\delta}{2^{m}}. $$ By Lemma~\ref{lem:EMU} there is $y_{m+1}\in M$ such that
$$ |d^{(B_m)_x}(x, b)-d^{\bar{M}}(y_{m+1},b)|<\displaystyle\frac{\delta}{2^{m+1}} \mbox{ for all $b\in B_m$,} $$
for any unary $R\in\mathcal{L}$,
$$ |R^{(B_m)_x}(x)-R^{\bar{M}}(y_{m+1})|<\mathfrak{u}_{R,1}\left(\displaystyle\frac{\delta}{2^{m+1}}\right), $$
and for any $n$-ary $R\in\mathcal{L}$ and $\overline{z}\in (B_m)_x^n$,
$$ |R^{(B_m)_x}(\overline{z})-R^{\bar{M}}(\overline{z}(y_{m+1}|x))|<\displaystyle\frac{\delta K}{2^{m+1}}. $$
Now let
$$ B_{m+1}=B_m\cup\{y_{m+1}\}=A\cup \{y_0, \dots, y_m, y_{m+1}\}\subseteq \bar{M} $$
and $(\mathcal{B}_{m+1})_x$ be given by Lemma~\ref{lem:dEe} such that
$$ d^{(B_{m+1})_x}(x, y_{m+1})=\max\{\mu(x, y_{m+1}), \rho(x, y_{m+1}), \lambda(x, y_{m+1})\}<\displaystyle\frac{\delta}{2^{m+1}}. $$
The inductive step of the definition is complete, and the inductive hypothesis is maintained. Note that
$$\begin{array}{rcl} d^M(y_m, y_{m+1})&\leq & d^{(B_m)_x}(x, y_{m})+d^{(B_{m+1})_x}(x, y_{m+1}) \\
&< &\displaystyle\frac{\delta}{2^m}+\frac{\delta}{2^{m+1}}<\frac{\delta}{2^{m-1}}.
\end{array}
$$
Thus $(y_m)$ is a $d^M$-Cauchy sequence in $M$. Let $y=\lim_m y_m\in\bar{M}$. Then for any $a\in A$, since
$$ |d^{A_x}(x, a)- d^{\bar{M}}(y_m, a)|<\displaystyle\frac{\delta}{2^m}, $$
letting $m\to\infty$, we get that
$$ d^{A_x}(x,a)=d^{\bar{M}}(y, a). $$
For any unary $R\in \mathcal{L}$, we have
$$ |R^{A_x}(x)-R^{\bar{M}}(y_m)|<\mathfrak{u}_{R,1}
\left(\displaystyle\frac{\delta}{2^{m}}\right).$$
Since $\mathfrak{u}$ is continuous at $0$, by letting $m\to\infty$ we get
$$ R^{A_x}(x)=R^{\bar{M}}(y). $$
For any $n$-ary $\mathcal{R}\in\mathcal{L}$ and $\overline{z}\in A_x^n$,
$$ |R^{A_x}(\overline{z})-R^{\bar{M}}(\overline{z}(y_m|x))|<\displaystyle\frac{\delta K}{2^m}. $$
Letting $m\to \infty$, we obtain
$$ R^{A_x}(\overline{z})=R^{\bar{M}}(\overline{z}(y|x)). $$
This shows that $\mathcal{A}_x$ and $\mathcal{A}_y$ are isomorphic, and in particular there is an isomorphic embedding $\varphi: \mathcal{A}_x\to \mathcal{\bar{M}}$ with $\varphi(a)=a$ for all $a\in A$.
\end{proof}

Thus we have shown that for any proper continuous signature $\mathcal{L}$ and continuous $\mathcal{L}$-pre-structure $\mathcal{M}$, $\bar{\mathcal{M}}_\omega$ is a continuous $\mathcal{L}$-structure with the Urysohn property. Since any two separable continuous $\mathcal{L}$-structures with the Urysohn property are isomorphic, by the above procedure we obtain this unique separable continuous $\mathcal{L}$-structure as $\bar{\mathcal{M}}_\omega$ for any separable continuous $\mathcal{L}$-pre-structure $\mathcal{M}$.

We denote this unique separable continuous $\mathcal{L}$-structure with the Urysohn property as $\mathbb{U}_{\mathcal{L}}$.

\subsection{Properties of the Urysohn structures}

In the preceding subsection we showed one direction of Theorem~\ref{thm:U} (i), i.e., if $\mathcal{L}$ is a proper continuous signature, then there exists a (separable) continuous $\mathcal{L}$-structure with the Urysohn property. In the following we prove the rest of Theorem~\ref{thm:U}.

We first show that if $\mathcal{L}$ is semiproper, then for any continuous $\mathcal{L}$-pre-structure $\mathcal{M}$, $\mathcal{M}_\omega$ is a continuous $\mathcal{L}$-pre-structure with the Urysohn property. Note that $\mathcal{E}(\mathcal{M}, \omega)$ is not necessarily complete or even has a completion. We also do not need to address the separability of $\mathcal{M}_\omega$. The proof requires only an observation that the upper semicontinuity is not used in the construction of $\mathcal{M}_\omega$.

\begin{theorem}\label{thm:Uiihalf} Let $\mathcal{L}$ be a semiproper continuous signature and $\mathcal{M}$ be a continuous $\mathcal{L}$-pre-structure. Then $\mathcal{M}_\omega$ is a continuous $\mathcal{L}$-pre-structure with the Urysohn property.
\end{theorem}

\begin{proof}
We first claim that for any continuous $\mathcal{L}$-pre-structure $\mathcal{M}$ and finite subsets $F\subseteq G$ of $M$, if the one-point extension $\mathcal{M}_x$ has support $F$, then it also has support $G$. It is easy to verify that for any $y\in M$, 
$$d^{M_x}(x, y)=\min\{d^{M_x}(x, z)+d^{M}(z, y)\,:\, z\in F\}=\min\{d^{M_x}(x, z)+d^M(z, y)\,:\,z\in G\}.$$
If $R\in\mathcal{L}$ is $n$-ary, where $n\geq 2$, and $\bar{u}\in M^n_x$, then
\begin{align*}
R^{M_x}(\bar{u})&=\max\{0, \sup\{R^{M_x}(\bar{v})-d^{M_x}_R(\bar{u}, \bar{v}): \bar{v}\in M^n\cup F^n_x\}\}\\
&\leq\max\{0, \sup\{R^{M_x}(\bar{v})-d^{M_x}_R(\bar{u}, \bar{v}): \bar{v}\in M^n\cup G^n_x\}\};
\end{align*}
On the other hand, if $\bar{v}\in M^n\cup G^n_x$, 
\begin{align*}
R^{M_x}(\bar{v})-d^{M_x}_R(\bar{u}, \bar{v})&\leq\sup\{R^{M_x}(\bar{w})-d^{M_x}_R(\bar{v}, \bar{w})-d^{M_x}_R(\bar{u}, \bar{v}): \bar{w}\in M^n\cup F^n_x\}\\
&\leq\sup\{R^{M_x}(\bar{w})-d^{M_x}_R(\bar{u}, \bar{w}): \bar{w}\in M^n\cup F^n_x\}.
\end{align*}
So we have $$R^{M_x}(\bar{u})=\max\{0, \sup\{R^{M_x}(\bar{v})-d^{M_x}_R(\bar{u}, \bar{v}): \bar{v}\in M^n\cup G^n_x\}\}$$ 
as required.

To complete the proof, it suffices to show that $\mathcal{E}(\mathcal{M}, \omega)$ is a continuous $\mathcal{L}$-pre-structure, since $\mathcal{M}_\omega$ would be easily defined from iterating the definition of $\mathcal{E}(\mathcal{M},\omega)$, and $\mathcal{M}_\omega$ would have the Urysohn property by the construction. For this, we only observe that Lemmas~\ref{lem:dEM} and \ref{lem:preEM} can be proved for $\mathcal{E}(\mathcal{M},\omega)$ without using upper semicontinuity. For Lemma~\ref{lem:dEM}, only note that the use of Lemma~\ref{lem:u*} to obtain subadditivity of $\mathfrak{u}^*_{R,1}$ does not require upper semicontinuity of $\mathfrak{u}_{R,1}$. For Lemma~\ref{lem:preEM}, the use of Lemma~\ref{lem:u*} (3) in its proof does require upper semicontinuity. Below we provide an alternative proof of the relevant part for $\mathcal{E}(\mathcal{M},\omega)$ without using upper semicontinuity. 

Let $R\in\mathcal{L}$ be unary and $x,y\in E'(\mathcal{M},\omega)$, we need to show
$$ |R^{M_x}(x)-R^{M_y}(y)|\leq \mathfrak{u}_{R,1}(\inf\{d^{M_x}(x,z)+d^{M_y}(z,y)\,:\, z\in M\}). $$
Suppose $\mathcal{M}_x$ has finite support $F\subseteq M$ and $\mathcal{M}_y$ has finite support $G\subseteq M$. Then by the above claim both $\mathcal{M}_x$ and $\mathcal{M}_y$ have finite support $H=F\cup G$. It follows that
$$ \inf\{d^{M_x}(x, z)+d^{M_y}(z, y)\,:\, z\in M\}=\min\{d^{H_x}(x, z)+d^{H_y}(z, y)\,:\, z\in H\}.$$ By the superadditivity of $\mathfrak{u}_{R,1}$, we have that for any $z\in H$,
$$\begin{array}{rcl}
|R^{M_x}(x)-R^{M_y}(y)|&=& |R^{H_x}(x)-R^{H_y}(y)| \\
&\leq & |R^{H_x}(x)-R^H(z)|+|R^H(z)-R^{H_y}(y)| \\
&\leq& \mathfrak{u}_{R,1}(d^{H_x}(x,z))+\mathfrak{u}_{R,1}(d^{H_y}(z,y)) \\
&\leq& \mathfrak{u}_{R,1}(d^{H_x}(x,z)+d^{H_y}(z,y)).
\end{array}
$$
The required inequality follows from the monotonicity of $\mathfrak{u}_{R,1}$ because $H$ is finite.
\end{proof}

\begin{theorem}\label{thm:Uii} Let $\mathcal{L}$ be a continuous signature with only finitely many relation symbols, where all the associated moduli of continuity are nondecreasing. 
\begin{enumerate}
\item If there exists a continuous $\mathcal{L}$-pre-structure $\mathcal{U}$ with the Urysohn property, then $\mathcal{L}$ is semiproper.
\item If there exists a (separable) continuous $\mathcal{L}$-structure $\mathcal{U}$ with the Urysohn property, then $\mathcal{L}$ is proper.
\end{enumerate}
\end{theorem}

\begin{proof} (1) By Theorem~\ref{thm:finap} it suffices to show that the AP holds for $\Kfin$. Let $\mathcal{M}$, $\mathcal{P}$, $\mathcal{Q}$ be finite continuous $\mathcal{L}$-structures with isomorphic embeddings $\varphi: M\to P$ and $\psi: M\to Q$. By the Urysohn property of $\mathcal{U}$, we can find a substructure $\mathcal{P}'$ of $\mathcal{U}$ that is isomorphic to $\mathcal{P}$. Let $i:P\to P'$ be an isomorphism. Let $M'=i\circ \varphi(M)\subseteq U$. Then $i\circ \varphi$ is an isomorphism between $\mathcal{M}$ and $\mathcal{M}'$ as a substructure of $\mathcal{U}$ (or $\mathcal{P}'$). Let $e:\psi(M)\to M'$ be $i\circ\varphi\circ \psi^{-1}$. Then $e$ is an isomorphism, and we have $i\circ \varphi=e\circ\psi$.

By the Urysohn property of $\mathcal{U}$ we can find an extension $\mathcal{Q}'$ of $\mathcal{M}'$ with $Q'\subseteq U$ such that $\mathcal{Q}'$ is isomorphic to $\mathcal{Q}$, and letting $j: Q\to Q'$ be the isomorphism, we have $j\rest \psi(M)=e$. Now let $N=P'\cup Q'\subseteq U$. Then $\mathcal{N}$ is a substructure of $U$, $i: P\to N$ is an isomorphic embedding from $\mathcal{P}$ into $\mathcal{N}$, and $j: P\to N$ is an isomorphic embedding from $\mathcal{Q}$ into $\mathcal{N}$. By our construction, $i\circ \varphi=j\circ \psi$. Thus $\mathcal{N}$ is an amalgam of $\mathcal{P}$ and $\mathcal{Q}$ over $\mathcal{M}$.

(2) Assume that $\mathcal{U}$ is a continuous $\mathcal{L}$-structure with the Urysohn property. By (1), $\mathcal{L}$ is semiproper. It remains to show that given any unary $R\in\mathcal{L}$, $\mathfrak{u}_{R, 1}$ is upper semicontinuous on $I_{R, 1}$. Let $I^\circ_{R, 1}$ be the interior of $I_{R, i}$ and assume $a\in I^\circ_{R, i}$. Since $\mathcal{U}$ has the Urysohn property, there is a countable substructure $\mathcal{M}$ of $\mathcal{U}$ where 
\begin{itemize}
\item $M=\{x_n\}_{n\geq 1}\cup\{y\}$,
\item for all $m, n\geq 1$, $d^M(x_n,y)=a+2^{-n}$, $d^M(x_m, x_n)=|2^{-m}-2^{-n}|$, 
\item $R^M(y)=\lim_{\delta\rightarrow 0^+}\mathfrak{u}_{R, 1}(a+\delta)<1$, $R^M(x_n)=0$ for all $n\geq 1$, and
\item for all other $R'\in\mathcal{L}$, ${R'}^M$ is identically $0$.
\end{itemize}
Then by the completeness of $\mathcal{U}$, there is $z\in\mathcal{U}$ such that $d^U(y, z)=a$ and $R^U(z)=0$, which implies that 
$$
\lim_{\delta\rightarrow 0^+}\mathfrak{u}_{R, 1}(a+\delta)=|R(y)-R(z)|\leq\mathfrak{u}_{R, 1}(a).
$$
\end{proof}


In case $\mathcal{L}$ is proper, one naturally wonders whether the metric space underlying $\mathbb{U}_{\mathcal{L}}$ is isometric to the universal Urysohn metric space $\mathbb{U}$. The answer is not always.

\begin{definition} A continuous signature $\mathcal{L}$ is {\em Lipschitz} if $\mathcal{L}$ consists of only finitely many relation symbols and for each $n$-ary $R\in\mathcal{L}$ and $1\leq i\leq n$,
$I_{R,i}$ is bounded and there is $K_{R,i}>0$ such that $\mathfrak{u}_{R,i}(r)=K_{R,i}r$ for all $r\in I_{r,i}$.
\end{definition}

Any Lipschitz continuous signature is proper. If $\mathcal{L}$ does not contain any unary relation symbols, then $\mathcal{L}$ is proper iff it is Lipschitz.

\begin{theorem}\label{thm:LUU} Let $\mathcal{L}$ be a proper continuous signature. Then the following are equivalent:
\begin{enumerate}
\item[(i)] $\mathcal{L}$ is Lipschitz.
\item[(ii)] The metric space underlying $\mathbb{U}_{\mathcal{L}}$ is isometric to $\mathbb{U}$.
\end{enumerate}
\end{theorem}

\begin{proof} For (i)$\Rightarrow$(ii) let $A\subseteq \mathbb{U}_{\mathcal{L}}$ be finite and $A_x$ be a metric space which is a one-point extension of $A$ as a metric space. It suffices to show that $x$ can be realized as a point in $\mathbb{U}_{\mathcal{L}}$. For this we only need to make a continuous $\mathcal{L}$-pre-structure $\mathcal{A}_x$ as a one-point extension of the substructure $\mathcal{A}$. Observe that for any $n$-ary $R\in\mathcal{L}$, since $\mathcal{L}$ is Lipschitz, we have that for any metric space $(X, d^X)$ and $\overline{u}=(u_1,\dots, u_n), \overline{v}=(v_1,\dots, v_n)\in X^n$,
$$ d^X_R(\overline{u},\overline{v})=\displaystyle\sum_{i=1}^n K_id^X(u_i, v_i). $$
Now for any $n$-ary $R\in\mathcal{L}$, $R^A$ is defined on $A^n$ and is $1$-Lipschitz with respect to $d^A_R$ by Lemma~\ref{lem:1L}. It follows that $R^{A_x}$ is partially defined on $\dom(R^{A_x})=A^n$ and is $1$-Lipschitz with respect to $d^{A_x}_R$. Thus by Lemma~\ref{lem:gconservative} we can define $R^{A_x}$ on $A_x^n$ so that $\mathcal{A}_x$ is an extension of $\mathcal{A}$. By the Urysohn property of $\mathbb{U}_{\mathcal{L}}$, $x$ can now be realized as a point in $\mathbb{U}_{\mathcal{L}}$ as desired.

For (ii)$\Rightarrow$(i) assume $\mathcal{L}$ is proper but not Lipschitz. Then for some unary $R\in\mathcal{L}$ we have $r_1, r_2>0$ with $r_1+r_2\in I_{R,1}$ and
$$ \mathfrak{u}_{R,1}(r_1+r_2)>\mathfrak{u}_{R,1}(r_1)+\mathfrak{u}_{R,1}(r_2). $$
Now consider the continuous $\mathcal{L}$-structure $\mathcal{M}$ defined as follows.
$$M=\{a, b\},\ d^M(a, b)=r_1+r_2,\ R^M(a)=0,\ R^M(b)=\mathfrak{u}_{R,1}(r_1+r_2),$$ and for any other $R'\in\mathcal{L}$, ${R'}^M$ is identically $0$. Let $\mathcal{A}$ be a substructure of $\mathbb{U}_{\mathcal{L}}$ that is isomorphic to $\mathcal{M}$. Suppose $A=\{x, y\}\subseteq \mathbb{U}_{\mathcal{L}}$ where $d^U(x,y)=r_1+r_2$, $R^U(x)=0$ and $R^U(y)=\mathfrak{u}_{R,1}(r_1+r_2)$. Now consider a metric one-point extension of $(A, d^A)$ with an extra point $z$ such that
$$ d^{A_z}(z, x)=r_1 \mbox{ and } d^{A_z}(z, y)=r_2. $$
Assume toward a contradiction that $\mathbb{U}_{\mathcal{L}}$ is isometric to $\mathbb{U}$. Then there is such a $z\in \mathbb{U}_{\mathcal{L}}$. Then
$$\begin{array}{rcl} \mathfrak{u}_{R,1}(r_1+r_2)=|R^U(x)-R^U(y)|&\leq& |R^U(x)-R^U(z)|+|R^U(z)-R^U(y)| \\
&\leq&
\mathfrak{u}_{R,1}(r_1)+\mathfrak{u}_{R,2}(r_2),
\end{array}$$
contradicting our assumption.
\end{proof}

When $\mathcal{L}$ is proper but not Lipschitz, we can decompose $\mathbb{U}_{\mathcal{L}}$ into continuum many substructures each of which has an isometric copy of $\mathbb{U}$ as its underlying metric space. The parameter space is $[0,1]^k$ for some finite $k$. In fact, let $R_1,\dots, R_k$ be all the unary relation symbols in $\mathcal{L}$ whose associated moduli of continuity is not Lipschitz. For each $\overline{p}=(p_1,\dots, p_k)\in [0,1]^k$, let
$$ \mathbb{U}_{\mathcal{L},\overline{p}}=\{x\in \mathbb{U}_{\mathcal{L}}\,:\, R_i^U(x)=p_i \mbox{ for all } 1\leq i\leq p\}. $$
Then for each $\overline{p}\in [0,1]^k$, the metric space underlying each $\mathbb{U}_{\mathcal{L}, \overline{p}}$ is isometric to $\mathbb{U}$.


Let $\mathcal{L}$ be a proper continuous signature. For any separable continuous $\mathcal{L}$-structure $\mathcal{M}$, equip $\Aut(\mathcal{M})$ with the pointwise convergence topology. Then $\Aut(\mathcal{M})$ becomes a Polish group.

It follows easily from Uspenskij's method \cite{Uspenskij} and our construction that $\Aut(\mathbb{U}_{\mathcal{L}})$ is a universal Polish group.

\section{Fra\"iss\'e Classes of Finite Continuous Structures}

In this section we consider some classes of finite continuous $\mathcal{L}$-structures for semiproper $\mathcal{L}$ and show that they are Fra\"iss\'e classes. In particular, we construct the rational Urysohn $\mathcal{L}$-pre-structure $\mathbb{QU}_{\mathcal{L}}$ and show that its completion is isomorphic to $\mathbb{U}_{\mathcal{L}}$ when $\mathcal{L}$ is proper.

\begin{definition}\label{def:goodpair} Let $\mathcal{L}$ be a continuous signature. Let $\Delta\subseteq \mathbb{R}^+$ and $V\subseteq [0,1]$.
\begin{enumerate}
\item[(i)] We call $\Delta$ a {\em distance value set} if for all $a, b\in \Delta$,
    $$ \min\{a+b, \sup(\Delta)\}\in \Delta. $$
\item[(ii)] We call the pair $(\Delta, V)$ a {\em good value pair} for $\mathcal{L}$ if
    \begin{itemize}
    \item $\Delta$ is a distance value set,
    \item if $\Delta$ is bounded, then for any $n$-ary $R\in\mathcal{L}$ and $1\leq i\leq n$, $\mathfrak{u}_{R,i}(\sup(\Delta))\geq 1$,
    \item $0\in V$, and
    \item for any $v\in V$, any $\delta\in \Delta$, any $n$-ary $R\in\mathcal{L}$ and $1\leq i\leq n$, if $v>\mathfrak{u}_{R,i}(\delta)$, then $v-\mathfrak{u}_{R,i}(\delta)\in V$.
    \end{itemize}
\item[(iii)] We say that $(\Delta, V)$ is {\em finite} ({\em countable}, respectively) if both $\Delta$ and $V$ are finite (countable, respectively).
\end{enumerate}
\end{definition}

\begin{definition}\label{def:DV} Let $\mathcal{L}$ be a continuous signature and $\mathcal{M}$ a continuous $\mathcal{L}$-pre-structure. Let $(\Delta, V)$ be a good value pair for $\mathcal{L}$.
\begin{enumerate}
\item[(i)] We say that $\mathcal{M}$ is {\em $(\Delta, V)$-valued} if for all $x, y\in M$, $d^M(x,y)\in \Delta$ and for all $n$-ary $R\in\mathcal{L}$ and $\overline{x}\in M^n$, $R^M(\overline{x})\in V$.
\item[(ii)] Let $\mathcal{K}_{(\Delta, V)}$ be the class of all finite $(\Delta, V)$-valued continuous $\mathcal{L}$-structures.
\end{enumerate}
\end{definition}

\begin{lemma}\label{lem:goodvaluepair} Let $\mathcal{L}$ be a semiproper continuous signature.
 \begin{enumerate}
 \item[(i)] For any finite continuous $\mathcal{L}$-structure $\mathcal{M}$, there exists a finite good value pair $(\Delta, V)$ for $\mathcal{L}$ such that $\mathcal{M}$ is $(\Delta, V)$-valued.
 \item[(ii)] For any finite (countable, respectively) distance value set $\Delta$ and finite (countable, respectively) $W\subseteq [0,1]$, there exists a finite (countable, respectively) $V$ such that $W\subseteq V$ and $(\Delta, V)$ is a good value pair for $\mathcal{L}$.
     \end{enumerate}
\end{lemma}

\begin{proof} (i) Let
$$ P=\{d^A(x,y)\,:\, x\neq y\in M\}. $$
Let $\delta$ be sufficiently large such that $\delta\geq \sup(P)$ and for any $n$-ary $R\in\mathcal{L}$ and $1\leq i\leq n$, $\delta\geq \sup(I_{R,i})$. Let $\Delta$ be the smallest distance value set such that $P\subseteq \Delta$ and $\sup(\Delta)=\delta\in \Delta$. Since $P$ is finite, so is $\Delta$.

Let
$$ Q=\{\mathfrak{u}_{R,i}(\delta)\,:\, R\in\mathcal{L} \mbox{ is $n$-ary}, 1\leq i\le n, \delta\in \Delta\}, $$
$$ W=\{R^M(\overline{x})\,:\, R\in\mathcal{L} \mbox{ is $n$-ary}, \overline{x}\in M^n\}, $$
and
$$ V=\{0\}\cup W\cup ([0,1]\cap \left\{w-\sum_{i=1}^{\ell}q_i\,:\, w\in W,\ q_1,\dots, q_\ell\in Q\right\}). $$
Since $\mathcal{L}$ is semiproper, $Q$, $W$ and $V$ are all finite. It is clear that $(\Delta, V)$ is a good value pair for $\mathcal{L}$ and $\mathcal{M}$ is $(\Delta, V)$-valued.

(ii) is proved similarly.
\end{proof}

\begin{definition}\label{def:Fraisee} Let $\mathcal{L}$ be a continuous signature and $\mathcal{K}$ be a set of continuous $\mathcal{L}$-pre-structures.
\begin{enumerate}
\item[(i)] $\mathcal{K}$ has the {\em hereditary property} (HP for short) if for any continuous $\mathcal{L}$-pre-structures $\mathcal{M}$ and $\mathcal{N}$, if $\mathcal{M}$ is a substructure of $\mathcal{N}$ and $\mathcal{N}\in\mathcal{K}$, then $\mathcal{M}\in\mathcal{K}$.
\item[(ii)] $\mathcal{K}$ has the {\em joint embedding property} (JEP for short) if for any $\mathcal{M}, \mathcal{N}\in \mathcal{K}$ there exist $\mathcal{P}\in\mathcal{K}$ and isomorphic embeddings $\varphi:\mathcal{M}\to\mathcal{P}$ and $\psi:\mathcal{N}\to\mathcal{P}$.
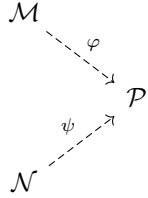
\begin{center}
\begin{tikzcd}
& \mathcal{M}\arrow[dr, dashed, "\varphi"] & \\
 & & \mathcal{P} \\
& \mathcal{N}\arrow[ur, dashed, "\psi"] &
\end{tikzcd}
\end{center}
\item[(iii)] $\mathcal{K}$ is a {\em Fra\"iss\'e class} if $\mathcal{K}$ has the HP, JEP, and AP.
\end{enumerate}
\end{definition}

Our main theorem of the section is the following.

\begin{theorem}\label{thm:Fraisee} Let $\mathcal{L}$ be a semiproper continuous signature and let $(\Delta, V)$ be a countable good value pair for $\mathcal{L}$. Then the class $\mathcal{K}_{(\Delta, V)}$ is a countable Fra\"iss\'e class.
\end{theorem}

\begin{proof} It is obvious that $\mathcal{K}_{(\Delta, V)}$ is countable. The HP for $\mathcal{K}_{(\Delta, V)}$ is obvious. For the JEP, suppose $\mathcal{M}$ and $\mathcal{N}$ are finite $(\Delta, V)$-valued continuous $\mathcal{L}$-structures. Let $X$ be the disjoint union of $M$ and $N$. Let $\delta\in \Delta$ be such that $\delta\geq\diam(M), \diam(N)$ and for any $n$-ary $R\in\mathcal{L}$ and $1\leq i\leq n$, $\delta\geq \sup(I_{R,i})$. Note that such $\delta$ exists since $(\Delta, V)$ is a good value pair for $\mathcal{L}$. Then define a metric $d^X$ on $X$ by
$$ d^X(x, y)=\left\{\begin{array}{ll} d^{M}(x, y) & \mbox{ if $x, y\in M$} \\
d^N(x,y) & \mbox{ if $x, y\in N$} \\ \delta & \mbox{ otherwise.}
\end{array}\right.
$$
For every $R\in \mathcal{L}$, $R^X$ is naturally defined on
$$\dom(R^X)=M^n\cup N^n $$
and takes values in $V$. It is clear that $\mathcal{X}=(X, d^X, (R^X)_{R\in \mathcal{L}})$ is a partially defined continuous $\mathcal{L}$-pre-structure. Then by Lemma~\ref{lem:gconservative}, $\mathcal{X}$ has a conservative extension $\mathcal{P}$. It is easy to see that $\mathcal{M}$ and $\mathcal{N}$ embed into $\mathcal{P}$ as substructures.

The proof of Theorem~\ref{thm:finap} gives the AP. Only note that for the partially defined structure $\mathcal{X}$ defined in that proof, $d^X$ can be made to take values in $\Delta$ and for any $R\in\mathcal{L}$, $R^X$ takes values in $V$. Then in the application of the proof of Lemma~\ref{lem:gconservative}, for the amalgam $\mathcal{P}$ and any $R\in\mathcal{L}$, $R^P$ takes values in $V$. Thus the resulting amalgam $\mathcal{P}$ is an element of $\mathcal{K}_{(\Delta, V)}$.
\end{proof}

By a standard argument there exists a Fra\"iss\'e limit of the class $\mathcal{K}_{(\Delta, V)}$ when $\mathcal{L}$ is semiproper and $(\Delta, V)$ a countable good value pair for $\mathcal{L}$. We denote it as $\mathbb{U}_{(\Delta, V)}$.

$\mathbb{U}_{(\Delta, V)}$ is the unique countable continuous $\mathcal{L}$-pre-structure that is universal for all finite $(\Delta, V)$-valued continuous $\mathcal{L}$-structures and is ultrahomogeneous. It is also universal for all countable $(\Delta, V)$-valued continuous $\mathcal{L}$-pre-structures.

It is well known that for any countable distance value set $\Delta$, the class of all finite metric spaces whose distance takes values in $\Delta$ (known as $\Delta$-metric spaces) form a Fra\"iss\'e class (see e.g. \cite{EGLMM}), and the Fra\"iss\'e limit is denoted $\mathbb{U}_{\Delta}$. As usual, $\mathbb{U}_{\Delta}$ is characterized by the property that it is universal for all finite $\Delta$-metric spaces and it is ultrahomogenous.

By an argument identical to the proof of Theorem~\ref{thm:LUU} (i)$\Rightarrow$(ii), we have that if $\mathcal{L}$ is a Lipschitz continuous signature and $(\Delta, V)$ is a countable good value pair for $\mathcal{L}$, then the metric space underlying $\mathbb{U}_{(\Delta, V)}$ is isometric to $\mathbb{U}_{\Delta}$.

When $\mathcal{L}$ is semiproper and $(\Delta, V)$ is a countable good value pair for $\mathcal{L}$, we consider the automorphism group of $\mathbb{U}_{(\Delta, V)}$ and denote it as $\Aut(\mathbb{U}_{(\Delta, V)})$. Since $\mathbb{U}_{(\Delta, V)}$ is countable, we may regard $\Aut(\mathbb{U}_{(\Delta, V)})$ as a subgroup of the infinite permutation group $S_\infty$. Under the usual topology of $S_\infty$, it is a closed subgroup, hence is itself a Polish group. From the above discussions, we know that $\Aut(\mathbb{U}_{(\Delta, V)})$ is universal among all $\Aut(\mathcal{M})$ for $(\Delta,V)$-valued continuous $\mathcal{L}$-pre-structures $\mathcal{M}$.

When $\Delta=\mathbb{Q}^+$, we know by Lemma~\ref{lem:goodvaluepair} (ii) that there is a countable $V$ such that $\mathbb{Q}\cap [0,1]\subseteq V$ and $(\Delta, V)$ is a good value pair for $\mathcal{L}$. Moreover, there is a smallest such $V$, which we fix. We denote $\mathbb{U}_{(\Delta, V)}$ also by $\mathbb{QU}_{\mathcal{L}}$ following the convention for the analogous classical structure. We also call $\mathbb{QU}_{\mathcal{L}}$ the {\em rational Urysohn continuous $\mathcal{L}$-pre-structure}.

In the following we show that if $\mathcal{L}$ is proper, then the completion of $\mathbb{QU}_{\mathcal{L}}$ is isomorphic to the unique separable Urysohn continuous $\mathcal{L}$-structure $\mathbb{U}_{\mathcal{L}}$.

\begin{theorem}\label{thm:QUL} Let $\mathcal{L}$ be a proper continuous signature. Then the completion of $\mathbb{QU}_{\mathcal{L}}$ is isomorphic to $\mathbb{U}_{\mathcal{L}}$.
\end{theorem}

\begin{proof} Let $\Delta=\mathbb{Q}^+$ and let $V$ be the smallest such that $\mathbb{Q}\cap[0,1]\subseteq V$ and $(\Delta, V)$ is a good value pair for $\mathcal{L}$.

$\mathbb{QU}_{\mathcal{L}}$ obviously has the following {\em rational Urysohn property}: given any finite $(\Delta, V)$-valued continuous $\mathcal{L}$-structure $\mathcal{M}$, a one-point extension $\mathcal{N}$ of $\mathcal{M}$ that is $(\Delta, V)$-valued, and an isomorphic embedding $\varphi$ from $\mathcal{M}$ into $\mathbb{QU}_{\mathcal{L}}$, there is an isomorphic embedding $\psi$ from $\mathcal{N}$ into $\mathbb{QU}_{\mathcal{L}}$ such that $\psi\,\rest M=\varphi$.

Using the rational Urysohn property to replace the Urysohn property, the proofs of Theorem~\ref{prop:EMo} and Lemma~\ref{lem:EMU} can be repeated so that the following statement holds: given any finite substructure $\mathcal{A}$ of $\bar{\mathbb{QU}}_{\mathcal{L}}=\bar{\mathcal{M}}$, a one-point extension $\mathcal{A}_x$ of $\mathcal{A}$, and $\epsilon>0$, there is $u\in \mathbb{QU}_{\mathcal{L}}$ such that
$$ |d^{A_x}(x, a)-d^{\bar{M}}(u,a)|<\epsilon\ \ \mbox{ for all $a\in A$,} $$
and for any $n$-ary $R\in\mathcal{L}$ and $\overline{y}\in A_x^n$,
$$ |R^{A_x}(\overline{y})-R^{\bar{M}}(\overline{y}(u|x))|<\epsilon. $$

Then by repeating the proof of Theorem~\ref{prop:UM}, again replacing the Urysohn property by the rational Urysohn property in the assumption, we get the Urysohn property for the completion of $\mathbb{QU}_{\mathcal{L}}$. This shows that the completion of $\mathbb{QU}_{\mathcal{L}}$ is isomorphic to $\mathbb{U}_{\mathcal{L}}$.
\end{proof}

If $\mathcal{L}$ is proper, by Theorem~\ref{thm:QUL} every element of $\Aut(\mathbb{QU}_{\mathcal{L}})$ extends uniquely to an element of $\Aut(\mathbb{U}_{\mathcal{L}})$. Let $\eta: \Aut(\mathbb{QU}_{\mathcal{L}}) \to \Aut(\mathbb{U}_{\mathcal{L}})$ be this canonical extension map. Then it is easy to see that $\eta$ is a group isomorphic embedding.

Before closing this section we show that $\eta(\Aut(\mathbb{QU}_{\mathcal{L}}))$ is dense in $\Aut(\mathbb{U}_{\mathcal{L}})$. For this we need to fix some notation. For any continuous $\mathcal{L}$-pre-structure $\mathcal{M}$ with one-point extensions $\mathcal{M}_x$ and $\mathcal{M}_y$, let
$$ d^E_{\mathcal{M}}(x,y)=\max\{\,\mu(x,y),\,\rho(x,y),\,\lambda(x,y)\,\} $$
be defined as in Subsection~\ref{subsec:4.1}. Let $A>\max\{NJ,1\}$, where $N$ and $J$ are defined as in the proof of Theorem~\ref{prop:EMo}.

\begin{lemma}\label{lem:QUapprox} Let $\mathcal{M}=\mathbb{QU}_{\mathcal{L}}$, $\mathcal{A}$ be a finite substructure of $\mathcal{M}$, $\mathcal{A}_x$ and $\mathcal{A}_y$ be one-point extensions of $\mathcal{A}$, and $\epsilon>0$. Suppose $y\in M$, $d^E_{\mathcal{A}}(x,y)<\epsilon$, and there is $x_0\in M$ such that $d^E_{\mathcal{A}}(x, x_0)=0$. Then there is some $x'\in M$ such that $d^E_{\mathcal{A}}(x, x')=0$ and $d^{M}(x',y)<\epsilon$.
\end{lemma}

\begin{proof} Let $\Delta=\mathbb{Q}^+$ and let $V$ be the smallest such that $\mathbb{Q}\cap[0,1]\subseteq V$ and $(\Delta, V)$ is a good value pair for $\mathcal{L}$.

By Lemma~\ref{lem:dEe} there is an amalgam $\mathcal{A}_{xy}$ of both $\mathcal{A}_x$ and $\mathcal{A}_y$ such that $A_{xy}=A\cup\{x, y\}$ and
$d^{A_{xy}}(x,y)=d^E_{\mathcal{A}}(x,y)$. We define a finite continuous $\mathcal{L}$-structure $\mathcal{F}$ where $\mathcal{F}$ and $\mathcal{A}_{xy}$ agree on everything except that $d^F(x,y)$ is a rational number such that both
$$ d^F(x,y)\leq \inf \{d^{A_x}(x, z)+d^{A_y}(y,z)\,:\, z\in A\}$$
and
$$d^{A_{xy}}(x,y)\leq d^F(x,y)<\epsilon. $$

It follows from our assumptions that $\mathcal{F}$ is a finite $(\Delta, V)$-valued continuous $\mathcal{L}$-structure that is an extension of $\mathcal{A}_y$. By the rational Urysohn property of $\mathcal{M}$, we obtain a point $x'\in M$ such that $d^E_{\mathcal{A}}(x, x')=0$ and $d^M(x',y)=d^F(x, y)<\epsilon$.
\end{proof}

\begin{theorem}\label{thm:AutQUdense} $\eta(\Aut(\mathbb{QU}_{\mathcal{L}}))$ is dense in $\Aut(\mathbb{U}_{\mathcal{L}})$.
\end{theorem}

\begin{proof} Let $\mathcal{M}=\mathbb{QU}_{\mathcal{L}}$ and $\bar{\mathcal{M}}=\mathbb{U}_{\mathcal{L}}$. Let $g\in \Aut(\bar{\mathcal{M}})$, $x_1,\dots, x_n\in \bar{\mathcal{M}}$ and $\epsilon>0$. We will define $y_1,\dots, y_n, z_1,\dots, z_n\in M$ such that the map $y_i\mapsto z_i$, $1\leq i\leq n$, is a partial isomorphism of $\mathcal{M}$, and for all $1\leq i\leq n$,
$$ d^{\bar{M}}(x_i,y_i)<\epsilon $$
and
$$ d^{\bar{M}}(g(x_i), z_i)<2(1+iA^i)\epsilon. $$

First, let $y_1,\dots, y_n\in M$ be such that $d^{\bar{M}}(x_i,y_i)<\epsilon$ for all $1\leq i\leq n$. Let $u_1, \dots, u_n\in M$ be such that $d^E_{\varnothing}(g(y_i), u_i)<\epsilon$ and $d^{\bar{M}}(g(y_i), u_i)<\epsilon$. We define $z_1,\dots, z_n\in M$ by induction on $1\leq i\leq n$ with the following inductive hypothesis: letting $A_{i-1}=\{z_1,\dots, z_{i-1}\}$, the map $y_j\mapsto z_j$, $1\leq j\leq i-1$, is a partial isomorphism of $\mathcal{M}$ and $d^{\bar{M}}(u_j, z_j)<2jA^j\epsilon$ for all $1\leq j\leq i-1$. Let $F_{i-1}=\{y_1,\dots, y_{i-1}\}$. Then by our inductive hypothesis $\mathcal{A}_{i-1}$ and $\mathcal{F}_{i-1}$ are isomorphic. Let $\pi_{i-1}: F_{i-1}\to A_{i-1}$ be the isomorphism given by $\pi_{i-1}(y_j)=z_j$ for $1\leq j\leq i-1$. Consider the one-point extension of $\mathcal{F}_{i-1}$ by $y_i$. Via $\pi_{i-1}$ this gives a one-point extension of $\mathcal{A}_{i-1}$ by a point which we denote by $v_i$. Consider also the one-point extension of $\mathcal{A}_{i-1}$ by $u_i$. In the following calculation we denote $\mathcal{A}_{i-1}$ by $\mathcal{A}$ and $\mathcal{F}_{i-1}$ by $\mathcal{F}$ for notational simplicity.

For any unary $R\in \mathcal{L}$, we have
$$\begin{array}{rcl} \mathfrak{u}^*_{R,1}(|R^{A_{u_i}}(u_i)-R^{A_{v_i}}(v_i)|)
&=&
\mathfrak{u}^*_{R,1}(|R^{A_{u_i}}(u_i)-R^{F_{y_i}}(y_i)|) \\
&=&
\mathfrak{u}^*_{R,1}(|R^{A_{u_i}}(u_i)-R^{g(F)_{g(y_i)}}(g(y_i))|) \\
&\leq& d^E_{\varnothing}(u_i,g(y_i))<\epsilon.
\end{array}
$$
Thus $\mu(u_i,v_i)<\epsilon$.

Let $1\leq j\leq i-1$ be such that $\lambda(u_i, v_i)=|d^{A_{u_i}}(u_i, z_j)-d^{A_{v_i}}(v_i,z_j)|$. Then
$$\begin{array}{rcl} \lambda(u_i, v_i)&=&|d^{A_{u_i}}(u_i, z_j)-d^{A_{v_i}}(v_i,z_j)| \\
&=& |d^{A_{u_i}}(u_i, z_j)-d^{F_{y_i}}(y_i,y_j)| \\
&=& |d^{\bar{M}}(u_i, z_j)-d^{\bar{M}}(y_i,y_j)| \\
&=& |d^{\bar{M}}(u_i,z_j)-d^{\bar{M}}(g(y_i), g(y_j))| \\
&\leq & d^{\bar{M}}(u_i, g(y_i))+d^{\bar{M}}(g(y_j),z_j) \\
&\leq & d^{\bar{M}}(u_i, g(y_i))+d^{\bar{M}}(g(y_j),u_j)+d^{\bar{M}}(u_j, z_j) \\
&\leq & \epsilon+\epsilon+2jA^j\epsilon \leq 2iA^i\epsilon.
\end{array}
$$

Finally, let $R\in\mathcal{L}$ be $n$-ary where $n\geq 2$, and let $\overline{w}\in A_{u_i}^n\setminus A^n$. Define $\overline{w}'\in g(F)_{g(y_i)}^n$ by
$$ w'_k=\left\{\begin{array}{ll} g(y_j) & \mbox{ if $w_k=z_j$ for some $1\leq j\leq i-1$} \\
g(y_i) & \mbox{ if $w_k=u_i$.}
\end{array}\right.
$$
We have
$$\begin{array}{rcl} & & \|\overline{w}\|^{-1}\left|R^{A_{u_i}}(\overline{w})-R^{A_{v_i}}(\overline{w}(v_i|u_i))\right| \\ &=& \|\overline{w}\|^{-1}\left|R^{A_{u_i}}(\overline{w})-R^{g(F)_{g(y_i)}}(\overline{w}')\right| \\
&\leq& A(d^{\bar{M}}(u_i,g(y_i))+\max\{d^{\bar{M}}(z_j, g(y_j))\,:\, 1\leq j\leq i-1\}) \\
&\leq& A(d^{\bar{M}}(u_i,g(y_i))+\max\{d^{\bar{M}}(z_j, u_j)+ d^{\bar{M}}(u_j, g(y_j))\,:\, 1\leq j\leq i-1\}) \\
&\leq & A(\epsilon+2(i-1)A^{i-1}\epsilon+\epsilon)\leq 2iA^i\epsilon.
\end{array}
$$
This implies that $\rho(u_i,v_i)<2iA^i\epsilon$.

Combining the above computations, we get $d^E_{\mathcal{A}}(u_i,v_i)<2iAi\epsilon$. By Lemma~\ref{lem:QUapprox}, there is some $z_i\in M$ such that the map $y_j\mapsto z_j$, $1\leq j\leq i$, is a partial isomorphism of $\mathcal{M}$ and $d^M(u_j, z_j)<2iA^i\epsilon$. This completes the induction.

We note that for any $1\leq i\leq n$,
$$ d^{\bar{M}}(g(x_i),z_i)\leq d^{\bar{M}}(g(x_i), g(y_i))+d^{\bar{M}}(g(y_i), u_i)+d^{\bar{M}}(u_i, z_i)\leq 2(1+iA^i)\epsilon $$
as desired.
\end{proof}

\section{Coherent EPPA for Continuous Structures}
In this section we continue to work with a semiproper continuous signature $\mathcal{L}$. We prove that the class of all finite continuous $\mathcal{L}$-structures has the coherent EPPA as defined by Hrushovski \cite{Hru} and Siniora--Solecki \cite{Siniora}.

\begin{theorem}\label{thm:EPPA} Let $\mathcal{L}$ be a semiproper continuous signature and let $\mathcal{M}$ be a finite continuous $\mathcal{L}$-structure. There is a finite continuous $\mathcal{L}$-structure $\mathcal{N}$ such that
\begin{enumerate}
\item[(i)] $\mathcal{N}$ is an extension of $\mathcal{M}$;
\item[(ii)] there is a map $\phi$ from the set $\Part(\mathcal{M})$ of all partial automorphism of $\mathcal{M}$ to the set $\Aut(\mathcal{N})$ of all automorphism of $\mathcal{N}$ such that for any $\pi\in \Part(\mathcal{M})$, $\phi(\pi)$ extends $\pi$;
\item[(iii)] for the map $\phi$ in {\rm (ii)} we have that for all $\pi, \tau\in\Part(\mathcal{M})$, if $\range(\tau)=\dom(\pi)$ then
$$\phi(\pi\circ \tau)=\phi(\pi)\circ \phi(\tau). $$
\end{enumerate}
\end{theorem}

\begin{proof} Let
$$ P=\{ d^M(x, y)\,:\, x\neq y\in M\} $$
and
$$ V=\{ R^M(\overline{x})\,:\, \overline{x}\in M^n, R\in\mathcal{L} \mbox{ is $n$-ary}\}. $$
We introduce a new binary relation symbol $D_p$ for each $p\in P$ and an $n$-ary relation symbol $R_{v}$ for each $n$-ary $R\in\mathcal{L}$ and $v\in V$. Let $\mathcal{L}^*$ be the classical signature consisting of all $D_p$ and $R_v$ for $p\in P$, $R\in\mathcal{L}$ and $v\in V$. Let $\mathcal{T}$ be the collection of all finite classical $\mathcal{L}^*$-structures $\mathcal{S}$ satisfying one of the following conditions:
\begin{enumerate}
\item[(a)] for some $m\geq 1$ and $p, p_1, \dots, p_m\in P$ such that
$$ \displaystyle\sum_{i=1}^m p_i<p, $$
$\mathcal{S}$ has at most $m+1$ elements and there are $z_0, z_1, \dots, z_m\in S$ such that
$$ \mathcal{S}\models D_{p}(z_0, z_m)\wedge D_{p}(z_m, z_0)\wedge \bigwedge_{1\leq i\leq m} (D_{p_i}(z_{i-1}, z_i)\wedge D_{p_i}(z_i,z_{i-1})); $$
or
\item[(b)] for some $n$-ary $R\in\mathcal{L}$, $v, v'\in V$, $m_1,\dots, m_n\geq 1$ and $$p_{1,1}, \dots, p_{1, m_1}, p_{2, 1},\dots, p_{2, m_2},\dots, p_{n, 1}, \dots, p_{n, m_n}\in P$$ such that
$$ \displaystyle\sum_{i=1}^n\sum_{j=1}^{m_i}\mathfrak{u}_{R,i}(p_{i,j})<|v-v'|, $$
$\mathcal{S}$ has at most $\sum_{i=1}^n(m_i+1)=n+\sum_{i=1}^nm_i$ elements and there are
$$ z_{1, 0}, \dots, z_{1, m_1}, z_{2,0}, \dots, z_{2, m_2}, \dots, z_{n, 0}, \dots, z_{n, m_n}\in S$$ such that
$$ \mathcal{S}\models \bigwedge_{1\leq i\leq n}\bigwedge_{1\leq j\leq m_i} (D_{p_{i,j}}(z_{i,j-1}, z_{i,j})\wedge D_{p_{i,j}}(z_{i,j}, z_{i,j-1})) $$
and
$$\mathcal{S}\models R_{v}(z_{1,0}, z_{2,0},\dots, z_{n, 0}) \wedge R_{v'}(z_{1, m_1}, z_{2, m_2}, \dots, z_{n, m_n}); $$
or
\item[(c)] for some $n$-ary $R\in\mathcal{L}$ and $v\neq v'\in V$, $\mathcal{S}$ has at most $n$ elements and there are $x_1, \dots, x_n\in S$ such that
$$ \mathcal{S}\models R_v(x_1,\dots, x_n)\wedge R_{v'}(x_1,\dots, x_n). $$
\end{enumerate}
Since $D$ and $V$ are finite, $\mathcal{T}$ is also finite.

For any continuous $\mathcal{L}$-structure $\mathcal{A}$ we obtain a classical $\mathcal{L}^*$-structure $\mathcal{A}^*$ in the obvious way: for $x, y\in A^*$ and $p\in P$, define
$$ \mathcal{A^*}\models D_p(x, y) \iff d^A(x,  y)=p, $$
and for any $n$-ary $R\in\mathcal{L}$, $v\in V$ and $\overline{x}\in ({A}^*)^n$, define
$$ \mathcal{A^*}\models R_v(\overline{x})\iff R^{U}(\overline{x})=v. $$
By Lemma~\ref{lem:goodvaluepair} (i) there is a finite good value pair $(\Delta, W)$ for $\mathcal{L}$ such that $\mathcal{M}$ is $(\Delta, W)$-valued. Moreover,  $P\subseteq \Delta$ and $V\subseteq W$. By Theorem~\ref{thm:Fraisee} the class $\mathcal{K}_{(\Delta, W)}$ is a countable Fra\"iss\'e class and thus has a Fra\"iss\'e limit $\mathbb{U}_{(\Delta, W)}$. Then $\mathbb{U}_{(\Delta, W)}^*$ is a classical $\mathcal{L}^*$-structure which extends $\mathcal{M}^*$ such that any partial automorphism of $\mathcal{M}^*$ extends to an automorphism of $\mathbb{U}_{(\Delta, W)}^*$. Moreover, $\mathbb{U}_{(\Delta, W)}^*$ is $\mathcal{T}$-free under homomorphisms, i.e., there is no $\mathcal{S}\in\mathcal{T}$ and homomorphism from $\mathcal{S}$ into $\mathbb{U}_{(\Delta, W)}^*$.

By Theorem~1.11 of \cite{Siniora} there is a finite $\mathcal{L}^*$-structure $\mathcal{X}$ such that $\mathcal{X}$ is an extension of $\mathcal{M}^*$ and there is a map $\psi: \Part(\mathcal{M}^*)\to \Aut(\mathcal{X})$ such that for any $\pi\in \Part(\mathcal{M}^*)$, $\psi(\pi)$ extends $\pi$, and for all $\pi, \tau\in\Part(\mathcal{M}^*)$, if $\dom(\pi)=\range(\tau)$, then
$$\psi(\pi\circ \tau)=\psi(\pi)\circ \psi(\tau). $$
Define an equivalence relation $\sim$ on $X$ by $x\sim y$ iff there are $x=z_0, \dots, z_m=y\in X$ and $p_1,\dots, p_m\in P$ such that
$$ \mathcal{X}\models \bigwedge_{1\leq i\leq m} (D_{p_i}(z_{i-1}, z_i)\wedge D_{p_i}(z_i, z_{i-1})). $$
Clearly $M$ is contained in one of the $\sim$-equivalence classes. We denote this class by $Y$. Let $\mathcal{Y}$ be the substructure of $\mathcal{X}$ with domain $Y$. We define a metric $d^Y$ on $Y$ as follows. For any $x, y\in Y$, define $d^Y(x, y)=0$ if $x=y$, and
$$\begin{array}{rcl} d^Y(x, y)&=&\inf\left\{ \displaystyle\sum_{i=1}^m p_i\right. \,:\, p_1,\dots, p_m\in P, \exists z_0=x, z_1,\dots, z_m=y\in Y \\
& & \ \ \ \ \mbox{ such that } \left.\mathcal{X}\models \displaystyle\bigwedge_{1\leq i\leq m} (D_{p_i}(z_{i-1}, z_i)\wedge D_{p_i}(z_i,z_{i-1}))\right\}
\end{array}$$
if $x\neq y$. Since $\mathcal{Y}$ is $\mathcal{T}$-free, it follows that $d^Y$ is a metric on $Y$ and that $(Y, d^Y)$ is an extension of $(M, d^M)$ as a metric space.

For each nonempty $\pi\in \Part(\mathcal{M}^*)$, $\psi(\pi)\in \Aut(\mathcal{X})$ must fix $Y$ setwise, thus $\psi(\pi)\rest Y$ is an automorphism of $\mathcal{Y}$ still extending $\pi$, which implies that $\psi(\pi)\rest Y$ is an isometry of $(Y, d^Y)$ extending $\pi$ as a partial isometry of $(M, d^M)$.

We continue to define a partially defined continuous $\mathcal{L}$-pre-structure on $Y$. For any $n$-ary $R\in \mathcal{L}$, $v\in V$, and $\overline{x}\in Y^n$, let $R^Y(\overline{x})=v$ iff $\mathcal{X}\models R_v(\overline{x})$. Thus
$$ \dom(R^Y)=\{\overline{x}\in Y^n\,:\, \exists v\in V\ \mathcal{X}\models R_v(\overline{x})\}. $$
Since $\mathcal{Y}$ is $\mathcal{T}$-free, it follows that for any $n$-ary $R\in\mathcal{L}$ and $\overline{x}, \overline{y}\in \dom(R^Y)$, we have $R^Y(\overline{x}), R^Y(\overline{y})\in V$, and
$$ |R^Y(\overline{x})-R^Y(\overline{y})|\leq d^Y_R(\overline{x}, \overline{y}). $$
Hence $(Y, d^Y, (R^Y)_{R\in\mathcal{L}})$ is a partially defined continuous $\mathcal{L}$-pre-structure.

Let $\mathcal{N}$ be a continuous $\mathcal{L}$-pre-structure that is a conservative extension of $$(Y, d^Y, (R^Y)_{R\in\mathcal{L}})$$ given by Lemma~\ref{lem:gconservative}. Then $\mathcal{N}$ is an extension of $\mathcal{M}$. For any $\pi\in\Part(\mathcal{M})$, let $\phi(\pi)=\psi(\pi)\rest Y$. To complete our proof, it suffices to show that for any $\pi\in \Part(\mathcal{M})$, $\phi(\pi)\in \Aut(\mathcal{N})$. For this, we only  note that for any $n$-ary $R\in \mathcal{L}$ and $\overline{z}\in N^n=Y^n$,
$$ R^N(\overline{z})=\max\{0, \sup\{R^Y(\overline{x})-d^Y_R(\overline{x}, \overline{z})\,:\, \overline{x}\in \dom(R^Y)\}. $$
Thus $R^N(\overline{z})=R^N(\phi(\pi)(\overline{z}))$ since $\phi(\pi)(\overline{x})\in \dom(R^Y)$ iff $\overline{x}\in \dom(R^Y)$, and for $\overline{x}\in\dom(R^Y)$, $R^Y(\overline{x})=R^Y(\phi(\pi)(\overline{x}))$.
\end{proof}

In Theorem~\ref{thm:EPPA} the first two clauses define the property EPPA for $\Kfin$ and (i)--(iii) together give the definition of coherent EPPA. Before closing this section we show that these properties are equivalent to the semiproperness of the continuous signature $\mathcal{L}$. Note, however, that we assume a slightly stronger background condition than before for the continuous signature $\mathcal{L}$.

\begin{theorem}\label{thm:EPPAeq} Let $\mathcal{L}$ be a continuous signature with only finitely many relation symbols, where all the associated moduli of continuity are strictly increasing. The following are equivalent:
\begin{enumerate}
\item $\mathcal{L}$ is semiproper.
\item $\Kfin$ has the EPPA.
\item $\Kfin$ has the coherent EPPA.
\end{enumerate}
\end{theorem}

\begin{proof} (1)$\Rightarrow$(3) by Theorem~\ref{thm:EPPA}. (3)$\Rightarrow$(2) is obvious. We show (2)$\Rightarrow$(1). For this, let $R\in\mathcal{L}$ be $n$-ary and let $1\leq i\leq n$.

We first show that $I_{R,i}$ is bounded. Toward a contradiction, assume $I_{R,i}$ is unbounded, that is, $I_{R,i}=[0,+\infty)$. Let $M=\sup\{\,\mathfrak{u}_{R,i}(r)\,:\, r\geq 0\}\leq 1$. Since $\mathfrak{u}_{R,i}$ is strictly increasing, we have $\mathfrak{u}_{R,i}(r)<M$ for all $r\geq 0$. Let $a, b>0$ be such that $\mathfrak{u}_{R,i}(a)+\mathfrak{u}_{R,i}(b)>M$. Let $\epsilon>0$ be such that $\epsilon< M-\mathfrak{u}_{R,i}(a+b)$. Let $\delta>0$ be sufficiently large such that $\mathfrak{u}_{R,i}(a+\delta+b)>M-\epsilon$. Consider the metric space $(M, d^M)$ consisting of four points $y, s, t, z$ on the real line where $s-y=a$, $t-s=\delta$ and $z-t=b$. For $\overline{x}=(x_1,\dots, x_n)\in M^n$, define
$$ R^M(\overline{x})=\left\{\begin{array}{ll} M-\epsilon & \mbox{ if $x_i=y$} \\
\mathfrak{u}_{R,i}(b) & \mbox{ if $x_i=s$ or $x_i=t$} \\
0 & \mbox{ if $x_i=z$.}
\end{array}\right.
$$
For all other $\tilde{R}\in \mathcal{L}$, define $\tilde{R}^M$ to be identically $0$. It is easy to check that this defines a finite continuous $\mathcal{L}$-structure $\mathcal{M}$. Let $p: \{s, t\}\to \{s, t\}$ be defined as $p(s)=t$ and $p(t)=s$. Then $p$ is a partial isomorphism of $\mathcal{M}$. By the EPPA, there is a finite continuous $\mathcal{L}$-structure $\mathcal{N}$ and an automorphism $f$ of $\mathcal{N}$ such that $\mathcal{N}$ is an extension of $\mathcal{M}$ and $f$ is an extension of $p$. Let $y'=f(y)$. Then $d^N(y',t)=d^N(f(y), f(s))=d^N(y,s)=d^M(y,s)=a$.
Define $\overline{x}^0$, $\overline{x}^1$, $\overline{x}^2$ as follows.
$$\begin{array}{ccc} x^0_j=\left\{\begin{array}{ll} y & \mbox{ if $j=i$} \\ s & \mbox{ if $j\neq i$;}
\end{array}\right.
&
x^1_j=\left\{\begin{array}{ll} y' & \mbox{ if $j=i$} \\ t & \mbox{ if $j\neq i$;}
\end{array}\right.
&
x^2_j=\left\{\begin{array}{ll} z & \mbox{ if $j=i$} \\ t & \mbox{ if $j\neq i$.}
\end{array}\right.
\end{array}
$$
Then $\overline{x}^1=f(\overline{x}^0)$ and thus $R^N(\overline{x}^1)=R^N(\overline{x}^0)=R^M(\overline{x}^0)=
M-\epsilon$. It follows that
$$M-\epsilon =|R^N(\overline{x}^1)-
R^N(\overline{x}^2)| \leq \mathfrak{u}_{R,i}(d^N(y',z)) \leq \mathfrak{u}_{R,i}(d^N(y',t)+d^N(t,z))=\mathfrak{u}_{R,i}(a+b),
$$
contradicting our choice of $\epsilon$.

Next we show that $\mathfrak{u}_{R,i}$ is superadditive on $I_{R,i}$. For this we use a construction similar to the above one. Let $a, b>0$ be such that $a+b\in I_{R,i}$. Let $\delta> \sup I_{R,i}$. Consider the metric space $(M, d^M)$ consisting of four points $y, s, t, z$ on the real line where $s-y=a$, $t-s=\delta$ and $z-t=b$. For $\overline{x}=(x_1,\dots, x_n)\in M^n$, define
$$ R^M(\overline{x})=\left\{\begin{array}{ll} \mbox{min}\{1,\mathfrak{u}_{R,i}(a)+\mathfrak{u}_{R,i}(b)\} & \mbox{ if $x_i=y$} \\
\mathfrak{u}_{R,i}(b) & \mbox{ if $x_i=s$ or $x_i=t$} \\
0 & \mbox{ if $x_i=z$.}
\end{array}\right.
$$
For all other $\tilde{R}\in \mathcal{L}$, define $\tilde{R}^M$ to be identically $0$. It is easy to check that this defines a finite continuous $\mathcal{L}$-structure $\mathcal{M}$. Let $p: \{s, t\}\to \{s, t\}$ be defined as $p(s)=t$ and $p(t)=s$. Then $p$ is a partial isomorphism of $\mathcal{M}$. By the EPPA, there is a finite continuous $\mathcal{L}$-structure $\mathcal{N}$ and an automorphism $f$ of $\mathcal{N}$ such that $\mathcal{N}$ is an extension of $\mathcal{M}$ and $f$ is an extension of $p$. Let $y'=f(y)$. Then $d^N(y',t)=d^N(f(y), f(s))=d^N(y,s)=d^M(y,s)=a$.
Define $\overline{x}^0$, $\overline{x}^1$, $\overline{x}^2$ as before.
$$\begin{array}{ccc} x^0_j=\left\{\begin{array}{ll} y & \mbox{ if $j=i$} \\ s & \mbox{ if $j\neq i$;}
\end{array}\right.
&
x^1_j=\left\{\begin{array}{ll} y' & \mbox{ if $j=i$} \\ t & \mbox{ if $j\neq i$;}
\end{array}\right.
&
x^2_j=\left\{\begin{array}{ll} z & \mbox{ if $j=i$} \\ t & \mbox{ if $j\neq i$.}
\end{array}\right.
\end{array}
$$
Then $\overline{x}^1=f(\overline{x}^0)$ and thus $R^N(\overline{x}^1)=R^N(\overline{x}^0)=R^M(\overline{x}^0)=
\mbox{min}\{1,\mathfrak{u}_{R,i}(a)+\mathfrak{u}_{R,i}(b)\}$. It follows that
$$\begin{array}{rcl}\mbox{min}\{1,\mathfrak{u}_{R,i}(a)+\mathfrak{u}_{R,i}(b)\} &=&|R^N(\overline{x}^1)-
R^N(\overline{x}^2)| \\
&\leq& \mathfrak{u}_{R,i}(d^N(y',z)) \\
&\leq& \mathfrak{u}_{R,i}(d^N(y',t)+d^N(t,z))=\mathfrak{u}_{R,i}(a+b).
\end{array}
$$
Since $\mathfrak{u}_{R,i}(a+b)<1$, we have $\mathfrak{u}_{R,i}(a)+\mathfrak{u}_{R,i}(b)\leq \mathfrak{u}_{R,i}(a+b)$ as desired.

Finally we show that if $R\in\mathcal{L}$ is $n$-ary with $n\geq 2$ and $1\leq i\leq n$, then there is $K_{R,i}>0$ such that $\mathfrak{u}_{R,i}(r)=K_{R,i}r$ for all $r\in I_{R,i}$. For this we show that $\mathfrak{u}_{R,i}$ is subadditive on $I_{R,i}$. Let $1\leq k\leq n$ be such that $k\neq i$. Let $a, b>0$ be such that $a+b\in I_{R,i}$. Consider the metric space $(M, d^M)$ consisting of five points $r, s, t, z, y$, where $r,s,t$ form an equilateral triangle with side length $a+b$, $d^N(y,s)=b$, $d^N(y, t)=a$, $d^N(y, r)=\mbox{min}\{2a+b, a+2b\}$, and $d^N(z,r)=d^N(z,s)=d^N(z,t)=d^N(z,y)=\delta$ for some $\delta>\sup I_{R,i}$. For $\overline{x}=(x_1,\dots, x_n)\in M^n$, define
$$ R^M(\overline{x})=\left\{\begin{array}{ll} \mathfrak{u}_{R,i}(a+b) & \mbox{ if $x_i=r$ and $x_k=z$} \\ 0 & \mbox{ otherwise.}\end{array}\right. $$
For all other $\tilde{R}\in\mathcal{L}$, define $\tilde{R}^M$ to be identically $0$. It is easy to check that this defines a finite continuous $\mathcal{L}$-structure $\mathcal{M}$. Let $p:\{r,s,t\}\to \{r,s,t\}$ be such that $p(r)=s$, $p(s)=t$ and $p(t)=r$. Then $p$ is a partial isomorphism of $\mathcal{M}$. By the EPPA, we obtain a finite continuous $\mathcal{L}$-structure $\mathcal{N}$ and an automorphism $f$ of $\mathcal{N}$ such that $\mathcal{N}$ is an extension of $\mathcal{N}$ and $f$ is an extension of $p$. Let $y'=f(y)\in N$. Then
$d^N(y', r)=d^N(f(y), f(t))=d^N(y, t)=a$ and $d^N(y', t)=d^N(f(y), f(s))=d^N(y, s)=b$. Define $\overline{x}^0$, $\overline{x}^1$, $\overline{x}^2$ as follows.
$$\begin{array}{ccc} x^0_j=\left\{\begin{array}{ll} r & \mbox{ if $j=i$} \\ z & \mbox{ if $j\neq i$;}
\end{array}\right.
&
x^1_j=\left\{\begin{array}{ll} t & \mbox{ if $j=i$} \\ z & \mbox{ if $j\neq i$;}
\end{array}\right.
&
x^2_j=\left\{\begin{array}{ll} y' & \mbox{ if $j=i$} \\ z & \mbox{ if $j\neq i$.}
\end{array}\right.
\end{array}
$$
Then
$$\begin{array}{rcl} \mathfrak{u}_{R,i}(a+b)&=&|R^N(\overline{x}^0)-R^N(\overline{x}^1)| \\
&\leq & |R^N(\overline{x}^0)-R^N(\overline{x}^2)|+|R^N(\overline{x}^2)-R^N(\overline{x}^1)| \\
&\leq& \mathfrak{u}_{R,i}(d^N(r,y'))+\mathfrak{u}_{R,i}(d^N(y',t))
=\mathfrak{u}_{R,i}(a)+\mathfrak{u}_{R,i}(b).
\end{array}
$$

\end{proof}

\section{Actions by Automorphisms on Continuous Structures}

In this section we prove that the actions by automorphisms on finite continuous $\mathcal{L}$-structures form a Fra\"iss\'e class for any semiproper continuous signature $\mathcal{L}$.

\begin{definition}\label{def:C} Let $\mathcal{L}$ be a semiproper continuous signature. Let $\mathcal{C}_{\mathcal{L}}$ be the class of all actions $\Gamma\actson \mathcal{M}$ such that
\begin{itemize}
\item $\mathcal{M}$ is a finite continuous $\mathcal{L}$-structure,
\item $\Gamma$ is a finite group, and
\item $\Gamma\actson \mathcal{M}$ is an action by automorphisms.
\end{itemize}
\end{definition}

Our objective is to show that $\mathcal{C}_{\mathcal{L}}$ is a Fra\"iss\'e class. By this we mean that $\mathcal{C}_{\mathcal{L}}$ has the hereditary property, the joint embedding property, and the amalgamation property. They are defined below.

\begin{definition} Let $\mathcal{L}$ be a semiproper continuous signature. Let $\Lambda, \Gamma$ be groups, $\mathcal{M}, \mathcal{N}$ be continuous $\mathcal{L}$-pre-structures, and $\tau:\Lambda\actson \mathcal{M}$ and $\pi:\Gamma\actson \mathcal{N}$ be actions by automorphisms.
\begin{enumerate}
\item An {\em embedding} from $\tau$ into $\pi$ is a pair $(e, f)$ where $e$ is a group isomorphic embedding from $\Lambda$ into $\Gamma$ and $f$ is an embedding from $\mathcal{M}$ to $\mathcal{N}$ as continuous $\mathcal{L}$-pre-structures such that for all $g\in \Lambda$ and $x\in M$, $e(g)\cdot_{\pi}f(x)=f(g\cdot_{\tau}x)$.
\item We say that $\mathcal{C}_{\mathcal{L}}$ has the {\em hereditary property} (HP for short) if for any $\tau:\Lambda\actson \mathcal{M}$ and $\pi:\Gamma\actson \mathcal{N}$, whenever $\pi\in\mathcal{C}_{\mathcal{L}}$ and there is an embedding from $\tau$ into $\pi$, then $\tau\in\mathcal{C}_{\mathcal{L}}$.
\item We say that $\mathcal{C}_{\mathcal{L}}$ has the {\em joint embedding property} (JEP for short) if given any $\tau,\pi\in\mathcal{C}_{\mathcal{L}}$ there is a $\kappa\in\mathcal{C}_{\mathcal{L}}$ and embeddings from $\tau$ into $\kappa$ and from $\pi$ into $\kappa$.
\item We say that $\mathcal{C}_{\mathcal{L}}$ has the {\em amalgamation property} (AP for short) if given any $\mu, \tau, \pi\in \mathcal{C}_{\mathcal{L}}$ and embeddings $(e,f)$ from $\mu$ into $\tau$ and $(p,q)$ from $\mu$ into $\pi$, there exist a $\kappa\in\mathcal{C}_{\mathcal{L}}$ and embeddings $(g,h)$ from $\tau$ into $\kappa$ and $(r,s)$ from $\pi$ into $\kappa$ such that
    $$ g\circ e=r\circ p \mbox{ and } h\circ f=s\circ q. $$
\end{enumerate}
\end{definition}

 Our proof will follow the general line of arguments for Theorems 3.9 and 4.8 of \cite{EGLMM}. Two key ingredients of the proof are analogs of some results of Rosendal (Lemma 16 of \cite{Rosendal} and Theorem 7 of \cite{RosendalI}, also c.f. Lemma 4.7 of \cite{EGLMM}, Theorem 4.6 of \cite{EGLMM} and Theorem 3.3 of \cite{EG}). We develop these results first.

\subsection{Extensions of actions by automorphisms}

In this subsection we generalize a lemma of Rosendal (Lemma 16 of \cite{Rosendal}; also c.f. Lemma 4.7 of \cite{EGLMM}) to finite continuous $\mathcal{L}$-structures for a semiproper continuous signature $\mathcal{L}$.

We will prove a more general result with the stronger assumption that $\mathcal{L}$ is a proper continuous signature. The same proof will give the desired result for semiproper $\mathcal{L}$. So, for the time being let us assume $\mathcal{L}$ is proper.

Assume $\Gamma$ is a group, $\Lambda\leq \Gamma$ is a subgroup, $\mathcal{M}\subseteq \mathcal{N}$ are continuous $\mathcal{L}$-pre-structures, and $\Gamma\actson \mathcal{M}$ and $\Lambda\actson \mathcal{N}$ are compatible actions by automorphisms.

Define a pseudo-metric $\partial$ on $N\times\Gamma$ by
$$\begin{array}{rcl} & & \partial((y_1, g_1), (y_2, g_2)) \\
&=&\left\{\begin{array}{l} d^N(g_2^{-1}g_1\cdot y_1, y_2)
\ \ \ \ \ \ \ \mbox{ if either $y_1, y_2\in M$ or $g_2^{-1}g_1\in \Lambda$} \\
\inf_{x\in M}\{d^N(y_1, g_1^{-1}\cdot x)+d^N(y_2, g_2^{-1}\cdot x)\} \ \ \ \ \ \ \  \mbox{ otherwise.}
\end{array}\right.
\end{array}
$$
Then $\partial$ is a pseudo-metric on $N\times \Gamma$ and for all $(y_1, g_1), (y_2, g_2)\in N\times \Gamma$,
$$ \partial((y_1, g_1), (y_2, g_2))\leq \inf_{x\in X}\{d^N(y_1, g_1^{-1}\cdot x)+d^N(y_2, g_2^{-1}\cdot x)\}. $$

Define an equivalence relation $\sim$ on $N\times \Gamma$ by
$$ (y_1, g_1)\sim (y_2, g_2)\iff \partial ((y_1, g_1),(y_2, g_2))=0. $$
Then $(y_1, g_1)\sim (y_2, g_2)$ iff
$$(y_1, y_2\in M \mbox{ or } g_2^{-1}g_1\in\Lambda) \mbox{ and } g_2^{-1}g_1\cdot y_1=y_2. $$
Let $[y, g]$ denote the $\sim$-equivalence class of $(y, g)$. Let
$$ P=(N\times \Gamma)/\sim $$
and let
$$ d^P([y_1, g_1], [y_2, g_2])=\partial((y_1, g_1), (y_2, g_2)). $$
Then $d^P$ is a metric on $P$.

Let $\varphi: N\to P$ be defined as $\varphi(y)=[y,1]$. Then $\varphi$ is an isometric embedding from $(N,d^N)$ into $(P, d^P)$, since
$$ d^P(\varphi(y_1),\varphi(y_2))=\partial((y_1, 1), (y_2, 1))=d^N(y_1, y_2) $$
for any $y_1, y_2\in N$.

Define $\Gamma\actson P$ by
$$ g\cdot [y, h]=[y, gh] $$
for $g, h\in \Gamma$ and $y\in N$.
It is easy to see that this is well defined.

Now suppose $R\in\mathcal{L}$ is an $n$-ary relation symbol. Define
$$ R^P([y_1, g_1], \dots, [y_1, g_n])=R^N(z_1, \dots, z_n) $$
if $z_1, \dots, z_n\in N$, $g\in \Gamma$, and
$$ [y_1, g_1]=[z_1, g], \dots, [y_n, g_n]=[z_n, g]. $$
Thus
$$\begin{array}{c} [y_1, g_1], \dots, [y_1, g_n]\in \dom(R^P)\iff \ \ \ \ \ \ \ \ \ \ \ \ \ \ \ \ \ \ \ \ \ \ \ \ \ \ \  \\
\ \ \ \ \ \ \ \ \ \ \ \ \ \ \ \exists g\in \Gamma, z_1, \dots, z_n\in N\ \bigwedge_{i=1}^n [y_i, g_i]=[z_i, g].
\end{array}$$
We verify that $R^P$ is well defined on its domain. For this let
$$([y_1, g_1], \dots, [y_n, g_n])\in \dom(R^P).$$ Suppose $z_1, \dots, z_n, z_1', \dots, z_n'\in N$ and $g, g'\in \Gamma$ such that for all $1\leq i\leq n$,
$$ [y_i, g_i]=[z_i, g]=[z_i', g']. $$
We need to show that $R^N(z_1, \dots, z_n)=R^N(z_1',\dots, z_n')$. Note that we have $g^{-1}g'\cdot z_i'=z_i$ for all $1\leq i\leq n$, and thus
$$ R^N(z_1,\dots, z_n)=R^N(g^{-1}g'\cdot z_1', \dots, g^{-1}g'\cdot z_n')=R^N(z_1', \dots, z_n'). $$

We also note that $\dom(R^P)$ is invariant under the $\Gamma$ action. That is, if $[\overline{y}, h]\in \dom(R^P)$ and $g\in\Gamma$ then
$$ g\cdot [\overline{y}, h]=[\overline{y}, gh]\in \dom(R^P) $$
and
$$ R^P(g\cdot[\overline{y}, h])=R^N(\overline{y})=R^P([\overline{y}, h]). $$

\begin{lemma}\label{lem:RP} For any $R\in\mathcal{L}$, $R^P$ on $\dom(R^P)$ is $1$-Lipschitz with respect to $d^P_R$.
\end{lemma}

\begin{proof} We first consider unary $R\in\mathcal{L}$. For any $[y_1, g_1], [y_2, g_2]\in P$, if $y_1, y_2\in M$ or $g_2^{-1}g_1\in \Lambda$ then
$$\begin{array}{rcl}& &|R^P([y_1, g_1])-R^P([y_2, g_2])| =|R^N(g_2^{-1}g_1\cdot y_1)-R^N(y_2)| \\
&\leq& \mathfrak{u}_{R,1}(d^N(g_2^{-1}g_1\cdot y_1, y_2))=\mathfrak{u}_{R,1}(d^P([y_1, g_1], [y_2, g_2])).
\end{array}
$$
Otherwise, for any $x\in M$,
$$\begin{array}{rcl}& &|R^P([y_1, g_1])-R^P([y_2, g_2])| =|R^N(y_1)-R^N(y_2)|  \\
&\leq& |R^N(y_1)-R^N(g_1^{-1}\cdot x)|+|R^N(g_2^{-1}\cdot x)-R^N(y_2)|\\
&\leq& \mathfrak{u}_{R,1}(d^N(y_1, g_1^{-1}\cdot x))+\mathfrak{u}_{R,1}(d^N(y_2, g_2^{-1}\cdot x)) \\
&\leq  &\mathfrak{u}_{R,1}(d^N(y_1, g_1^{-1}\cdot x)+d^N(y_2, g_2^{-1}\cdot x))
\end{array}
$$
if $\mathfrak{u}_{R,1}(d^N(y_1, g_1^{-1}\cdot x)+d^N(y_2, g_2^{-1}\cdot x))<1$, and otherwise we have nonetheless
$$ |R^P([y_1, g_1])-R^P([y_2, g_2])| \leq  \mathfrak{u}_{R,1}(d^N(y_1, g_1^{-1}\cdot x)+d^N(y_2, g_2^{-1}\cdot x)). $$
Taking the infimum of the right hand side over $x\in M$, and using the upper semicontinuity of $\mathfrak{u}_{R,1}$, we get that
$$  |R^P([y_1, g_1])-R^P([y_2, g_2])| \leq  \mathfrak{u}_{R,1}(d^P([y_1,g_1], [y_2,g_2])). $$
Thus (\ref{eqn:ucl}) holds for $R^P$. It follows that for any $[y_1, g_1], [y_2, g_2]\in P$,
$$ |R^P([y_1, g_1])-R^P([y_2, g_2])| \leq  d^P_R([y_1,g_1], [y_2,g_2]). $$

Next we consider $n$-ary $R\in\mathcal{L}$ where $n\geq 2$. Let $y_1, \dots, y_n, z_1, \dots, z_n\in N$ and $g, h\in \Gamma$. If $h^{-1}g\in \Lambda$ or $y_1, \dots, y_n, z_1, \dots, z_n\in M$, then
$$\begin{array}{rcl}& &|R^P([y_1, g], \dots, [y_n, g])-R^P([z_1, h],\dots, [z_n, h])| \\
&=&|R^N(h^{-1}g\cdot y_1, \dots, h^{-1}g\cdot y_n)-R^N(z_1,\dots, z_n)| \\
&\leq& d^N_R(h^{-1}g\cdot \overline{y}, \overline{z})=d^P_R([\overline{y}, g], [\overline{z}, h]).
\end{array}
$$
Otherwise, for any $\overline{x}\in M^n$, we have
$$\begin{array}{rcl}& &|R^P([y_1, g], \dots, [y_n, g])-R^P([z_1, h],\dots, [z_n, h])| \\
&\leq &|R^N(y_1,\dots, y_n)-R^N(g^{-1}\cdot x_1, \dots, g^{-1}\cdot x_n)|+\\
& & |R^N(h^{-1}\cdot x_1, \dots, h^{-1}\cdot x_n)-R^N(z_1,\dots, z_n)| \\
&\leq& \displaystyle\sum_{i=1}^n K_{R,i}d^N(y_i, g^{-1}\cdot x_i)+K_{R,i}d^N(z_i, h^{-1}\cdot x_i) \\
&=  &\displaystyle\sum_{i=1}^n K_{R,i}(d^N(y_i, g^{-1}\cdot x_i)+d^N(z_i, h^{-1}\cdot x_i)).
\end{array}
$$
Taking the infimum of the right hand side over all $\overline{x}\in M$, we get
$$|R^P([\overline{y}, g])-R^P([\overline{z}, h])|\leq d^P_R([\overline{y}, g], [\overline{z}, h]). $$
Here we use the fact that if $y_i, z_i\in M$ then we also have
$$ d^P([y_i, g], [z_i,h])=d^N(h^{-1}g\cdot y_i, z_i)=\inf_{x_i\in M}\{d^N(y_i, g^{-1}\cdot x_i)+d^N(z_i, h^{-1}\cdot x_i)\}. $$
\end{proof}

Thus $\mathcal{P}=(P, d^P, (R^P)_{R\in \mathcal{L}})$ is a partially defined continuous $\mathcal{L}$-pre-structure. We define a conservative extension $\mathcal{Q}$ of $\mathcal{P}$ with an action $\Gamma\actson \mathcal{Q}$. Let $(Q, d^Q)=(P, d^P)$. For any $n$-ary $R\in\mathcal{L}$ and $\overline{w}\in Q^n$, define
$$R^Q(\overline{w})=\max\{0, \sup\{ R^P(\overline{z})-d^P_R(\overline{w},\overline{z})\,:\, \overline{z}\in \dom(R^P)\}\}. $$
Then for any $g\in \Gamma$,
$$\begin{array}{rcl} R^Q(g\cdot \overline{w})&=&\max\{0, \sup\{R^P(g\cdot\overline{z})-d^P_R(g\cdot\overline{w}, g\cdot\overline{z})\,:\, \overline{z}\in \dom(R^P)\} \\
&=&\max\{0,\sup\{ R^P(\overline{z})-d^P_R(\overline{w},\overline{z})\,:\, \overline{z}\in \dom(R^P)\} \\
&=& R^Q(\overline{w}).
\end{array}
$$
We show that $\mathcal{Q}$ is a continuous $\mathcal{L}$-pre-structure, i.e., (\ref{eqn:ucl}) holds for $\mathcal{Q}$. For unary $R\in\mathcal{L}$ there is nothing to prove since all $[y,g]\in \dom(R^P)$ and we have Lemma~\ref{lem:RP}. For $n\geq 2$ and $n$-ary $R\in\mathcal{L}$ the proof is identical to the proof of Lemma~\ref{lem:gconservative}.

Thus we obtain an extension $\mathcal{Q}$ of $\mathcal{N}$ with an action $\Gamma\actson \mathcal{Q}$. It is clear that the action $\Gamma\actson \mathcal{Q}$ is compatible with both $\Gamma\actson\mathcal{M}$ and $\Lambda\actson \mathcal{N}$.

Thus we have proved the following.

\begin{theorem}\label{lem:Fraext} Let $\mathcal{L}$ be a proper continuous signature. Let $\Gamma$ be a group, $\Lambda\leq \Gamma$ be a subgroup, $\mathcal{M}\subseteq \mathcal{N}$ be continuous $\mathcal{L}$-pre-structures, and $\Gamma\actson \mathcal{M}$ and $\Lambda\actson \mathcal{N}$ be compatible actions by automorphisms. Then there is a continuous $\mathcal{L}$-pre-structure $\mathcal{Q}$ and an action $\Gamma\actson \mathcal{Q}$ by automorphisms such that $\mathcal{Q}$ is an extension of $\mathcal{N}$, the action $\Gamma\actson \mathcal{Q}$ is compatible with both $\Gamma\actson \mathcal{M}$ and $\Lambda\actson \mathcal{N}$.

Moreover, if $\mathcal{L}$ is semiproper and $\mathcal{N}$ and $\Gamma$ are finite, then the same holds with a finite $\mathcal{Q}$.
\end{theorem}

Before closing this subsection we note that the above extension properties are in fact equivalent to the semiproperness and the properness of $\mathcal{L}$ respectively. Here again we assume the slightly strong background condition for the continuous signature $\mathcal{L}$.

\begin{theorem}\label{thm:Fraexteq} Let $\mathcal{L}$ be a continuous signature with only finitely many relation symbols, where all the associated moduli of continuity are strictly increasing. The following are equivalent:
\begin{enumerate}
\item[(i)] $\mathcal{L}$ is semiproper.
\item[(ii)] Let $\Gamma$ be a group, $\Lambda\leq \Gamma$ be a subgroup, $\mathcal{M}\subseteq \mathcal{N}$ be finite continuous $\mathcal{L}$-structures, and $\Gamma\actson \mathcal{M}$ and $\Lambda\actson \mathcal{N}$ be compatible actions by automorphisms. Then there is a finite continuous $\mathcal{L}$-structure $\mathcal{Q}$ and an action $\Gamma\actson \mathcal{Q}$ by automorphisms such that $\mathcal{Q}$ is an extension of $\mathcal{N}$, the action $\Gamma\actson \mathcal{Q}$ is compatible with both $\Gamma\actson \mathcal{M}$ and $\Lambda\actson \mathcal{N}$.
\item[(iii)] Let $\Gamma$ be a group, $\Lambda\leq \Gamma$ be a subgroup, $\mathcal{M}\subseteq \mathcal{N}$ be finite continuous $\mathcal{L}$-structures, and $\Gamma\actson \mathcal{M}$ and $\Lambda\actson \mathcal{N}$ be compatible actions by automorphisms. Then there is a continuous $\mathcal{L}$-(pre)-structure $\mathcal{Q}$ and an action $\Gamma\actson \mathcal{Q}$ by automorphisms such that $\mathcal{Q}$ is an extension of $\mathcal{N}$, the action $\Gamma\actson \mathcal{Q}$ is compatible with both $\Gamma\actson \mathcal{M}$ and $\Lambda\actson \mathcal{N}$.
\end{enumerate}
\end{theorem}

\begin{proof} (i)$\Rightarrow$(ii) by Theorem~\ref{lem:Fraext}. (ii)$\Rightarrow$(iii) is obvious. For (iii)$\Rightarrow$(i) we use the same constructions and similar arguments in the proof of
(2)$\Rightarrow$(1) of Theorem~\ref{thm:EPPAeq}.
\end{proof}

\begin{theorem}\label{thm:Fraexteq2} Let $\mathcal{L}$ be a continuous signature with only finitely many relation symbols, where all the associated moduli of continuity are strictly increasing. The following are equivalent:
\begin{enumerate}
\item[(i)] $\mathcal{L}$ is proper.
\item[(ii)] Let $\Gamma$ be a group, $\Lambda\leq \Gamma$ be a subgroup, $\mathcal{M}\subseteq \mathcal{N}$ be continuous $\mathcal{L}$-pre-structures, and $\Gamma\actson \mathcal{M}$ and $\Lambda\actson \mathcal{N}$ be compatible actions by automorphisms. Then there is a continuous $\mathcal{L}$-pre-structure $\mathcal{Q}$ and an action $\Gamma\actson \mathcal{Q}$ by automorphisms such that $\mathcal{Q}$ is an extension of $\mathcal{N}$, the action $\Gamma\actson \mathcal{Q}$ is compatible with both $\Gamma\actson \mathcal{M}$ and $\Lambda\actson \mathcal{N}$.
\end{enumerate}
\end{theorem}

\begin{proof} (i)$\Rightarrow$(ii) by Theorem~\ref{lem:Fraext}. For (ii)$\Rightarrow$(i), we have from the implication (iii)$\Rightarrow$(i) of Theorem~\ref{thm:Fraexteq} that $\mathcal{L}$ is semiproper. It remains to show that for any unary $R\in\mathcal{L}$, $\mathfrak{u}_{R,1}$ is upper semicontinuous. For this, let $a$ be a point of the interior of $I_{R,1}$. Let $\delta>\sup I_{R,1}$. Let $(N, d^N)$ be a metric space consisting of the following points on the real line
$$r_1, r_2, \dots, r_n, \dots, r, y, s_1, s_2,\dots, s_n, \dots $$
where $r-r_n=2^{-n}$, $y-r=\delta$, $s_n-y=a+2^{-n}$. Define $R^N(y)=\inf\{\mathfrak{u}_{R,1}(r)\,:\, r>a\}$ and $R^N(x)=0$ for $x\neq y$. For all other $\tilde{R}\in\mathcal{L}$, define $\tilde{R}^N$ to be identically $0$. It is easy to check that this defines a continuous $\mathcal{L}$-pre-structure. Let $\mathcal{M}$ be the substructure of $\mathcal{N}$ with $M=\{r_1,\dots, r_n, \dots, s_1,\dots, s_n, \dots\}$. Let $\Gamma=\langle g\rangle$ where $g^2=1$, and let $\Gamma\actson \mathcal{M}$ by $g\cdot r_n=s_n$ and $g\cdot s_n=r_n$. Let $\Lambda$ be the trivial group. By (ii) there is a continuous $\mathcal{L}$-pre-structure $\mathcal{Q}$ and an action $\Gamma\actson \mathcal{Q}$ by automorphisms such that $\mathcal{Q}$ is an extension of $\mathcal{N}$ and the action $\Gamma\actson \mathcal{Q}$ is compatible with $\Gamma\actson \mathcal{M}$. Let $s=g\cdot r\in Q$. Then $d^Q(s, s_n)=d^Q(g\cdot r, g\cdot r_n)=d^Q(r, r_n)=d^N(r, r_n)=2^{-n}$.  It follows that $d^Q(y,s)=\lim_n d^Q(y,s_n)=a$. Also $R^Q(s)=R^Q(r)=0$. We have
$$ \inf\{\mathfrak{u}_{R,1}(r)\,:\, r>a\}=|R^Q(y)-R^Q(s)|\leq \mathfrak{u}_{R,1}(d^Q(y, s))=\mathfrak{u}_{R,1}(a).
$$
\end{proof}

\subsection{Finite approximations of actions}

In this section we prove a generalization of a theorem of Rosendal (Theorem 7 of \cite{RosendalI}, also c.f. Theorem 3.3 of \cite{EG} and Theorems 3.7 and 4.6 of \cite{EGLMM}) to continuous $\mathcal{L}$-structures for semiproper $\mathcal{L}$.

In order to state the theorem we need the definition of the HL-property of a group (c.f. \cite{HL} and \cite{EG}; HL stands for Herwig--Lascar).

\begin{definition} Let $G$ be a group.
 \begin{enumerate}
 \item[(i)] Let $H_1, \dots, H_n\leq G$. A {\em left system} of equations on $H_1, \dots, H_n$ is a finite set of equations with variables $x_1,\dots, x_m$ and constants $g_1,\dots, g_\ell\in G$ such that each equation is of the form
     $$ x_iH_j=g_kH_j \mbox{ or } x_iH_j=x_rg_kH_j $$
     where $1\leq i, r\leq m$, $1\leq k\leq \ell$ and $1\leq j\leq n$.
 \item[(ii)] We say that $G$ has the {\em HL-property} if for every finitely generated $H_1,\dots, H_n$ $\leq G$ and left system of equations on $H_1,\dots, H_n$ that does not have a solution, there exist normal subgroups of finite index $N_1, \dots, N_n\unlhd G$ such that the same left system of equations on $N_1H_1, \dots, N_nH_n$ does not have a solution.
 \end{enumerate}
 \end{definition}

 Note that by introducing new variables we may also include equations of the form
 $$ x_ig_kH_j=x_rg_sH_j $$
 in a left system.

\begin{theorem} \label{thm:Rosendal} Let $\mathcal{L}$ be a semiproper continuous signature. Let $\Gamma$ be a countable group with the HL-property. Assume that $\pi: \Gamma\actson \mathcal{M}$ is an action by automorphisms on a continuous $\mathcal{L}$-pre-structure $\mathcal{M}$. Then for any finite $A\subseteq M$ and finite $F\subseteq \Gamma$ there exist a finite continuous $\mathcal{L}$-structure $\mathcal{N}$ and an action by automorphisms $\pi':\Gamma\actson \mathcal{N}$ such that $\mathcal{N}$ is an extension of $\mathcal{A}$ and for all $\gamma\in F$ and $a\in A$, if $\gamma\cdot_{\pi}a\in A$ then
we have $\gamma\cdot_{\pi'}a=\gamma\cdot_{\pi}a$.
\end{theorem}

The remainder of this subsection is devoted to a proof of Theorem~\ref{thm:Rosendal}. Our strategy is to follow Rosendal's proof of Theorem 7 \cite{RosendalI} (a)$\Rightarrow$(b) for the metric part (also c.f. \cite{EG0}) and then the Etedadialiabadi--Gao's proof of Theorem 3.3 \cite{EG} (i)$\Rightarrow$(ii) for the relations part. Rosendal's proof uses the so-called property (RZ) or RZ-property of the group $\Gamma$, and Etedadialiabadi--Gao's proof uses the HL-property. Since the RZ-property is implied by the HL-property, we will see that the two parts of the proof produces compatible requirements, which can be resolved by invoking the HL-property of $\Gamma$ once. In fact, we remark that the proof can go through with assuming RZ-property only; however, we use the HL-property as the presentation is notationally simpler. Our application of the HL-property will give a finite partially defined continuous $\mathcal{L}$-structure, and we complete the proof by Lemma~\ref{lem:gconservative}.

Suppose a proper continuous signature $\mathcal{L}$, a continuous $\mathcal{L}$-structure $\mathcal{M}$, a group $\Gamma$ with the HL-property, an action by automorphism $\pi: \Gamma\actson \mathcal{M}$, and a finite $A\subseteq M$ are given.

Without loss of generality we may assume that the action $\pi$ is faithful, i.e., for all $\gamma\in\Gamma$ such that $\gamma\neq 1_{\Gamma}$, there is some $x\in M$ such that $\gamma\cdot_{\pi} x\neq x$. In the following, when there is no danger of confusion, we omit the subscript in $\cdot_{\pi}$.

Let
$$ P_A=\{ d^M(x, y)\,:\, x\neq y\in A\} $$
$$ \delta_A=\max\{\mbox{diam}(A), \sup\{I_{R,i}\,:\, R\in\mathcal{L}\mbox{ is unary}, 1\leq i\leq n\}\} $$
and
$$ P=\{\delta_A\}\cup \left(\left\{\displaystyle\sum_{i=1}^m p_i\,:\, p_1,\dots, p_m\in P_A\right\}\cap [0, \delta_A]\right). $$
Since $A$ is finite, $P$ is finite.

We define a continuous $\mathcal{L}$-pre-structure $\tilde{\mathcal{M}}$ as follows. Let $\tilde{M}=M$ and
$$ d^{\tilde{M}}(x, y)=\left\{\begin{array}{ll} \min\{p\in P\,:\, d^M(x, y)\leq p\} & \mbox{ if $d^M(x,y)\leq \delta_A$} \\ \delta_A & \mbox{ otherwise.}\end{array}\right. $$
Then $d^{\tilde{M}}$ is a metric on $M$, and the action $\pi:\Gamma\actson M$ is still an action by isometries on $(M, d^{\tilde{M}})$ (c.f. Lemma 4 of \cite{RosendalI}).  For any $R\in\mathcal{L}$, we define $R^{\tilde{M}}=R^M$. Since $d^M(x,y)\leq d^{\tilde{M}}(x,y)$ for any $x, y\in M$ with $d(x,y)\leq \delta_A$, we have that (\ref{eqn:ucl}) holds for $R^{\tilde{M}}$ for all $R\in\mathcal{L}$. This shows that
$$ \tilde{\mathcal{M}}=(\tilde{M}, d^{\tilde{M}}, (R^{\tilde{M}})_{R\in\mathcal{L}}) $$
is a continuous $\mathcal{L}$-pre-structure.

Note that the action $\pi: \Gamma\actson \tilde{\mathcal{M}}$ is still by automorphism, and that $\mathcal{A}$ is a substructure of $\tilde{\mathcal{M}}$. Thus without loss of generality we may assume $\tilde{\mathcal{M}}=\mathcal{M}$. In particular, we have that for all $x, y\in M$, $d^M(x,y)\in P$.

Also without loss of generality we may assume that for any $x\in M$ there is $a\in A$ and $\gamma\in \Gamma$ with $x=\gamma\cdot a$. In fact, let
$$ M'=\{x\in M\,:\, \exists a\in A\ \exists \gamma\in\Gamma\ x=\gamma\cdot a\}. $$
Then $M'$ is obviously a $\pi$-invariant subset, and hence $\mathcal{M}'$ a $\pi$-invariant substructure with $A\subseteq M'$. Thus without loss of generality we may assume $\mathcal{M}=\mathcal{M}'$.

Fix $a_1, \dots, a_m\in A$ such that $\{\Gamma a_i\,:\, 1\leq i\leq m\}$ form a partition of $M$, where
$$ \Gamma a_i=\{\gamma\cdot a_i\,:\, \gamma\in \Gamma\}. $$

Define $\rho:\Gamma\to \Part(\mathcal{A})$ by letting, for any $\gamma\in\Gamma$ and $a\in A$, $\rho(\gamma)(a)=\gamma\cdot a$ if $\gamma\cdot a\in A$, and $\rho(\gamma)(a)$ is undefined otherwise.

Let $F\subseteq \Gamma$ be finite. Since $\pi$ is faithful, by extending $A$ and $F$ with finitely many elements if necessary, we may assume without loss of generality that
\begin{enumerate}
\item[(i)] $1_{\Gamma}\in F$;
\item[(ii)] $F=F^{-1}$;
\item[(iii)] for all $\gamma\in F$ such that $\gamma\neq 1_{\Gamma}$, there is $a\in A$ such that $a\neq\rho(\gamma)(a)\in A$;
\item[(iv)] for all $\gamma\in \Gamma$ there is some $\eta\in F$ such that $\rho(\gamma)=\rho(\eta)$.
\end{enumerate}

For $1\leq i\leq m$, define
$$ H_i=\{ \gamma_1\cdots\gamma_{\ell}\,:\, \gamma_1,\dots, \gamma_{\ell}\in F\mbox{ and } \rho(\gamma_1)(\cdots(\rho(\gamma_{\ell})(a_i))\cdots)=a_i \}. $$
Since $A$ and $F$ are finite, $H_i$ is a finitely generated subgroup of $\Gamma$. Note that for any $\gamma\in H_i$, we have $\gamma\cdot a_i=a_i$.

Let $X$ be the disjoint union of left-coset spaces
$$ X=\Gamma/H_1\sqcup \cdots \sqcup\Gamma/H_m. $$
Let $\Gamma$ act on $X$ by left multiplication.

We define a pseudometric $d^X$ on $X$ as follows. For $\gamma, \eta\in \Gamma$ and $1\leq i, j\leq m$,
$$ d^X(\gamma H_i, \eta H_j)= d^M(\gamma\cdot a_i, \eta\cdot a_j). $$
Then it is easy to see that $d^X$ is well defined, and the $\Gamma$ action on $X$ is by isometries.

Define $e: A\to X$ by
$$ e(a)=\left\{\begin{array}{ll} H_i & \mbox{ if $a=a_i$} \\ \gamma H_i & \mbox{ if $a\neq a_i$ and $a=\gamma\cdot a_i$ for $\gamma\in F$.}\end{array}\right. $$
To see that $e$ is well defined, first note that for any $a\in A\cap\Gamma a_i$ but $a\neq a_i$, by (iv) there is $\gamma\in F$ such that $\gamma\cdot a_i=a$, and thus $e(a)$ is defined. Next, assume $\gamma, \gamma'\in F$ such that $\gamma\cdot a_i=\gamma'\cdot a_i\in A$.
Then $\gamma^{-1}\in F$ and $\rho(\gamma^{-1})(\rho(\gamma')(a_i))=a_i$, and hence $\gamma^{-1}\gamma'\in H_i$ and $\gamma H_i=\gamma'H_i$.

We claim that $e$ is an isometric embedding from $(A, d^A)$ into $(X, d^X)$. To see this, assume $a=\gamma\cdot a_i\in A$ and $b=\eta\cdot a_j\in A$ for $\gamma, \eta\in F$. Then
$$ d^X(e(a), e(b))=d^X(\gamma H_i, \eta H_j)=d^M(\gamma\cdot a_i, \eta\cdot a_j)=d^A(a, b). $$

For any $n$-ary $R\in \mathcal{L}$, define $R^X$ on
$$ \dom(R^X)=\{(\gamma\gamma_1 H_{i_1}, \dots, \gamma\gamma_n H_{i_n})\,:\, \gamma\in\Gamma,\ \gamma_j\cdot a_{i_j}\in A \mbox{ for all } 1\leq j\leq n\} $$
by
$$ R^X(\gamma\gamma_1 H_{i_1}, \dots, \gamma\gamma_n H_{i_n})=R^M(\gamma_1\cdot a_{i_1}, \dots, \gamma_n\cdot a_{i_n}). $$

We claim that for any $n$-ary $R\in \mathcal{L}$ and $b_1, \dots, b_n\in A$, letting $\gamma_j\in F$ be such that $b_j=\gamma_j\cdot a_{i_j}$ for $1\leq j\leq n$, we have
$$ R^X(\gamma_1 H_{i_1},\dots, \gamma_n H_{i_n})=R^A(b_1,\dots, b_n). $$
This is because
$$ R^X(\gamma_1H_{i_1},\dots, \gamma_nH_{i_n})
= R^M(\gamma_1\cdot a_{i_1},\dots, \gamma_n\cdot a_{i_n})
= R^A(b_1,\dots, b_n).
$$

We claim that the following conditions hold:

\begin{enumerate}
\item[(C1)] For any $\gamma, \eta, \gamma', \eta'\in F$ and $1\leq i, j\leq m$, if $\gamma\cdot a_i, \gamma'\cdot a_{i}, \eta\cdot a_j, \eta'\cdot a_{j}\in A$ and
    $$ d^A(\gamma\cdot a_i, \eta\cdot a_j)\neq d^A(\gamma'\cdot a_{i}, \eta'\cdot a_{j}), $$
    then there does not exist $g\in \Gamma$ such that
    $$ g\gamma H_i=\gamma'H_i \mbox{ and } g\eta H_j=\eta' H_j. $$
\item[(C2)] For any $\gamma, \eta\in F$ and $1\leq i\leq m$, if $\gamma\cdot a_i, \eta\cdot a_i\in A$ and $\gamma\cdot a_i\neq \eta\cdot a_i$, then $\gamma^{-1}\eta\not\in  H_i$, i.e., $\gamma H_i\neq \eta H_i$.

\item[(C3)] For any $\gamma,\eta, \gamma_1, \eta_1, \dots, \gamma_{\ell}, \eta_{\ell}\in F$ and $1\leq i_0, \dots, i_{\ell}\leq m$, if $\gamma\cdot a_{i_0}, \eta\cdot a_{i_{\ell}}\in A$ and for all $1\leq j\leq \ell$, $\gamma_j\cdot a_{i_{j-1}}, \eta_j\cdot a_{i_{j}}\in A$, and
    $$ d^A(\gamma\cdot a_{i_0}, \eta\cdot a_{i_{\ell}})>\displaystyle\sum_{j=1}^{\ell} d^A(\gamma_j\cdot a_{i_{j-1}}, \eta_j\cdot a_{i_j}), $$
    then there do not exist $g_1,\dots, g_{\ell}\in \Gamma$ such that
    $$\begin{array}{rcl}
    \gamma H_{i_0}&=&g_1 \gamma_1 H_{i_0} \\
    g_1 \eta_1 H_{i_1}&=&g_2 \gamma_2 H_{i_1}, \\
    \cdots & &\cdots \\
    g_{\ell-1} \eta_{\ell-1} H_{i_{\ell-1}}& =& g_{\ell} \gamma_{\ell} H_{i_{\ell-1}} \\
     g_{\ell} \eta_{\ell} H_{i_{\ell}}&=& \eta H_{i_{\ell}}.
    \end{array}
    $$
\item[(C4)] For any $n$-ary $R\in\mathcal{L}$, $\gamma_1,\dots, \gamma_n, \eta_1,\dots, \eta_n\in F$ and $1\leq i_1,\dots, i_n\leq m$, if for all $1\leq j\leq n$, $\gamma_j\cdot a_{i_j}, \eta_j\cdot a_{i_j}\in A$ and
    $$ R^A(\gamma_1\cdot a_{i_1},\dots, \gamma_n\cdot a_{i_n})\neq R^A(\eta_1\cdot a_{i_1},\dots, \eta_n\cdot a_{i_n}), $$
    there there does not exist $g\in \Gamma$ such that
    $$\begin{array}{rcl} \gamma_1H_{i_1}&=&g\eta_1H_{i_1} \\
    \cdots & & \cdots \\
    \gamma_n H_{i_n}& =& g\eta_nH_{i_n}.
    \end{array}
    $$
\item[(C5)] For any $n$-ary $R\in\mathcal{L}$ where $n\geq 2$, $$E=\{\epsilon_1,\dots, \epsilon_t\}\subseteq\{1, \dots, n\},$$
    $$\gamma_1, \dots, \gamma_n, \eta_1,\dots, \eta_n\in F,$$
    $$ \gamma_{1,1},\dots, \gamma_{1,\ell_1}, \eta_{1,1},\dots, \eta_{1,\ell_1}\in F, $$
    $$ \cdots\cdots $$
    $$ \gamma_{t,1},\dots, \gamma_{t, \ell_t}, \eta_{t,1}, \dots, \eta_{t,\ell_t}\in F, $$
    $$1\leq i_1, \dots, i_n, j_1,\dots, j_n \leq m,$$
    and
    $$ 1\leq i_{1,0}, i_{1,1}, \dots, i_{1,\ell_1}, \dots, i_{t,0}, i_{t,1},\dots, i_{t, \ell_t}\leq m, $$
    if for all $1\leq k\leq n$, $$\gamma_k\cdot a_{i_k}, \eta_k\cdot a_{j_k}\in A,$$
    for all $1\leq s\leq t$ and $1\leq p\leq \ell_s$,
    $$ i_{s,0}=i_{\epsilon_s}, i_{s, \ell_s}=j_{\epsilon_s}, $$
    $$\gamma_s\cdot a_{i_{\epsilon_s}}, \eta_s\cdot a_{j_{\epsilon_s}}, \gamma_{s, p}\cdot a_{i_{s,p-1}}, \eta_{s, p}\cdot a_{i_{s,p}}\in A,$$
    for all $1\leq s\leq t$,
    $$ C_s=\displaystyle\sum_{p=1}^{\ell_s} d^A(\gamma_{s,p}\cdot a_{i_{s,p-1}}, \eta_{s,p}\cdot a_{i_{s, p}})<\delta_A, $$
    and
    $$\begin{array}{c}|R^A(\gamma_1\cdot a_{i_1},\dots, \gamma_n\cdot a_{i_n})-R^A(\eta_1\cdot a_{j_1},\dots, \eta_n\cdot a_{j_n})|\ \ \ \ \ \  \ \ \ \ \  \ \ \ \ \  \ \\
    \ \ \ \ \ \ \ \ \ \ \ \ \ \ \  \ \ \ \ \ \ \ \ \ \ \ \ \ \ \ \ \ \ >\ \displaystyle\sum_{k=\epsilon_s\in E} K_{R,k}C_s+\displaystyle\sum_{k\not\in E} K_{R,k}\delta_A,
    \end{array}
    $$
    then there do not exist $g, h, g_{1,1}, \dots, g_{1,\ell_1}, \dots, g_{t,1}, \dots, g_{t, \ell_t}\in \Gamma$ such that for all $1\leq s\leq t$,
    $$\begin{array}{rcl}
    g\gamma_{\epsilon_s}H_{i_{s,0}}&=&g_{s,1}\gamma_{s,1}H_{i_{s,0}}, \\ g_{s,1}\eta_{s,1}H_{i_{s,1}}&=&g_{s,2}\gamma_{s,2}H_{i_{s,1}}, \\
     \dots & & \dots \\
    g_{s,\ell_s}\eta_{s, \ell_s}H_{i_{s, \ell_s}}&=&h\eta_{\epsilon_s}H_{i_{s, \ell_s}}.
    \end{array}
    $$

\end{enumerate}

The conditions (C1), (C2) and (C4) are easy to verify. For (C3), let $a=\gamma\cdot a_{i_0}$, $b=\eta\cdot a_{i_{\ell}}$ and for all $1\leq j\leq \ell$, let $b_j=\gamma_j\cdot a_{i_{j-1}}$ and $c_j=\eta_j\cdot a_{i_j}$. Note that all these elements are in $A$. By our assumption, we have
$$ d^A(a, b)>\displaystyle\sum_{j=1}^{\ell} d^A(b_j, c_j). $$
If there were $g_1, \dots, g_{\ell}\in \Gamma$ for which the left system of equations hold, then we have $e(a)=g_1\cdot e(b_1)$, $g_{\ell}\cdot e(c_{\ell})=e(b)$, and for all $1\leq j\leq \ell-1$,
$ g_j\cdot e(c_j)=g_{j+1}\cdot e(b_j)$. Thus
$$ \begin{array}{rcl} d^X(e(a), e(b)) &\leq & \displaystyle\sum_{j=1}^{\ell-1} d^X(g_j\cdot e(b_j), g_j\cdot e(c_j)) \\
&=& \displaystyle\sum_{j=1}^{\ell-1} d^X(e(b_j), e(c_j)) \\
&=& \displaystyle\sum_{j=1}^{\ell-1} d^A(b_j, c_j)<d^A(a, b)=d^X(e(a), e(b)),
\end{array}
$$
a contradiction.

For (C5), assume that all the hypotheses hold but there are
$$g, h, g_{1,1},\dots,g_{1,\ell_1}, \dots, g_{t,1},\dots, g_{t,\ell_1}\in \Gamma$$ such that the displayed left system of equations hold. Then for all $1\leq s\leq t$,
$$\begin{array}{rcl} g\gamma_{\epsilon_s}\cdot a_{i_{s,0}}&=&g_{s,1}\gamma_{s, 1}\cdot a_{i_{s,0}},\\
g_{s,1}\eta_{s,1}\cdot a_{i_{s,1}}& =& g_{s,2}\gamma_{s,2}\cdot a_{i_{s,1}},\\
 \dots & & \dots \\
g_{s,\ell_s}\eta_{s,\ell_s}\cdot a_{i_{s, \ell_s}}& =& h\eta{\epsilon_s}\cdot a_{i_{s, \ell_s}}.
\end{array}
$$
Thus for $k=\epsilon_s\in E$,
$$ d^M(g\gamma_k\cdot a_{i_k}, h\eta_k\cdot a_{j_k})\leq \displaystyle\sum_{p=1}^{\ell_s} d^A(\gamma_{s, p}\cdot a_{i_{s,p-1}}, \eta_{s,p}\cdot a_{i_{s,p}})=C_s<\delta_A. $$
Since $\mathcal{M}$ is a continuous $\mathcal{L}$-pre-structure and the action of $\Gamma$ on $\mathcal{M}$ is by automorphisms, we have
$$ \begin{array}{rcl}
& & |R^A(\gamma_1\cdot a_{i_1},\dots, \gamma_n\cdot a_{i_n})-R^A(\eta_1\cdot a_{j_1},\dots, \eta_n\cdot a_{j_n})| \\
&=& |R^M(g\gamma_1\cdot a_{i_1},\dots, g\gamma_n\cdot a_{i_n})-R^M(h\eta_1\cdot a_{j_1},\dots, h\eta_n\cdot a_{j_n})| \\
&\leq & \displaystyle\sum_{k=1}^n K_{R,k}d^M(g\gamma_k\cdot a_{i_k}, h\eta_k\cdot a_{j_k}) \\
&\leq & \displaystyle\sum_{k=\epsilon_s\in E} K_{R,k}C_s+\sum_{k\not\in E} K_{R,k}\delta_A,
\end{array}
$$
a contradiction.

All of the conditions (C1)--(C5) are of the form which state that certain left systems of equations do not have solutions. Together, they amount to a finite number of such left systems. By the HL-property of $\Gamma$, we obtain normal subgroups of finite index
$K_1, \dots, K_m\unlhd \Gamma$ such that the same left systems of equations on $K_1H_1, \dots, K_mH_m$ still do not have solutions. By taking $K=\bigcap_{1\leq i\leq m} K_i$ we may assume without loss of generality that $K_1=\cdots=K_m=K$. This will simplify our notation and our discussions. The fact that the left systems on $KH_1,\dots, KH_m$ do not have solutions can in turn be expressed as the following conditions.

\begin{enumerate}
\item[(K1)] For any $\gamma, \eta, \gamma', \eta'\in F$ and $1\leq i, j\leq m$, if $\gamma\cdot a_i, \gamma'\cdot a_{i}, \eta\cdot a_j, \eta'\cdot a_{j}\in A$ and
    $$ d^A(\gamma\cdot a_i, \eta\cdot a_j)\neq d^A(\gamma'\cdot a_{i}, \eta'\cdot a_{j}), $$
    then there does not exist $g\in \Gamma$ such that
    $$ g\gamma KH_i=\gamma'KH_i \mbox{ and } g\eta KH_j=\eta' KH_j. $$
\item[(K2)] For any $\gamma, \eta\in F$ and $1\leq i\leq m$, if $\gamma\cdot a_i, \eta\cdot a_i\in A$ and $\gamma\cdot a_i\neq \eta\cdot a_i$, then $\gamma^{-1}\eta\not\in  KH_i$, i.e., $\gamma KH_i\neq \eta KH_i$.

\item[(K3)] For any $\gamma,\eta, \gamma_1, \eta_1, \dots, \gamma_{\ell}, \eta_{\ell}\in F$ and $1\leq i_0, \dots, i_{\ell}\leq m$, if $\gamma\cdot a_{i_0}, \eta\cdot a_{i_{\ell}}\in A$ and for all $1\leq j\leq \ell$, $\gamma_j\cdot a_{i_{j-1}}, \eta_j\cdot a_{i_{j}}\in A$, and
    $$ d^A(\gamma\cdot a_{i_0}, \eta\cdot a_{i_{\ell}})>\displaystyle\sum_{j=1}^{\ell} d^A(\gamma_j\cdot a_{i_{j-1}}, \eta_j\cdot a_{i_j}), $$
    then there do not exist $g_1,\dots, g_{\ell}\in \Gamma$ such that
    $$\begin{array}{rcl}
    \gamma KH_{i_0}&=&g_1 \gamma_1 KH_{i_0} \\
    g_1 \eta_1 KH_{i_1}&=&g_2 \gamma_2 KH_{i_1}, \\
    \cdots & &\cdots \\
    g_{\ell-1} \eta_{\ell-1} KH_{i_{\ell-1}}& =& g_{\ell} \gamma_{\ell} KH_{i_{\ell-1}} \\
     g_{\ell} \eta_{\ell} KH_{i_{\ell}}&=& \eta KH_{i_{\ell}}.
    \end{array}
    $$
\item[(K4)] For any $n$-ary $R\in\mathcal{L}$, $\gamma_1,\dots, \gamma_n, \eta_1,\dots, \eta_n\in F$ and $1\leq i_1,\dots, i_n\leq m$, if for all $1\leq j\leq n$, $\gamma_j\cdot a_{i_j}, \eta_j\cdot a_{i_j}\in A$ and
    $$ R^A(\gamma_1\cdot a_{i_1},\dots, \gamma_n\cdot a_{i_n})\neq R^A(\eta_1\cdot a_{i_1},\dots, \eta_n\cdot a_{i_n}), $$
    there there does not exist $g\in \Gamma$ such that
    $$\begin{array}{rcl} \gamma_1KH_{i_1}&=&g\eta_1KH_{i_1} \\
    \cdots & & \cdots \\
    \gamma_n KH_{i_n}& =& g\eta_nKH_{i_n}.
    \end{array}
    $$
\item[(K5)] For any $n$-ary $R\in\mathcal{L}$ where $n\geq 2$, $$E=\{\epsilon_1,\dots, \epsilon_t\}\subseteq\{1, \dots, n\},$$
    $$\gamma_1, \dots, \gamma_n, \eta_1,\dots, \eta_n\in F,$$
    $$ \gamma_{1,1},\dots, \gamma_{1,\ell_1}, \eta_{1,1},\dots, \eta_{1,\ell_1}\in F, $$
    $$ \cdots\cdots $$
    $$ \gamma_{t,1},\dots, \gamma_{t, \ell_t}, \eta_{t,1}, \dots, \eta_{t,\ell_t}\in F, $$
    $$1\leq i_1, \dots, i_n, j_1,\dots, j_n \leq m,$$
    and
    $$ 1\leq i_{1,0}, i_{1,1}, \dots, i_{1,\ell_1}, \dots, i_{t,0}, i_{t,1},\dots, i_{t, \ell_t}\leq m, $$
    if for all $1\leq k\leq n$, $$\gamma_k\cdot a_{i_k}, \eta_k\cdot a_{j_k}\in A,$$
    for all $1\leq s\leq t$ and $1\leq p\leq \ell_s$,
    $$ i_{s,0}=i_{\epsilon_s}, i_{s, \ell_s}=j_{\epsilon_s}, $$
    $$\gamma_s\cdot a_{i_{\epsilon_s}}, \eta_s\cdot a_{j_{\epsilon_s}}, \gamma_{s, p}\cdot a_{i_{s,p-1}}, \eta_{s, p}\cdot a_{i_{s,p}}\in A,$$
    for all $1\leq s\leq t$,
    $$ C_s=\displaystyle\sum_{p=1}^{\ell_s} d^A(\gamma_{s,p}\cdot a_{i_{s,p-1}}, \eta_{s,p}\cdot a_{i_{s, p}})<\delta_A, $$
    and
    $$\begin{array}{c}|R^A(\gamma_1\cdot a_{i_1},\dots, \gamma_n\cdot a_{i_n})-R^A(\eta_1\cdot a_{j_1},\dots, \eta_n\cdot a_{j_n})|\ \ \ \ \ \  \ \ \ \ \  \ \ \ \ \  \ \\
    \ \ \ \ \ \ \ \ \ \ \ \ \ \ \  \ \ \ \ \ \ \ \ \ \ \ \ \ \ \ \ \ \ >\ \displaystyle\sum_{k=\epsilon_s\in E} K_{R,k}C_s+\displaystyle\sum_{k\not\in E} K_{R,k}\delta_A,
    \end{array}
    $$
    then there do not exist $g, h, g_{1,1}, \dots, g_{1,\ell_1}, \dots, g_{t,1}, \dots, g_{t, \ell_t}\in \Gamma$ such that for all $1\leq s\leq t$,
    $$\begin{array}{rcl}
    g\gamma_{\epsilon_s}KH_{i_{s,0}}&=&g_{s,1}\gamma_{s,1}KH_{i_{s,0}}, \\ g_{s,1}\eta_{s,1}KH_{i_{s,1}}&=&g_{s,2}\gamma_{s,2}KH_{i_{s,1}}, \\
     \dots & & \dots \\
    g_{s,\ell_s}\eta_{s, \ell_s}KH_{i_{s, \ell_s}}&=&h\eta_{\epsilon_s}KH_{i_{s, \ell_s}}.
    \end{array}
    $$
\end{enumerate}

Let $Y$ be the disjoint union of the left-coset spaces
$$ Y=\Gamma/KH_1\sqcup\cdots \sqcup\Gamma/KH_m. $$
Let $\Gamma$ act on $Y$ by left multiplication.

We define a metric $d^Y$ on $Y$ as follows. For $\gamma, \eta\in \Gamma$ and $1\leq i, j\leq m$,
$$d^Y(\gamma KH_i, \eta KH_j)=\min\bigl\{\delta_A, \inf\{ p_1+\cdots+p_{\ell}\,:\, (p_1,\dots, p_\ell)\in S_{\gamma, \eta, i, j}\}\bigr\}, $$
where $(p_1, \dots, p_\ell)\in S_{\gamma, \eta, i, j}$ if there are $$g_1, \dots, g_\ell \in \Gamma, $$
$$ \gamma_1, \eta_1, \dots, \gamma_{\ell}, \eta_{\ell}\in F, $$
and
$$ 1\leq i_0, \dots, i_\ell\leq m $$
such that
$$ i=i_0, j=i_{\ell}, $$
for all $1\leq k\leq \ell$,
$$ \gamma_k\cdot a_{i_{k-1}}, \eta_k\cdot a_{i_k}\in A, $$
$$ p_k=d^A(\gamma_k\cdot a_{i_{k-1}}, \eta_k\cdot a_{i_k}), $$
and
$$\begin{array}{rcl}
    \gamma KH_{i_0}&=&g_1 \gamma_1 KH_{i_0} \\
    g_1 \eta_1 KH_{i_1}&=&g_2 \gamma_2 KH_{i_1}, \\
    \cdots & &\cdots \\
    g_{\ell-1} \eta_{\ell-1} KH_{i_{\ell-1}}& =& g_{\ell} \gamma_{\ell} KH_{i_{\ell-1}} \\
     g_{\ell} \eta_{\ell} KH_{i_{\ell}}&=& \eta KH_{i_{\ell}}.
    \end{array}
    $$

It is clear that $d^Y$ is invariant under the action of $\Gamma$.

\begin{lemma}\label{lem:dN} $d^Y$ is a metric on $Y$.
\end{lemma}

\begin{proof} We first verify that if $i\neq j$ or ($i=j$ but $\gamma^{-1}\eta\not\in KH_i$) then
$$d^Y(\gamma KH_i, \eta KH_j)>0.$$
If $S_{\gamma, \eta, i, j}=\varnothing$ then $d^Y(\gamma KH_i, \eta KH_j)=\delta_A>0$. So suppose $S_{\gamma, \eta, i, j}\neq \varnothing$. Note that, however, since $F$ and $A$ are finite, the number of elements $(p_1, \dots, p_{\ell})\in S_{\gamma, \eta, i, j}$ such that $p_1, \dots, p_\ell\neq 0$ and $p_1+\cdots+p_{\ell}\leq \delta_A$ is finite. Now suppose $i\neq j$. Then since $\Gamma a_i\cap \Gamma a_j=\varnothing$, for any $(p_1, \dots, p_\ell)\in S_{\gamma, \eta, i, j}$, $p_1+\cdots+p_{\ell}>0$ because some term is positive. Thus $d^Y(\gamma KH_i, \eta KH_j)>0$. Next suppose $i=j$ but $\gamma^{-1}\eta\not\in KH_i$. In this case if some $i_k\neq i$ for $1\leq k\leq\ell$, then by the same argument $p_1+\dots+p_\ell>0$. Otherwise, $i_k=i$ for all $1\leq k\leq \ell$. In this case
$$ p_1+\cdots+p_{\ell}=\displaystyle\sum_{j=1}^{\ell}d^A(\gamma_j\cdot a_i, \eta_j\cdot a_i). $$
Again, if some of them are positive, then we are done. If all of them are zero, then $\gamma_j^{-1}\eta_j\in H_i$ for all $1\leq j\leq \ell$. In particular $\gamma_j KH_i=\eta_j KH_i$ for all $1\leq j\leq \ell$. By the defining relation for $(p_1,\dots, p_\ell)\in S_{\gamma,\eta, i, j}$ we get
$$ \begin{array}{rcl}
    \gamma KH_{i}&=&g_1 \gamma_1 KH_{i} \\
    g_1 \eta_1 KH_{i}&=&g_2 \gamma_2 KH_{i}, \\
    \cdots & &\cdots \\
    g_{\ell-1} \eta_{\ell-1} KH_{i}& =& g_{\ell} \gamma_{\ell} KH_{i} \\
     g_{\ell} \eta_{\ell} KH_{i}&=& \eta KH_{i}.
    \end{array}
    $$
But now all the terms of these equations are equal. So we get
$\gamma KH_i=\eta KH_i$.

Next we verify the triangle inequality. Suppose
$$ d^Y(\gamma KH_i, \eta KH_j)=p_1+\cdots+p_{\ell} $$
and
$$ d^Y(\eta KH_j, \mu KH_k)=q_1+\cdots+q_{t} $$
for $(p_1,\dots, p_\ell)\in S_{\gamma, \eta, i, j}$ and $(q_1, \dots, q_t)\in S_{\eta, \mu, j, k}$. Concatenating the defining relations for $(p_1,\dots, p_\ell)$ and $(q_1,\dots, q_t)$, we obtain a left system for $(p_1,\dots, p_{\ell}, q_1,\dots, q_t)\in S_{\gamma, \mu, i, k}$. Thus $$d^Y(\gamma KH_i, \mu KH_k)\leq d^Y(\gamma KH_i, \eta KH_j)+d^Y(\eta KH_j, \mu KH_k). $$
\end{proof}

\begin{lemma}\label{lem:dYdA} For any $\gamma, \eta\in F$ and $1\leq i,j\leq m$, if $\gamma\cdot a_i, \eta\cdot a_j\in A$, then
$$ d^A(\gamma\cdot a_i,\eta\cdot a_j)= d^Y(\gamma KH_i, \eta KH_j). $$
\end{lemma}

\begin{proof} Suppose
$$ d^Y(\gamma KH_i,\eta KH_j)=p_1+\cdots+p_{\ell} $$
where $(p_1,\dots, p_\ell)\in S_{\gamma, \eta, i, j}$ is witnessed by
$g_1,\dots, g_\ell\in \Gamma$, $\gamma_1,\eta_1,\dots, \gamma_\ell, \eta_\ell\in F$ and $1\leq i_0, \dots, i_\ell\leq m$. Then by the defining left system of equations for $(p_1,\dots, p_\ell)\in S_{\gamma, \eta, i,j}$ and (K3), we have
$$ d^A(\gamma\cdot a_i, \eta\cdot a_j)\leq p_1+\dots+p_\ell=d^Y(\gamma KH_i, \eta KH_j). $$
Conversely, from the trivial system of equations
$$\begin{array}{c} \gamma KH_i=\gamma KH_i \\ \eta KH_j=\eta KH_j \end{array}
$$
we get that
$$ d^Y(\gamma KH_i, \eta KH_j)\leq d^A(\gamma\cdot a_i, \eta\cdot a_j). $$
\end{proof}

Next we define $R^Y$ for an $n$-ary $R\in\mathcal{L}$ on
$$ \dom(R^Y)=\{(\gamma\gamma_1KH_{i_1}, \dots, \gamma\gamma_nH_{i_n})\,:\, \gamma\in \Gamma, \gamma_j\in F, \gamma_j\cdot a_{i_j}\in A \mbox{ for all $1\leq j\leq n$}\} $$
by
$$ R^Y(\gamma\gamma_1KH_{i_1}, \dots, \gamma\gamma_nH_{i_n})=R^A(\gamma_1\cdot a_{i_1},\dots, \gamma_n\cdot a_{i_n}). $$
By (K4), this is well defined. It is clear that $\dom(R^Y)$ and $R^Y$ are invariant under the action of $\Gamma$.

\begin{lemma} $\mathcal{Y}=(Y, d^Y, (R^Y)_{R\in\mathcal{L}})$ is a partially defined continuous $\mathcal{L}$-pre-structure.
\end{lemma}

\begin{proof} Suppose first $R$ is unary. We verify (\ref{eqn:ucl}) for $R$. Let $\gamma, \eta\in \Gamma$ and $1\leq i,j\leq m$. Assume
$$ d^Y(\gamma KH_i, \eta KH_j)=p_1+\cdots+p_\ell $$
where $(p_1,\dots, p_\ell)\in S_{\gamma, \eta, i, j}$ is witnessed by $g_1, \dots, g_\ell\in \Gamma$, $\gamma_1,\eta_1,\dots, \gamma_\ell, \eta_\ell\in F$, and $1\leq i_0, \dots, i_{\ell}\leq m$. For $1\leq k\leq \ell$, let $b_k=\gamma_k\cdot a_{i_{k-1}}\in A$ and $c_k=\eta_k\cdot a_{i_k}\in A$. Then
$$R^Y(\gamma KH_i)=R^A(a_i)=R^A(\gamma_1\cdot a_i)=R^A(b_1), $$
$$ R^Y(\eta KH_j)=R^A(a_j)=R^A(\eta_\ell\cdot a_j)=R^A(c_\ell),$$
and
$$d^Y(\gamma KH_i, \eta KH_j)= p_1+\cdots+p_\ell =\displaystyle\sum_{k=1}^{\ell} d^A(b_k, c_k). $$
We may assume $d^Y(\gamma KH_i, \eta KH_j)\in I_{R,1}$; otherwise the desired inequality trivially holds. Now by the defining equations of $S_{\gamma, \eta, i, j}$, we have
$$\begin{array}{rcl}
|R^Y(\gamma KH_i)-R^Y(\eta KH_j)|
&\leq & \displaystyle\sum_{k=1}^{\ell} |R^Y(\gamma_kKH_{i_{k-1}})-R^Y(\eta_kKH_{i_k})| \\
&= & \displaystyle\sum_{k=1}^{\ell} |R^A(b_k)-R^A(c_k)| \\
&\leq& \displaystyle\sum_{k=1}^{\ell} \mathfrak{u}_{R,1}(d^A(b_k, c_k)) \\
&\leq& \mathfrak{u}_{R,1}(d^Y(\gamma KH_i, \eta KH_j)).
\end{array}
$$
Now it follows that
$$ |R^Y(\gamma KH_i)-R^Y(\eta KH_j)|\leq d_R^Y(\gamma KH_i, \eta KH_j). $$

Next suppose $R$ is $n$-ary for $n\geq 2$. Let $\gamma, \eta\in\Gamma$, $\gamma_1, \dots, \gamma_n, \eta_1, \dots, \eta_n\in F$ and $1\leq i_1,\dots, i_n, j_1, \dots, j_n\leq m$. Suppose $\gamma_k\cdot a_{i_k}, \eta_k\cdot a_{j_k}\in A$ for all $1\leq k\leq n$. We need to show that
$$\begin{array}{c} |R^Y(\gamma\gamma_1KH_{i_1}, \dots, \gamma\gamma_nKH_{i_n})-R^Y(\eta\eta_1KH_{j_1},\dots, \eta\eta_nKH_{j_n})|
 \ \ \ \ \ \ \ \ \ \ \ \  \ \ \\
\ \ \ \ \ \  \ \ \ \ \  \ \ \ \  \ \ \ \ \ \ \  \ \ \ \ \ \ \ \ \ \ \ \ \ \ \ \ \ \ \ \ \ \ \ \ \ \ \ \leq\ \displaystyle\sum_{k=1}^n K_{R,k}d^Y(\gamma\gamma_kKH_{i_k}, \eta\eta_kKH_{j_k}).
\end{array} $$
Let $E$ be the set of $k\in\{1, \dots, n\}$ such that
$$ d^Y(\gamma\gamma_kKH_{i_k}, \eta\eta_kKH_{j_k})<\delta_A. $$
Enumerate the elements of $E$ as $\epsilon_1,\dots, \epsilon_t$. Then for each $1\leq s\leq t$, by the definition of $d^Y$, there are
$$g_{s,1},\dots, g_{s,\ell_s}\in \Gamma,$$
$$\gamma_{s,1},\dots, \gamma_{s, \ell_s}, \eta_{s,1}, \dots, \eta_{s, \ell_s}\in F, $$
and
$$ 1\leq i_{s, 0}, i_{s,1},\dots, i_{s, \ell_s}\leq m $$
such that
$$ i_{\epsilon_s}=i_{s,0},\ j_{\epsilon_s}=i_{s, \ell_s}, $$
for all $1\leq p\leq \ell_s$,
$$ \gamma_{s, p}\cdot a_{i_{s,p-1}}, \eta_{s,p}\cdot a_{i_{s,p}}\in A, $$
$$ d^Y(\gamma\gamma_{\epsilon_s}KH_{i_{\epsilon_s}},\eta\eta_{\epsilon_s}KH_{j_{\epsilon_s}})
=\displaystyle\sum_{p=1}^{\ell_s} d^A(\gamma_{s,p}\cdot a_{i_{s,p-1}}, \eta_{s,p}\cdot a_{i_{s,p}})=C_s, $$
and
$$\begin{array}{rcl}
    \gamma\gamma_{\epsilon_s}KH_{i_{s,0}}&=&g_{s,1}\gamma_{s,1}KH_{i_{s,0}}, \\ g_{s,1}\eta_{s,1}KH_{i_{s,1}}&=&g_{s,2}\gamma_{s,2}KH_{i_{s,1}}, \\
     \dots & & \dots \\
    g_{s,\ell_s}\eta_{s, \ell_s}KH_{i_{s, \ell_s}}&=&\eta\eta_{\epsilon_s}KH_{i_{s, \ell_s}}.
    \end{array}
    $$
For $k\not\in E$,
$$ d^Y(\gamma\gamma_kKH_{i_k}, \eta\eta_kKH_{j_k})=\delta_A. $$
Thus
$$
\displaystyle\sum_{k=1}^n K_{R,k}d^Y(\gamma\gamma_kKH_{i_k}, \eta\eta_kKH_{j_k})=
\displaystyle\sum_{k=\epsilon_s\in E} K_{R,k}C_s+\displaystyle\sum_{k\not\in E} K_{R,k}\delta_A.
$$By (K5) we have
$$\begin{array}{l} |R^A(\gamma_1\cdot a_{i_1}, \dots, \gamma_n\cdot a_{i_n})-R^A(\eta_1\cdot a_{j_1},\dots, \eta_1\cdot a_{j_n})| \\
\leq \displaystyle\sum_{k=\epsilon_s\in E} K_{R,k}C_s+\displaystyle\sum_{k\not\in E} K_{R,k}\delta_A \\
= \displaystyle\sum_{k=1}^n K_{R,k}d^Y(\gamma\gamma_kKH_{i_k}, \eta\eta_kKH_{j_k}).
\end{array}
$$
However, note that
$$ \begin{array}{rcl}
& & |R^Y(\gamma\gamma_1KH_{i_1}, \dots, \gamma\gamma_nKH_{i_n})-R^Y(\eta\eta_1KH_{j_1},\dots, \eta\eta_nKH_{j_n})| \\
&=& |R^Y(\gamma_1KH_{i_1}, \dots, \gamma_nKH_{i_n})-R^Y(\eta_1KH_{j_1},\dots, \eta_nKH_{j_n})| \\
&=& |R^A(\gamma_1\cdot a_{i_1}, \dots, \gamma_n\cdot a_{i_n})-R^A(\eta_1\cdot a_{j_1},\dots, \eta_1\cdot a_{j_n})|.
\end{array}
$$
We thus obtained the desired inequality.
\end{proof}

Let $\mathcal{N}$ be the conservative extension of $\mathcal{Y}$ given by Lemma~\ref{lem:gconservative}. In particular, for every $n$-ary $R\in \mathcal{L}$ and $\overline{x}\in N^n$,
$$ R^N(\overline{x})=\max\{0,\sup\{R^Y(\overline{y})-d^Y(\overline{x},\overline{y})\,:\, \overline{y}\in \dom(R^Y)\}\}. $$
It is clear that $R^N$ is invariant under the action of $\Gamma$.

Define $f: A \to N$ by
$$ f(a)=\left\{\begin{array}{ll} KH_i & \mbox{ if $a=a_i$} \\
\gamma KH_i & \mbox{ if $a\neq a_i$ and $a=\gamma\cdot a_i$ for $\gamma\in F$.}
\end{array}\right.
$$
It is easy to see that $f$ is well defined. We claim that $f$ is an isomorphic embedding from $\mathcal{A}$ into $\mathcal{N}$. For injectivity, let $\gamma, \eta\in F$ such that $\gamma\cdot a_i, \eta\cdot a_i\in A$ but $\gamma\cdot a_i\neq \eta\cdot a_i$. Then by (K2), $\gamma KH_i\neq \eta KH_i$. Isometry follows from Lemma~\ref{lem:dYdA}. The preservation of $R$-value follows directly from the definition of $R^Y$.

Finally, we verify that for all $g\in F$ and $a\in A$, if $g\cdot a\in A$ then
$$ f(g\cdot a)=g f(a). $$
Assume $a=\gamma \cdot a_i$ and $g\cdot a=\eta\cdot a_i$ for $\gamma, \eta\in F$ and $1\leq i\leq m$. Then
$$\rho(\eta^{-1})(\rho(g)(\rho(\gamma)(a_i))))=a_i.$$
Thus $\eta^{-1}g\gamma\in H_i\subseteq KH_i$, and
$$ f(g\cdot a)=f(\eta\cdot a_i)=\eta KH_i=g\gamma KH_i=gf(a). $$

This completes the proof of Theorem~\ref{thm:Rosendal}.

We remark that the same method can be used to give a direct proof of Theorem~\ref{thm:EPPA} (also c.f. \cite{EG0}).

\subsection{The Fra\"iss\'e class of actions by automorphisms}
In this final subsection we prove that for any semiproper continuous signature $\mathcal{L}$ the class $\mathcal{C}_{\mathcal{L}}$ defined in Definition~\ref{def:C} is a Fra\"iss\'e class. The argument follows closely the proof of Theorems 3.9 and 4.8 of \cite{EGLMM}. We give some details for the convenience of the reader. The key fact is the following theorem.

\begin{theorem} \label{thm:FraAP} Let $\mathcal{L}$ be a semiproper continuous signature. Then $\mathcal{C}_{\mathcal{L}}$ has the AP.
\end{theorem}

\begin{proof} Assume $\Gamma_1$ and $\Gamma_2$ are finite groups, and $\Lambda$ is a subgroup of both $\Gamma_1$ and $\Gamma_2$. Assume $\mathcal{M}_1$ and $\mathcal{M}_2$ are finite continuous $\mathcal{L}$-structures, and $\mathcal{A}$ a substructure of both $\mathcal{M}_1$ and $\mathcal{M}_2$. Assume $\pi_1:\Gamma_1\actson \mathcal{M}_1$, $\pi_2:\Gamma_2\actson \mathcal{M}_2$ and $\tau:\Lambda\actson \mathcal{A}$. Furthermore, assume that
$$ \tau=\pi_1\rest (\Lambda\times A)=\pi_2\rest (\Lambda\times A). $$
Let $\Gamma=\Gamma_1*_{\Lambda}\Gamma_2$. By Proposition 4.4 of \cite{EGLMM}, $\Gamma$ has the HL-property.

Since by Theorem~\ref{thm:finap}, $\Kfin$ has the SAP, we may obtain a strong amalgam of $\mathcal{M}_1$ and $\mathcal{M}_2$ over $\mathcal{A}$, which we denote by $\mathcal{N}_0$, such that $\pi_1$, $\pi_2$, $\tau$ are actions respectively on $\mathcal{M}_1$, $\mathcal{M}_2$, $\mathcal{A}$ as substructures of $\mathcal{N}$, still satisfying the compatibility condition above. Moreover, by the proof of Theorem~\ref{thm:finap}, we may assume $N_0=M_1\cup M_2$ and $M_1\cap M_2=A$. Now $\Lambda$ as a subgroup of both $\Gamma_1$ and $\Gamma_2$ acts on $\mathcal{N}_0$ in a natural way, and this action is still by automorphisms.

We now inductively define a sequence of extensions
$$ \mathcal{N}_1\subseteq \mathcal{N}_2\subseteq \cdots $$
of $\mathcal{N}_0$. To define $\mathcal{N}_1$, apply Lemma~\ref{lem:Fraext} to the actions $\Lambda\actson \mathcal{N}_0$ and $\Gamma_1\actson \mathcal{M}_1$. We obtain $\mathcal{N}_1$ and an action $\Gamma_1\actson \mathcal{N}_1$ that is compatible with the actions $\Lambda\actson\mathcal{N}_0$ and $\Gamma_1\actson M_1$. This action also induces an action $\Lambda\actson \mathcal{N}_1$. To define $\mathcal{N}_2$, we apply Lemma~\ref{lem:Fraext} to this induced action $\Lambda\actson \mathcal{N}_1$ and $\Gamma_2\actson \mathcal{M}_2$. The result is an extension $\mathcal{N}_2$ together with an action $\Gamma_2\actson \mathcal{N}_2$ that is compatible with the actions $\Lambda\actson \mathcal{N}_1$ and $\Gamma_2\actson \mathcal{M}_2$. In general, suppose $\mathcal{N}_{2k}$ has been defined with an action $\Gamma_2\actson \mathcal{N}_{2k}$ and $\Gamma_1\actson N_{2k-1}$. We apply Lemma~\ref{lem:Fraext} to the induced action $\Lambda\actson \mathcal{N}_{2k}$ and the action $\Gamma_1\actson N_{2k-1}$ to obtain $\mathcal{N}_{2k+1}$ with an action $\Gamma_1\actson \mathcal{N}_{2k+1}$ that is compatible with the two actions mentioned before. Similarly, apply Lemma~\ref{lem:Fraext} to the induced action $\Lambda\actson \mathcal{N}_{2k+1}$ and the action $\Gamma_2\actson \mathcal{N}_{2k}$ to obtain $\mathcal{N}_{2k+2}$ and action $\Gamma_2\actson \mathcal{N}_{2k+2}$.

Let $\mathcal{N}_\infty=\bigcup_n \mathcal{N}_n$. Then $\Gamma=\Gamma_1*_\Lambda\Gamma_2$ acts on $\mathcal{N}_\infty$ naturally, and this action is by automorphisms. We denote this action by $\pi$.

Now applying Theorem~\ref{thm:Rosendal} to $\Gamma\actson \mathcal{N}_\infty$, $\mathcal{N}_0$ and $\Gamma_1\cup\Gamma_2$, we obtain a finite continuous $\mathcal{L}$-structure $\mathcal{N}$ and an action by automorphisms $\pi':\Gamma\actson \mathcal{N}$ such that $\mathcal{N}$ is an extension of $\mathcal{N}_0$ and for all $\gamma\in \Gamma_1\cup\Gamma_2$ and $x\in \mathcal{N}_0$, if $\gamma\cdot_\pi x\in A$, then $\gamma\cdot_\pi x=\gamma\cdot_{\pi'} x$.

Let $G$ be the subgroup of $\Aut(\mathcal{N})$ generated by $\Gamma_1\cup\Gamma_2$. Since $\mathcal{N}$ is finite, so is $G$. Now the action $G\actson \mathcal{N}$ is an amalgam of $\Gamma_1\actson \mathcal{M}_1$ and $\Gamma_2\actson\mathcal{M}_2$ over $\Lambda\actson \mathcal{A}$.
\end{proof}

The following theorem implies the joint embedding property (JEP) for $\mathcal{C}_{\mathcal{L}}$.

\begin{theorem}\label{thm:FraJEP} Let $\mathcal{L}$ be a semiproper continuous signature. Let $\Gamma, \Lambda$ be groups, let $\mathcal{M}$ and $\mathcal{N}$ be continuous $\mathcal{L}$-pre-structures, and $\Gamma\actson \mathcal{M}$ and $\mathcal{N}\actson \mathcal{N}$ are actions by automorphisms. Suppose $\diam(M), \diam(N)<\infty$. Then there exists a continuous $\mathcal{L}$-pre-structure $\mathcal{P}$, an action by automorphisms $\Gamma\times\Lambda\actson \mathcal{P}$  and embeddings $i$ from $\Gamma\actson\mathcal{M}$ into $\Gamma\times\Lambda\actson\mathcal{P}$ and $j$ from $\Lambda\actson\mathcal{N}$ into $\Gamma\times\Lambda\actson\mathcal{P}$. In particular, the JEP holds for $\mathcal{C}_{\mathcal{L}}$.
\end{theorem}

\begin{proof} Let $X$ be the disjoint union of $M$ and $N$. Let $\delta\geq \diam(M), \diam(N)$ and for all $n$-ary $R\in\mathcal{L}$ and all $1\leq i\leq n$, $\delta\geq \sup I_{R,i}$. Then define a metric $d^X$ on $X$ by
$$ d^X(x, y)=\left\{\begin{array}{ll} d^{M}(x, y) & \mbox{ if $x, y\in M$} \\
d^N(x,y) & \mbox{ if $x, y\in N$} \\ \delta & \mbox{ otherwise.}
\end{array}\right.
$$
For every $R\in \mathcal{L}$, $R^X$ is naturally defined on
$$\dom(R^X)=M^n\cup N^n. $$
It is clear that $\mathcal{X}=(X, d^X, (R^X)_{R\in \mathcal{L}})$ is a partially defined continuous $\mathcal{L}$-pre-structure. Then by Lemma~\ref{lem:gconservative}, $\mathcal{X}$ has a conservative extension $\mathcal{P}$. It is easy to see that $\mathcal{M}$ and $\mathcal{N}$ embed into $\mathcal{P}$ as substructures.

Define an action $\Gamma\times\Lambda\actson \mathcal{P}$ by
$$ (\gamma, \lambda)\cdot x=\left\{\begin{array}{ll}\gamma(x) & \mbox{ if $x\in M$} \\
\lambda(x) & \mbox{ if $x\in N$.}\end{array}\right. $$
It is easy to see that this action is by automorphisms on $\mathcal{P}$. Also,  $\Gamma\actson \mathcal{M}$ and $\Lambda\actson \mathcal{N}$ embed into $\Gamma\times\Lambda\actson \mathcal{P}$.
\end{proof}

\begin{corollary}\label{cor:Fra} Let $\mathcal{L}$ be a semiproper continuous signature. Then $\mathcal{C}_{\mathcal{L}}$ is a Fra\"iss\'e class.
\end{corollary}

\begin{proof} Noting that the HP is obvious, this follows immediately from Theorems~~\ref{thm:FraAP} and \ref{thm:FraJEP}.
\end{proof}

\section{Dense Locally Finite Subgroups of the Automorphism Groups}
In this last section of the paper we prove that Hall's universal locally finite group $\mathbb{H}$ can be embedded as a dense subgroup of $\Aut(\mathbb{U}_{\mathcal{L}})$ for proper $\mathcal{L}$.

P. Hall first constructed $\mathbb{H}$ in \cite{Hall} and proved that it is the unique countable locally finite group that is universal for all finite groups and is ultrahomogeneous. He also showed that $\mathbb{H}$ is universal for all countable locally finite groups. It is well known that $\mathbb{H}$ is the Fra\"iss\'e limit of the countable Fra\"iss\'e class of all finite groups. It was shown in \cite{EGLMM} that $\mathbb{H}$ embeds as a dense subgroup of the isometry group of $\mathbb{U}_\Delta$ for any countable distance value set $\Delta$; the same is true for the isometry group of $\mathbb{U}$ (Theorems 3.14 and 3.16 of \cite{EGLMM}).

Here we consider first a semiproper continuous signature $\mathcal{L}$ and a countable good value pair $(\Delta, V)$ for $\mathcal{L}$. By Theorem~\ref{thm:Fraisee}, the class $\mathcal{K}_{(\Delta,V)}$ of all finite $(\Delta, V)$-valued continuous $\mathcal{L}$-structures is a countable Fra\"iss\'e class. We consider the subclass of $\mathcal{C}_{\mathcal{L}}$ defined as
$$\begin{array}{rcl} \mathcal{C}_{(\Delta,V)}&=&\mbox{ the subclass of $\mathcal{C}_{\mathcal{L}}$ consisting of actions by automorphisms $\Gamma\actson \mathcal{M}$} \\
 & & \mbox{ where $\Gamma$ is finite and $\mathcal{M}\in\mathcal{K}_{(\Delta, V)}$.}
 \end{array}
$$
The entire Section 7 can be repeated for $\mathcal{C}_{(\Delta,V)}$ to give the following result.

\begin{theorem}\label{thm:cFra} Let $\mathcal{L}$ be a semiproper continuous signature and let $(\Delta, V)$ be a countable good value pair for $\mathcal{L}$. Then $\mathcal{C}_{(\Delta, V)}$ is a countable Fra\"iss\'e class.
\end{theorem}

\begin{proof} The proof is identical to the proof of Corollary~\ref{cor:Fra}.
\end{proof}

Let $\Gamma_\infty\actson \mathcal{M}_\infty$ be the Fra\"iss\'e limit of $\mathcal{C}_{(\Delta, V)}$. Then $\mathcal{M}_\infty$ is a countable $(\Delta, V)$-valued continuous $\mathcal{L}$-pre-structure, $\Gamma_\infty$ is a locally finite group, and the action is by automorphisms. By the arguments for Theorem 3.14 of \cite{EGLMM}, we prove the following lemmas.

\begin{lemma} $\mathcal{M}_\infty$ is isomorphic to $\mathbb{U}_{(\Delta,V)}$.
\end{lemma}

\begin{proof} $\mathbb{U}_{(\Delta, V)}$ is is the unique countable continuous $\mathcal{L}$-pre-structure with the {\em $(\Delta, V)$-Urysohn property}: given any finite $(\Delta, V)$-valued continuous $\mathcal{L}$-structure $\mathcal{M}$, a one-point extension $\mathcal{N}$ of $\mathcal{M}$ that is $(\Delta, V)$-valued, and an isomorphic embedding $\varphi$ from $\mathcal{M}$ into $\mathbb{U}_{(\Delta, V)}$, there is an isomorphic embedding $\psi$ from $\mathcal{N}$ into $\mathbb{U}_{(\Delta,V)}$ such that $\psi\,\rest M=\varphi$.

We verify that $\mathcal{M}_\infty$ has this property. It is clear that $\mathcal{M}_\infty$ is $(\Delta, V)$-valued. Suppose $A\subseteq \mathcal{M}_\infty$ and let $\mathcal{A}_x$ be a $(\Delta, V)$-valued one-point extension of $\mathcal{A}$. Let $\Gamma$ be the trivial group. Then $\Gamma\actson \mathcal{A}$ embeds into $\Gamma\actson \mathcal{A}_x$ trivially. By the universality and ultrahomogeneity of $\Gamma_\infty\actson \mathcal{M}_\infty$, we obtain an embedding $\varphi$ of $\Gamma\actson \mathcal{A}_x$ into $\Gamma_\infty\actson \mathcal{M}_\infty$ such that $\varphi\rest M$ is the identity. Thus $\varphi$ witnesses the the $(\Delta, V)$-Urysohn property of $\mathcal{M}_\infty$.
\end{proof}

\begin{lemma} $\Gamma_\infty$ acts faithfully on $\mathcal{M}_\infty$.
\end{lemma}

\begin{proof} Let $g\in \Gamma_\infty$ be a non-identity element. Let $\Gamma$ be the subgroup of $\Gamma_\infty$ generated by $g$. Since $\Gamma_\infty$ is locally finite, $\Gamma$ is finite. Consider the finite continuous $\mathcal{L}$-structure $\mathcal{M}$ defined as follows. Let $\delta\in \Delta$. Let $M=\Gamma$. Define $d^M(x,y)=\delta$ for any $x\neq y\in M$. For any $R\in\mathcal{L}$, define $R^M$ to be identically $0$. Then $\mathcal{M}$ is a finite $(\Delta, V)$-valued continuous $\mathcal{L}$-structure and the left multiplication action $\Gamma\actson \mathcal{M}$ is an action by automorphisms. By the universality and ultrahomogeneity of $\Gamma_\infty\actson \mathcal{M}_\infty$, we know that $\mathcal{M}$ can be realized as a substructure of $\mathcal{M}_\infty$. If follows that there is $x\in \mathcal{M}_\infty$ such that $g\cdot x\neq x$.
\end{proof}

\begin{lemma} $\Gamma_\infty$ is isomorphic to $\mathbb{H}$.
\end{lemma}

\begin{proof} $\mathbb{H}$ is the unique countable locally finite group with the following extension property: given any finite groups $G\leq H$ and a group isomorphic embedding $\varphi$ from $G$ into $\mathbb{H}$, there is a group isomorphic embedding $\psi$ from $H$ into $\mathbb{H}$ such that $\psi\,\rest G=\varphi$.

We verify that $\Gamma_\infty$ has this property. Assume $G, H, \varphi$ are given. Consider the trivial structure $\mathcal{M}$ which is a singleton and such that all relations symbols are interpreted as identically $0$ functions. Then by the universality and ultrahomogeneity of $\Gamma_\infty\actson \mathcal{M}_\infty$ there is an embedding $\psi$ of $H\actson \mathcal{M}$ into $\Gamma_\infty\actson \mathcal{M}_\infty$ such that $\psi\,\rest G=\varphi$. This proves the extension property as required.
\end{proof}

Since $\mathcal{M}_\infty$ is countable, we equip $\mathcal{M}_\infty$ with the discrete topology and $\Aut(\mathcal{M}_\infty)$ with the corresponding pointwise convergence topology. Thus $\Aut(\mathcal{M}_\infty)$ becomes a Polish group. In fact, it is isomorphic to a closed subgroup of $S_\infty$, the infinite permutation group.

\begin{lemma} $\Gamma_\infty$ is dense in $\Aut(\mathcal{M}_\infty)$.
\end{lemma}

\begin{proof} Let $g\in \Aut(\mathcal{M}_\infty)$ and let $A\subseteq M_\infty$ be finite. We need to find a $\gamma\in \Gamma_\infty$ such that for every $a\in A$, $g(a)=\gamma\cdot a$. Let $\Gamma$ be the subgroup of $\Aut(\mathcal{M}_\infty)$ generated by $g$.

Suppose first $\Gamma$ is finite. Then let $B=\Gamma(A)$. $B$ is finite, and $\Gamma$ acts on $\mathcal{B}$ by automorphisms by definition. By the universality and ultrahomegeneity of $\Gamma_\infty\actson \mathcal{M}_\infty$, we may realize $\Gamma$ as a subgroup of $\Gamma_\infty$. This gives a $\gamma\in \Gamma_\infty$ such that for every $a\in A$, $g(a)=\gamma\cdot a$.

Next suppose $\Gamma$ is infinite. Then $\Gamma=\langle g\rangle$ is a free abelian group, and by a theorem of Herwig--Lascar (Theorem 3.3 of \cite{HL}), $\Gamma$ has the HL-property. Consider $F=\{g\}$ and $B=A\cup g(A)$. By Theorem~\ref{thm:Rosendal} there is a finite continuous $\mathcal{L}$-structure $\mathcal{N}$ and an action $\pi:\Gamma\actson \mathcal{N}$ such that $\mathcal{N}$ is an extension of $\mathcal{B}$ and for all $a\in A$, $g\cdot_{\pi} a=g(a)$. Let $\Lambda$ be the subgroup of $\Aut(\mathcal{N})$ generated by $g$. Since $\mathcal{N}$ is finite, $\Lambda$ is finite. By the universality and ultrahomogeneity of $\Gamma_\infty\actson
\mathcal{M}_\infty$, we may realize $\Lambda$ as a subgroup of $\Gamma_\infty$ and
$\mathcal{N}$ as a substructure of $\mathcal{M}_\infty$ extending $\mathcal{B}$. This gives a $\gamma\in \Gamma_\infty$ such that for every $a\in A$, $g(a)=\gamma\cdot a$.
\end{proof}

Putting these lemmas together, we have proved the following theorem.

\begin{theorem}\label{thm:HAut} Let $\mathcal{L}$ be a semiproper continuous signature and let $(\Delta, V)$ be a countable good value pair for $\mathcal{L}$. Then $\Aut(\mathbb{U}_{(\Delta, V)})$ contains $\mathbb{H}$ as a dense subgroup.
\end{theorem}

We have the following corollary.

\begin{corollary}\label{cor:AutUL} Let $\mathcal{L}$ be a proper continuous signature. Then $\Aut(\mathbb{U}_{\mathcal{L}})$ contains $\mathbb{H}$ as a dense subgroup.
\end{corollary}

\begin{proof} By Theorem~\ref{thm:HAut} we have that $\Aut(\mathbb{QU}_{\mathcal{L}})$ contains $\mathbb{H}$ as a dense subgroup. Since $\mathbb{U}_{\mathcal{L}}$ is isomorphic to the completion of $\mathbb{QU}_{\mathcal{L}}$ by Theorem~\ref{thm:QUL}, there is an isomorphic embedding $i: \mathbb{QU}_{\mathcal{L}}\to\mathbb{U}_{\mathcal{L}}$ with a dense range. This induces a group isomorphic embedding $\eta: \Aut(\mathbb{QU}_{\mathcal{L}})\to \Aut(\mathbb{U}_{\mathcal{L}})$. By Theorem~\ref{thm:AutQUdense}, $\eta(\Aut(\mathbb{QU}_{\mathcal{L}}))$ is dense in $\Aut(\mathbb{U}_{\mathcal{L}})$. It is easy to see that $\eta$ is continuous. Thus $\eta(\mathbb{H})$ is a dense subgroup of $\Aut(\mathbb{U}_{\mathcal{L}})$ isomorphic to $\mathbb{H}$.
\end{proof}


\end{document}